\documentclass[letterpaper,singlespaced,oneside,11pt,final]{ut-thesis}

\usepackage[sectionthm]{stephenpack}
\NewDocumentCommand \liez { og } { \IfValueTF {#1} {\lie{z}_{#1}(#2)} {\lie{z}(#2)} }

\NewDocumentCommand \wWhit { O{\ft} t~ s m }
  {\IfBooleanTF {#2}
    {\widetilde{\Whit}\IfBooleanTF{#3}{}{^{\raisebox{-0.3em}{$\scriptstyle #1$}}}_{#4}}
    {\Whit\IfBooleanTF{#3}{}{^{#1}}_{#4}}}

\NewDocumentCommand \wO { O{\ft} t~ s m }
  {\IfBooleanTF {#2}
    {\widetilde{\cO}\IfBooleanTF{#3}{}{^{#1}\mspace{-2mu}}(#4)}
    {\cO\IfBooleanTF{#3}{}{^{#1}\mspace{-2mu}}(#4)}}

\NewDocumentCommand \Weyl { G{V} o }
  { \IfNoValueTF {#2} {\mathbf{A}_{#1}\hspace{0pt}}
                      {\mathbf{A}_{#1}(#2)} }

\crefname{property}{property}{properties}
\crefname{condition}{condition}{conditions}

\NewDocumentCommand \typeA {sg} 
  {\IfNoValueTF{#2}{type~A}{type~$\mathrm{A}_{#2}$}}

\newcommand{\module}{\nobreakdash-module}
\newcommand{\modules}{\nobreakdash-modules}
\newcommand{\bimodule}{\nobreakdash-bimodule}

\newcommand{\action}{\nobreakdash-action}
\newcommand{\actions}{\nobreakdash-actions}

\newcommand{\equivariant}{\nobreakdash-equivariant}
\newcommand{\invariant}{\nobreakdash-invariant}
\newcommand{\invariants}{\nobreakdash-invariants}

\newcommand{\graded}{\nobreakdash-graded}
\newcommand{\grading}{\nobreakdash-grading}
\newcommand{\gradings}{\nobreakdash-gradings}

\newcommand{\Walgebra}{W\nobreakdash-algebra}
\newcommand{\Walgebras}{W\nobreakdash-algebras}

\newcommand{\sltriple}{$\fsl_2$\nobreakdash-triple}
\newcommand{\sltriples}{$\fsl_2$\nobreakdash-triples}

\theoremstyle{plain}
\newtheorem*{JM}{The Jacobson--Morozov theorem}
\NewDocumentCommand \JMthm {s} {Jacobson--Morozov theorem\IfBooleanTF{#1}{ (\cref{thm:JM})}{}}

\usepackage{indentfirst}
\usepackage{enumitem}
\usepackage{subcaption}
\usepackage{ifdraft}

\usepackage[boxsize=1.5em]{ytableau}
\usepackage{tikz}
\usepackage{tikz-cd}

\usetikzlibrary{positioning}

\degree{Doctor of Philosophy}
\department{Mathematics}
\gradyear{2014}
\author{Stephen Morgan}
\title{Quantum Hamiltonian reduction of W-algebras and category~
\texorpdfstring{$\cO$}{O}}

\hypersetup{pdfinfo={
    Title={Quantum Hamiltonian reduction of W-algebras and category O},
    Author={Stephen Morgan},
    Keywords={mathematics},
    Displaydoctitle=true,
    Fitwindow=false,
    Startview={FitH},
    },
    colorlinks = true,
}

\begin{document}

\begin{preliminary}

\maketitle
\clearpage

\begin{abstract}

\Walgebras\ are a class of non-commutative algebras related to the classical
universal enveloping algebras. They can be defined as a subquotient of $U(\fg)$
related to a choice of nilpotent element $e$ and compatible nilpotent subalgebra
$\fm$. The definition is a quantum analogue of the classical construction
of Hamiltonian reduction.

We define a quantum version of Hamiltonian reduction by stages and use it to
construct intermediate reductions between different \Walgebras\ $U(\fg,e)$ in
\typeA. This allows us to express the \Walgebra\ $U(\fg,e')$ as a subquotient
of $U(\fg,e)$ for adjacent nilpotent elements $e' \ge e$.
It also produces a collection of $(U(\fg,e),U(\fg,e'))$-bimodules analogous to
the generalised Gel'fand--Graev modules used in the classical definition of the
\Walgebra; these can be used to obtain adjoint functors between the
corresponding module categories.

The category of modules over a \Walgebra\ has a full subcategory defined in a
parallel fashion to that of the Bernstein-Gel'fand--Gel'fand (BGG)
category~$\cO$; this version of category~$\cO(e)$ for \Walgebras\ is equivalent
to an infinitesimal block of $\cO$ by an argument of Mili\v{c}i\'{c} and Soergel.
We therefore construct analogues of the translation functors between the
different blocks of $\cO$, in this case being functors between the categories~
$\cO(e)$ for different \Walgebras\ $U(\fg,e)$. This follows an argument of Losev,
and realises the category~$\cO(e')$ as equivalent to a full subcategory of the
category~$\cO(e)$ where $e' \ge e$ in the refinement ordering. Future work is to
use this to provide an alternate categorification of $U(\fsl_2)$ along the lines
of the work of Bernstein, Frenkel and Khovanov.

\end{abstract}

\begin{acknowledgements}
I would like to thank my adviser Joel Kamnitzer, for his enormous help in this
research. I~would also like to thank him for taking me in, and for recognising
that platypus can still lay eggs. In~addition, I would like to thank Chris
Dodd for his expertise and many helpful conversations. Finally, I would like to
thank my parents for their support and encouragement through the years, and my
partner Zsuzsi for her patience, support and love.
\end{acknowledgements}

\tableofcontents

\end{preliminary}

\chapter{Introduction}

In this chapter, we review the most important notions that will be
used in the rest of the thesis: \Walgebras\ and their relation to Slodowy
slices, quantum Hamiltonian reduction by stages, category~$\cO$ and
categorification.

\section{W-algebras}

\Walgebras, which we shall define in \cref{sec:WalgDef}, have been studied by a
number of mathematicians and physicists over the past 25 years
(cf.~\cite{dBT:QuantWAlg,Pre:PrimWAlg}).
The differing backgrounds and motivations of the researchers have produced
several equivalent definitions -- though not always obviously so -- rooted in
different fields and perspectives.
Physicists find them interesting due to the fact that they arise in the study of
conformal field theory, while representation theorists look to the insight they
can give us into the classical representation theory of Lie algebras.

For a semisimple Lie algebra $\fg$, the universal enveloping algebra $U(\fg)$ is
an associative algebra which completely controls the representation theory of
$\fg$. The category $\Mod{U(\fg)}$ is therefore of central concern to us. As in
the study of the representations of any algebra, a lot of information can be
determined by looking at how the centre $Z(\fg) \coloneqq Z(U(\fg))$ acts. A central
question to our work is what other algebras encapsulate important information
about the representation theory of $U(\fg)$? The \Walgebras\ $U(\fg,e)$ form
precisely one such class of algebras.

Given a semisimple Lie algebra $\fg$ with a chosen nilpotent element $e$, the
\Walgebra\ $U(\fg,e)$ is an associative algebra which lies between the algebras
$U(\fg)$ and $Z(\fg)$. More precisely, $U(\fg,e)$ is a subquotient of $U(\fg)$
determined by the nilpotent element $e$, where $U(\fg,0) = U(\fg)$ and
$U(\fg,e_{\text{reg}}) \simeq Z(\fg)$ for the regular nilpotent
$e_{\text{reg}} \in \fg$ (for example, the full Jordan block $J_n(0)$ in
$\fsl_n$). This second statement was known to Kostant, and the modern definition
of \Walgebras\ in many ways seeks to generalise his earlier work on this special
case \cite{Kos:WhitRep}. Between these two extremes lies one isomorphism class
of algebras for each nilpotent orbit in $\fg$. These intermediate \Walgebras\ 
control the representation theory of the blocks of $\Mod{U(\fg)}$ corresponding
to different central characters, where more singular nilpotent elements $e$ will
control blocks corresponding to less singular central characters.

Though this definition seems purely algebraic, it is actually a quantum version
of the classical geometric idea of Hamiltonian reduction of Poisson varieties.
Considering a nilpotent element $e$ in a semisimple Lie algebra $\fg$ associated
to an algebraic group $G$, one can consider the nilpotent orbit
$\orbit_e \coloneqq G \cdot e$ (note that the notation for the nilpotent orbit
$\orbit_e$ is similar to the notation for category~$\cO$ -- which is meant in a
given context will generally be clear). If $e$ is completed to an
\sltriple\ $\set{e,h,f}$, this gives rise to a natural transverse slice to
$\orbit_e$ known as the \emph{Slodowy slice}, $\slodowy_e = e + \ker \ad f$, for
example as in \cref{fig:SlodEg}. We usually consider the Slodowy slice
$\slodowy_e$ as a subset of $\fg^*$ using the isomorphism $\fg \simeq \fg^*$ induced
by the Killing form. In fact, the Slodowy slice can be expressed as the
Hamiltonian reduction of $\fg^*$ with respect to the action of a certain
unipotent algebraic group depending on the \sltriple. The subquotient
definition for \Walgebras\ can be understood in an analogous way.

\begin{figure}[h]
\begin{align*}
e & = \begin{psmallmatrix} 0&1&0 \\ 0&0&0 \\ 0&0&0 \end{psmallmatrix} &
h & = \begin{psmallmatrix} 1&0&0 \\ 0&-1&0 \\ 0&0&0 \end{psmallmatrix} &
f & = \begin{psmallmatrix} 0&0&0 \\ 1&0&0 \\ 0&0&0 \end{psmallmatrix} &
\slodowy_e & = \set*{ \begin{pmatrix} a&0&0 \\ b&a&d \\ c&0&-2a \end{pmatrix}
                       \st a,b,c,d \in \CC }
\end{align*}
\caption{The Slodowy slice of an $\fsl_2$-triple in $\fsl_3$.}
\label{fig:SlodEg}
\end{figure}

The universal enveloping algebra $U(\fg)$ is a filtered algebra whose associate
graded algebra is the ring of functions $\CC[\fg^*]$. Given that $\fg$ is a Lie
algebra, the ring of functions $\CC[\fg^*]$ acquires a Poisson bracket coming
from the Lie bracket on $\fg$, known as either the Lie--Poisson or the
Kostant--Kirillov bracket. The multiplication in $U(\fg)$ encodes both the
multiplication and the Poisson bracket in $\CC[\fg^*]$; this is known as a
deformation quantisation of the Poisson algebra $\CC[\fg^*]$. The \Walgebra\ 
$U(\fg,e)$ has a similar geometric interpretation: its associate graded algebra
is the ring of functions on the Slodowy slice $\slodowy_e \subseteq \fg^*$. Further,
$U(\fg,e)$ is the unique deformation quantisation (up to isomorphism) of the
ring of functions on the Slodowy slice $\slodowy_e$. The structure of the
\Walgebra\ $U(\fg,e)$ is therefore geometrically encoded in the Poisson variety
$\slodowy_e$ \cite{GG:QuantSlod}.

This geometric interpretation allows us to better understand the definition of
the \Walgebra: upon passing to the associate graded algebra, expressing
$U(\fg,e)$ as a subquotient of $U(\fg)$ corresponds to choosing expressing
$\slodowy_e$ as a Hamiltonian reduction of $\fg^*$. Correspondingly, the
subquotient definition of the \Walgebra\ can be understood as expressing
$U(\fg,e)$ as a \emph{quantum Hamiltonian reduction} (\emph{QHR}) of $U(\fg)$.
In this interpretation, the nilpotent element $e$ plays the role of a regular
value of the moment map, and a certain nilpotent Lie algebra $\fm$ (the Premet
subalgebra) associated to $e$ plays the role of the unipotent subgroup acting on
$\fg^*$.

\section{Quantum Hamiltonian reduction by stages}

It can be taken as the definition of \Walgebras\ that they can be expressed as
certain quantum Hamiltonian reductions of $U(\fg) = U(\fg,0)$, but one might ask
whether a \Walgebra\ $U(\fg,e')$ can be expressed as a QHR of another \Walgebra\ 
$U(\fg,e)$ in a way compatible with the original reduction -- that is, do there
exist pairs of nilpotent elements $e$ and $e'$ for which we can express
$U(\fg,e')$ as a QHR of $U(\fg,e)$ in such a way that \cref{fig:IntRedDiag}
commutes up to isomorphism? Upon de-quantising the diagram, this would
correspond to expressing $\slodowy_{e'}$ as a Hamiltonian reduction of $\fg^*$
by stages, passing through the Slodowy slice $\slodowy_e$ in the middle step.
The theory of Hamiltonian reduction by stages is a well-developed branch of
symplectic geometry, and has been outlined in -- for example -- \cite{MMO:HamRed}.
To construct our intermediate reductions, we therefore need to develop a quantum
version of Hamiltonian reduction by stages.

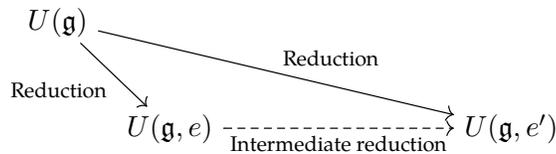
\begin{figure}[th!]
\begin{equation*}
\begin{tikzcd}[column sep=tiny]
U(\fg)
  \ar{dr}[swap]{\text{Reduction}}
  \ar{drrr}{\text{Reduction}} \\
 & U(\fg,e) \ar[dashed]{rr}[swap]{\text{Intermediate reduction}}
 &\hspace{6em} & U(\fg,e')
\end{tikzcd}
\end{equation*}
\caption{Reduction of W-algebras by stages.}
\label{fig:IntRedDiag}
\end{figure}

The first main result of this thesis is \cref{thm:qhr}, which states that we can
construct such an intermediate set of reductions in Lie algebras of \typeA\ for
all pairs of nilpotent elements such that $e'$ covers $e$ in the dominance
ordering, that is whenever $\orbit_e \subseteq \smash{\overline{\orbit}_{e'}}$ and for which
there are no intermediate orbits.
The result is to construct a sequence of reductions for all \Walgebras\
corresponding to nilpotent elements of the the Lie algebra $\fg$, such that
$U(\fg,e)$ can be reduced to $U(\fg,e')$ precisely if $e'$ covers $e$.
In such a way, a sequence of commuting reductions is obtained which is precisely
the reverse of the Hasse diagram for nilpotents orbits under the dominance
ordering (see \cref{fig:sl6Hasse}, for example).

\begin{conj} \label{thm:qhr}
Let $\fg$ be a semisimple Lie algebra of \typeA, and $e$ and $e'$ be two
nilpotent elements such that $e'$ covers $e$ in the dominance ordering.
Then there exists a Lie subalgebra $\fk \subseteq \fg$ and a left ideal
$U(\fg,e)\fk_\kappa \subseteq U(\fg,e)$ such that
\begin{equation*}
U(\fg,e') \simeq \pp*{U(\fg,e) \big/ U(\fg,e)\fk_\kappa }^\fk,
\end{equation*}
where invariants are taken with respect to the adjoint action of $\fk$.
\end{conj}


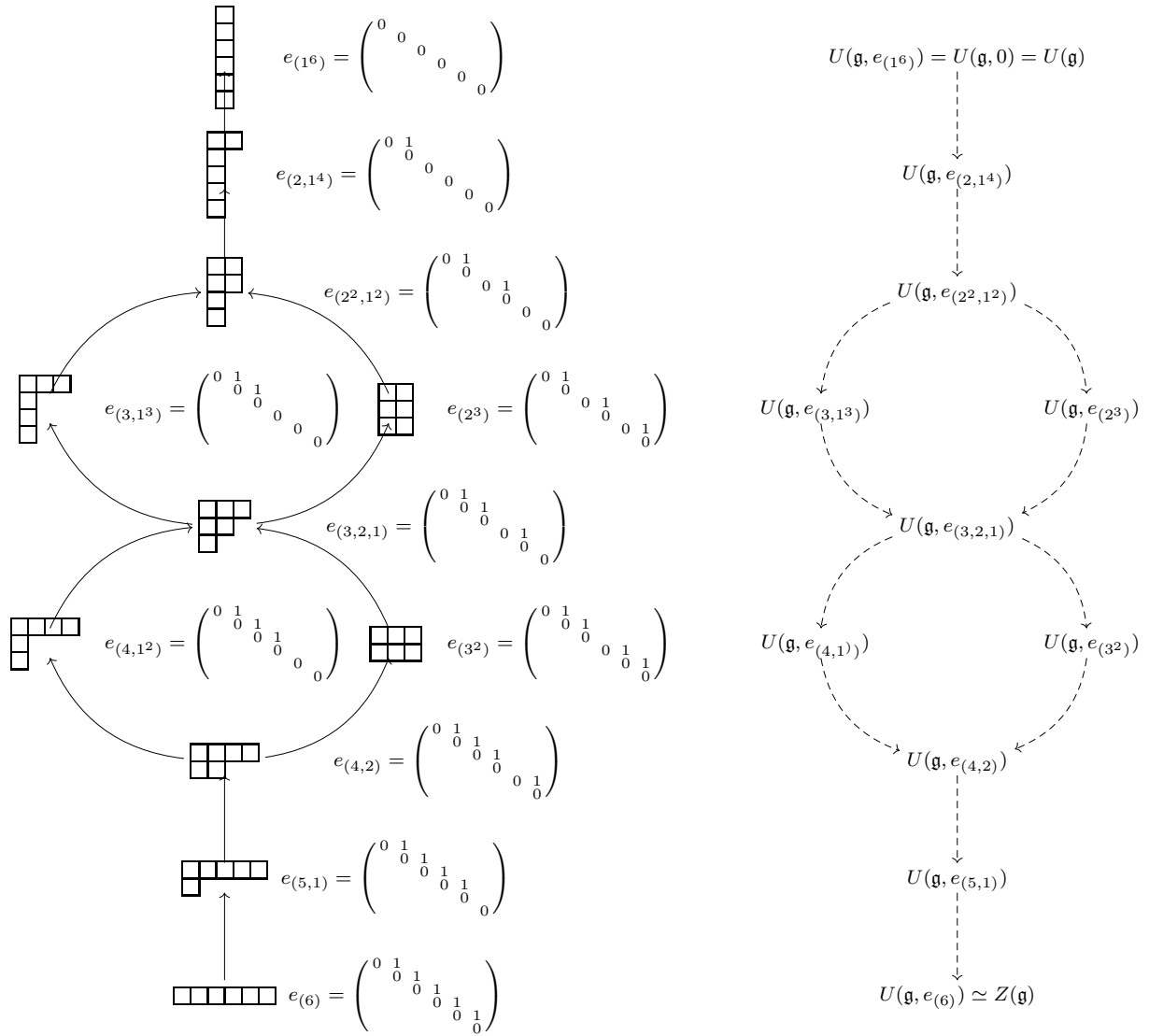
\begin{figure}[f]
\scriptsize
\ytableausetup{smalltableaux,centertableaux}
\begin{equation*}
\begin{tikzcd}[column sep=tiny, row sep=huge]
& \ydiagram{1,1,1,1,1,1} &
  \mathclap{
    e_{\seq{1^6}} =
      \begin{psmallmatrix} 0& & & & &  \\
                            &0& & & &  \\
                            & &0& & &  \\
                            & & &0& &  \\
                            & & & &0&  \\
                            & & & & &0
      \end{psmallmatrix}
  }&&\hspace{5mm}& 
& \mathclap{U(\fg,e_{\seq{1^6}}) = U(\fg,0) = U(\fg)}
    \ar[dashed]{d}\\
\
& \ydiagram{2,1,1,1,1} \ar{u} &
  \mathclap{
    e_{\seq{2,1^4}} =
      \begin{psmallmatrix} 0&1& & & &  \\
                            &0& & & &  \\
                            & &0& & &  \\
                            & & &0& &  \\
                            & & & &0&  \\
                            & & & & &0
      \end{psmallmatrix}
  }&&&
& U(\fg,e_{\seq{2,1^4}})
    \ar[dashed]{d}\\
\
& \ydiagram{2,2,1,1} \ar{u} &
  \mathclap{\hspace{1.5cm}
    e_{\seq{2^2,1^2}} =
      \begin{psmallmatrix} 0&1& & & &  \\
                            &0& & & &  \\
                            & &0&1& &  \\
                            & & &0& &  \\
                            & & & &0&  \\
                            & & & & &0
      \end{psmallmatrix}
  }&&&
& U(\fg,e_{\seq{2^2,1^2}})
    \ar[dashed, bend right]{dl}
    \ar[dashed, bend left]{dr}\\
\
\ydiagram{3,1,1,1} \ar[bend left]{ur} &
  e_{\seq{3,1^3}} =
    \begin{psmallmatrix} 0&1& & & &  \\
                          &0&1& & &  \\
                          & &0& & &  \\
                          & & &0& &  \\
                          & & & &0&  \\
                          & & & & &0
    \end{psmallmatrix} &
\ydiagram{2,2,2} \ar[bend right]{ul} &
  e_{\seq{2^3}} =
    \begin{psmallmatrix} 0&1& & & &  \\
                          &0& & & &  \\
                          & &0&1& &  \\
                          & & &0& &  \\
                          & & & &0&1 \\
                          & & & & &0
    \end{psmallmatrix} &&
U(\fg,e_{\seq{3,1^3}})
    \ar[dashed, bend right]{dr} &&
U(\fg,e_{\seq{2^3}})
    \ar[dashed, bend left]{dl} \\
\
& \ydiagram{3,2,1} \ar[bend left]{ul} \ar[bend right]{ur} &
  \mathclap{\hspace{1.5cm}
    e_{\seq{3,2,1}} =
      \begin{psmallmatrix} 0&1& & & &  \\
                            &0&1& & &  \\
                            & &0& & &  \\
                            & & &0&1&  \\
                            & & & &0&  \\
                            & & & & &0
      \end{psmallmatrix}
  }&&&
& U(\fg,e_{\seq{3,2,1}})
    \ar[dashed, bend right]{dl}
    \ar[dashed, bend left]{dr}\\
\
\ydiagram{4,1,1} \ar[bend left]{ur} &
  e_{\seq{4,1^2}} =
    \begin{psmallmatrix} 0&1& & & &  \\
                          &0&1& & &  \\
                          & &0&1& &  \\
                          & & &0& &  \\
                          & & & &0&  \\
                          & & & & &0
    \end{psmallmatrix} &
\ydiagram{3,3} \ar[bend right]{ul} &
  e_{\seq{3^2}} =
    \begin{psmallmatrix} 0&1& & & &  \\
                          &0&1& & &  \\
                          & &0& & &  \\
                          & & &0&1&  \\
                          & & & &0&1 \\
                          & & & & &0
    \end{psmallmatrix} &&
U(\fg,e_{\seq{4,1^)}})
    \ar[dashed, bend right]{dr} &&
U(\fg,e_{\seq{3^2}})
    \ar[dashed, bend left]{dl} \\
\
& \ydiagram{4,2} \ar[bend left]{ul} \ar[bend right]{ur} &
  \mathclap{\hspace{1.5cm}
    e_{\seq{4,2}} =
      \begin{psmallmatrix} 0&1& & & &  \\
                            &0&1& & &  \\
                            & &0&1& &  \\
                            & & &0& &  \\
                            & & & &0&1 \\
                            & & & & &0
      \end{psmallmatrix}
  }&&&
& U(\fg,e_{\seq{4,2}})
    \ar[dashed]{d}\\
\
& \ydiagram{5,1} \ar{u} &
  \mathclap{
    e_{\seq{5,1}} =
      \begin{psmallmatrix} 0&1& & & &  \\
                            &0&1& & &  \\
                            & &0&1& &  \\
                            & & &0&1&  \\
                            & & & &0&  \\
                            & & & & &0
      \end{psmallmatrix}
  }&&&
& U(\fg,e_{\seq{5,1}})
    \ar[dashed]{d}\\
\
& \ydiagram{6} \ar{u} &
  \mathclap{
    e_{\seq{6}} =
      \begin{psmallmatrix} 0&1& & & &  \\
                            &0&1& & &  \\
                            & &0&1& &  \\
                            & & &0&1&  \\
                            & & & &0&1 \\
                            & & & & &0
      \end{psmallmatrix}
  }&&&
& \mathclap{U(\fg,e_{\seq{6}}) \simeq Z(\fg)}
\end{tikzcd}
\end{equation*}
\ytableausetup{nosmalltableaux,nocentertableaux}

\caption{
The Hasse diagram and intermediate quantum Hamiltonian reductions for $\fg =
\fsl_6$.
Each partition $P$ of $6$ has a corresponding conjugacy class of nilpotent
matrices – here their representative $e_P$ in Jordan canonical form is shown to
the right of the corresponding partition – and each has an associated W-algebra
$U(\fg,e_P)$. The diagram showing the intermediate quantum Hamiltonian
reductions between W-algebras is the reverse of the Hasse diagram.}

\label{fig:sl6Hasse}
\end{figure}

To construct the reductions, we use the theory of \emph{pyramids} in semisimple
Lie algebras developed by Elashvili and Kac \cite{EK:ClassGG}. Pyramids are
combinatorial objects closely related to Young tableaux, however they encode not
only the nilpotent element $e$, but also a \emph{good grading} of $\fg$ for~$e$.
This allows much of the same information provided by an \sltriple\ to be
specified by a weaker piece of data, giving us extra flexibility needed for
this construction. This provides a construction which produces a quantum
Hamiltonian reduction of the \Walgebra\ $U(\fg,e)$, which conjecturally is
isomorphic to the \Walgebra\ $U(\fg,e')$.

\section{Category~\texorpdfstring{$\cO$}{O}}

For a semisimple Lie algebra $\fg$, there is a well-studied category of its
representations known as category~$\cO$. It consists of all representations
satisfying a certain finiteness condition, and in particular contains all
finite-dimensional modules. Category~$\cO$ has a number of remarkable
properties, and decomposes naturally into blocks in a way that allows it to be
equipped with a number of module structures that parallel important classical
constructions (cf.~\cite{BFK:CatTLAlg,KMS:CatSpecht,KMS:AbCat}): this kind of
enrichment is known as a \emph{categorification}. The representations of
\Walgebras\ are naturally related to the blocks of category~$\cO$, and in my work
I attempt to use this relation to construct categorifications of classical
objects.

The representation theory of $U(\fg,e)$ has been studied by a number of authors,
and in particular its relationship to the representation theory of $U(\fg)$ has
been examined in \cite{BGK:HWTWAlg, Web:CatOWAlg, Los:FinRepWAlg}. For
example, Skryabin has shown that the category $\Mod{U(\fg,e)}$ is equivalent to
a full subcategory of $\Mod{U(\fg)}$, characterised by a certain finiteness
property with respect to the action of a subalgebra $U(\fm)$ determined by $e$
\cite{Pre:TransSlice}.

Given a semisimple Lie algebra $\fg$, a choice of triangular decomposition
$\fg \simeq \fn_-\oplus \fh\oplus \fn_+$ allows us to define the highest weight modules as those
which are a finite union of sets of the form $U(\fn_-) \cdot v$ for $v$ a
\emph{highest weight vector} (i.e.\ $\fn_+ \cdot v = 0$).
Many important $U(\fg)$\modules\ are highest weight modules, including all
finite-dimensional representations and Verma modules.
The \emph{BGG category~$\cO$} \cite{BGG:CatO} is the minimal full subcategory of
$\Mod{U(\fg)}$ containing all highest-weight modules that is closed under direct
sums, sub- and quotient modules, kernels and cokernels, and tensoring with
finite-dimensional modules (the so-called \emph{translation functors}).
This remains rich enough to contain a lot of information about the
representation theory of $U(\fg)$, while having additional properties making it
more amenable to study.
In particular it is a \emph{highest weight category}, which allows us to choose
a number of convenient bases for its Grothendieck group (that is, the group
hose elements are formal differences of isomorphism classes of representations
and whose group operation is the direct sum of representations).

Category~$\cO$ has a natural \emph{block decomposition} $\cO \simeq \bigoplus \cO_\chi$
indexed by the generalised central characters of $U(\fg)$. There are no
homomorphisms or non-trivial extensions between modules belonging to different
blocks, so to understand the structure of category~$\cO$ it suffices to
understand the blocks $\cO_\chi$. In \cite{MS:CompWhit}, Mili\v{c}i\'{c} and Soergel apply
the theory of Harish-Chandra bimodules to construct an equivalence of categories
exchanging the condition on the (generalised) central character with a related
condition on the (generalised) nilpotent character. We choose a nilpotent
element $e$ compatible with $\chi$ in the sense that the stabiliser subgroup of $\chi$
under the `dotted Weyl action' is generated by the simple reflections
corresponding to~$e$. When combined with the Skryabin equivalence, this can be used
to construct an equivalence $\cO_\chi \simeq \cO_0(e)$, where $\cO_0(e)$ is a full
subcategory of $\Mod{U(\fg,e)}$ analogous to a regular block of category~$\cO$.
\cite{Los:CatOWAlg, Web:CatOWAlg}

The construction of quantum Hamiltonian reduction in \cref{thm:qhr} produces a
pair of adjoint functors $\Mod{U(\fg,e)} \leftrightarrows \Mod{U(\fg,e')}$,
however these functors do not preserve the finiteness conditions of the
categories $\cO(e)$ and $\cO(e')$. To solve this problem we adapt one of the
techniques of Losev in \cite{Los:CatOWAlg}, in which he proves that $\cO(e)$ is
equivalent to a full subcategory of $\Mod{U(\fg)}$ called the Whittaker
category. This allows us to realise the category $\cO(e')$ as a full subcategory
of $\Mod{U(\fg,e)}$. Future work is to use this embedding along with
\emph{averaging functors} to construct pairs of functors
$\cO(e) \leftrightarrows \cO(e')$, analogous to the classical translation
functors between different infinitesimal blocks of category~$\cO$. We hope in
the future to show that these functors intertwine with the translation
functors through the Mili\v{c}i\'{c}--Soergel equivalence.

\chapter{W-algebras}

In this chapter, we will work towards defining the basic objects of our study:
the \Walgebras\ $U(\fg,e)$. There are a number of equivalent definitions of
\Walgebras\ given in a number of different sources
(cf.~\cite[§3]{Wan:NilOrbWAlg}). We shall generally use a definition of
\Walgebras\ expressed as a certain subquotient of the universal enveloping
algebra $U(\fg)$ known as the Whittaker module definition, though we'll
occasionally remark upon and use the equivalence with other formulations.

\section{Nilpotent orbits and Slodowy slices}
\label{sec:Nilorbits}

We begin by recalling some basic facts about the nilpotent cone
$\nilcone \subseteq \fg$. The algebraic group $G$ acts on $\fg$ by the adjoint action,
and this action preserves the nilpotent cone. As a result, we have a
stratification of $\nilcone$ into nilpotent orbits $\orbit_e$, where
$\orbit_e \coloneqq G \cdot e$ is the orbit of the nilpotent element $e$ under the
adjoint action of $G$. This holds in an arbitrary semisimple Lie algebra, but it
has a particularly simple form in \typeA, where $\fg = \fsl_n$: $G = SL_n$ acts
on $\fg$ by conjugation and the nilpotent cone $\nilcone$ consists of all nilpotent
matrices in $\fg$, which are classified up to conjugacy by their Jordan
canonical form.

The nilpotent cone $\nilcone$ has a unique dense open orbit
$\orbit_{\text{reg}}$ called the \emph{regular orbit}, which consists of all
\emph{regular} nilpotent elements: that is elements $e \in \fg$ for which
$\dim Z_G(e) = \rank \fg$, where $Z_G(e)$ is the stabiliser of $e$ in $G$.
The complement of the regular orbit $\nilcone \smallsetminus
\orbit_{\text{reg}}$, itself has a unique open dense orbit $\orbit_{\text{sub}}$
called the \emph{subregular orbit}, and a unique \emph{minimal orbit}
$\orbit_{\text{min}}$ of smallest strictly-positive dimension.
In \typeA{n-1} these special orbits have explicit descriptions in
terms of the Jordan canonical form of the nilpotent elements comprising them:
\begin{itemize}
\item $\orbit_{\text{reg}}$ has elements with a single Jordan block of size $n$.
\item $\orbit_{\text{sub}}$ has elements with one block of size $n-1$ and
      another of size $1$.
\item $\orbit_{\text{min}}$ has elements with one block of size $2$ and $n-2$
      blocks of size $1$.
\end{itemize}

The set of nilpotent orbits in $\nilcone$ is naturally a partially ordered set,
where $\orbit' \le \orbit$ if and only if $\orbit' \subseteq \smash{\overline{\orbit}}$.
Under this ordering, we can make some statements about the three nilpotent
orbits described above. The orbit $\orbit_{\text{reg}}$ is the maximal element
of the poset, $\orbit_{\text{sub}}$ is the maximal element of the poset lying
under $\orbit_{\text{reg}}$, and $\orbit_{\text{min}}$ is the minimal element of
the poset lying above the zero orbit $\set{0}$.

\subsection{The Jacobson--Morozov theorem}

In order to discuss nilpotent orbits, it will frequently be convenient to
complete a given nilpotent element $e \in \fg$ to an \sltriple\ 
$\set{e,h,f} \subseteq \fg$, i.e.\ $e$, $h$ and~$f$ satisfy the $\fsl_2$ commutation
relations $[h,e] = 2e$, $[h,f] = -2f$, and $[e,f] = h$.
Another way of phrasing this condition is that there exists a Lie algebra
homomorphism $\fsl_2 \to \fg$ such that the image of $e$ in the standard basis
of $\fsl_2$ is the element $e \in \fg$. The \emph{Jacobson--Morozov theorem} states
that it is always possible to extend a nilpotent element $e$ to an
\sltriple\ in a semisimple Lie algebra. There are a number of different
proofs of this theorem in the literature (cf.~\cite[§3.7.25]{CG:RepTheoryGeom});
the version we present here comes from \cite[§3.3]{CM:Nilorbits}.

\begin{JM} \label{thm:JM}
For any non-zero nilpotent element $e$ in a semisimple Lie algebra~$\fg$, there
exists an \sltriple\ $\set{e,h,f} \subseteq \fg $ such that $[h,e] = 2e$, $[h,f] = -2f$,
and $[e,f] = h$.
\end{JM}

\begin{proof}
We prove this by induction on $\dim \fg$. The non-zero semisimple Lie algebra
with smallest possible dimension is $\fsl_2$ itself, and since any non-zero
nilpotent in $\fsl_2$ is conjugate to any other, we can simply conjugate the
standard \sltriple\ to coincide with $e$. We now proceed with the inductive
step. If $e$ lies in a proper semisimple subalgebra of $\fg$, then we can apply the
inductive hypothesis, so we now assume that $e$ lies in no proper semisimple
subalgebra of $\fg$.

We first prove that $\killing{e}{\liez{e}} = 0$. Consider $x \in \liez{e}$, and
recall that $\killing{e}{x} = \tr (\ad e \ad x)$. The Jacobi identity shows that
$\ad e$ commutes with $\ad x$ for any $x \in \liez{e}$, so we can state that
$(\ad e \ad x)^n = (\ad e)^n (\ad x)^n$ for any $n \in \NN$. The element $e$ is
nilpotent, and so $(\ad e)^n = 0$ for large enough $n$; as a result, the
operator $\ad e \ad x$ is also nilpotent and hence traceless, proving our claim.

We thus know that $e \in (\liez{e})^\bot$, where the orthogonal complement is
taken with respect to the Killing form; we now show that $(\liez{e})^\bot =
[\fg,e]$. First note that the associativity of the Killing form shows that
$[\fg,e] \subseteq (\liez{e})^\bot$; to show that this is everything we count
dimensions. Note that the map $\ad e \colon \fg \to \fg$ has kernel $\liez{e}$
and image $[\fg,e]$, and so the rank--nullity theorem tells us that
$\dim [\fg, e] = \dim \fg - \dim \liez{e}$. This is precisely the dimension of
$(\liez{e})^\bot$. As a result, we know that $[h,e] = 2e$ for some element
$h \in \fg$. We can further take $h$ to be semisimple, as if it were not its
semisimple part would also satisfy this property.

\begin{lem} \label{lem:JMhlem}
The element $h$ constructed above lies in $[\fg,e]$.
\end{lem}

Assuming this \namecref{lem:JMhlem} for the moment, we give a proof of the \JMthm.
Let $f' \in \fg$ be an element such that $[e,f'] = h$; since $h$ is semisimple we
have that $\fg = \bigoplus_{j} \fg_{\lambda_j}$ where
$\fg_{\lambda_j} = \set{x \in \fg \st [h,x] = \lambda_j x}$ is the eigenspace corresponding to
eigenvalue $\lambda_j$. We therefore have a decomposition $f' = \sum_j f_j$, where $f_j$
is the projection of $f'$ onto $\fg_{\lambda_j}$: this gives us that
$h = \sum_j [e,f_j]$. However, we can note that $h \in \fg_0$, and that
$[e,\fg_{\lambda_j}] \subseteq \fg_{\lambda_j+2}$; this allows us to see that there exists a $k$
with $\lambda_k = -2$, and that $h = [e,f_k]$. Taking $f = f_k$ gives us our
\sltriple\ $\set{e,h,f}$, and completes the proof of the \JMthm.

It remains to prove \cref{lem:JMhlem}. We shall prove this by contradiction: we
shall assume that $h \notin [\fg,e]$, and construct a proper semisimple
subalgebra of $\fg$ containing $e$, violating the assumption we made at the
beginning of the proof preventing us from using the inductive hypothesis.
Since $[\fg,e] = (\liez{e})^\bot$, we know that $\killing{h}{\liez{e}} \neq 0$.

Since $\ad h$ preserves $\liez{e}$, we can decompose it into $\ad h$ eigenspaces
$\bigoplus_j \liez{e}_{\mu_j}$, where $\liez{e}_0$ is the centraliser of $h$ in
$\liez{e}$: this gives us the decomposition
$\liez{e} = \liez[\liez{e}]{h} \oplus \bigoplus_{\mu_j \neq 0} \liez{e}_{\mu_j}$.
Associativity of the Killing form tells us that $\killing{h}{[h,\liez{e}]} = 0$,
while the eigenvalue decomposition tells us that for $x \in \liez{e}_{\mu_j}$ we
have that $\killing{h}{[h,x]} = \mu_j \killing{h}{x}$. We therefore know that
$h \in (\liez{e}_{\mu_j})^\bot$ for any $\mu_j \neq 0$. So to satisfy the condition
that $\killing{h}{\liez{e}} \neq 0$, there must exist an element
$z \in \liez{e}_0 = \liez[\liez{e}]{h}$ with $\killing{h}{z} \neq 0$. We can
assume that $z$ is semisimple by the same argument used above to prove that
there exists a semisimple element $h$ with $[h,e] = 2e$.

Since the centraliser of any semisimple element is reductive
(cf.~\cite[Lemma~2.1.2]{CM:Nilorbits}), we have that $[\liez{z},\liez{z}]$ is a
semisimple subalgebra of $\fg$. It is proper, since no non-zero semisimple
element commutes with all of $\fg$ (by the definition of semisimplicity). By
construction, $z$ commutes with both $h$ and $e$, and hence $2e = [h,e] \in
[\liez{z},\liez{z}]$. Thus $[\liez{z},\liez{z}]$ is a proper semisimple
subalgebra of $\fg$ containing $e$, contradicting our hypothesis. This completes
the proof of \cref{lem:JMhlem}, and of the \JMthm.
\end{proof}

\begin{eg}
\label{eg:TypeAsltriple}
The proof of the \JMthm\ is unfortunately non-constructive,
however in \typeA\ we can provide an explicit construction of an
\sltriple\ for a given nilpotent $e$. We can always conjugate $e$ into
Jordan canonical form, and it suffices to give an \sltriple\ for a Jordan
block: the \sltriple\ for the full Jordan canonical form can be constructed
from the given blocks.

For a nilpotent consisting of a single Jordan block of size~$n$, an \sltriple\
consists of the elements $h = \operatorname{diag} (n-1, n-3, \dotsc, 1-n)$ and
\begin{equation*}
f =
\begin{pmatrix}
  0   &        &        & 0 \\
  a_1 & 0      &        &   \\
      & \ddots & \ddots &   \\
  0   &        & a_{n-1}& 0
\end{pmatrix},
\end{equation*}
where $a_i = i(n-i)$.
\end{eg}

\subsection{Good gradings}

An \sltriple\ $\set{e,h,f}$ contains a semisimple element $h$, and
as a result $\fg$ will decompose into a direct sum of eigenspaces of the
operator $\ad h$. Specifically we can write $\fg = \bigoplus_{j \in \ZZ} \fg_j$,
where $\fg_j = \set{x \in \fg \st [h,x] = jx}$. The Jacobi identity implies that
$[\fg_i,\fg_j] \subseteq \fg_{i+j}$, and so any \sltriple\ endows $\fg$ with a
natural $\ZZ$\grading. In fact, every $\ZZ$\grading\ comes from the action of a
semisimple element in such a way.

\begin{lem} \label{lem:ZgradSS}
Given a $\ZZ$\grading\ $\Gamma \colon \fg = \bigoplus_{j \in \ZZ} \fg_j$, there exists a
semisimple element $h_\Gamma \in \fg$ such that
$\fg_j = \set{x \in \fg \st [h_\Gamma, x] = jx}$.
\end{lem}

\begin{proof}
Note that the degree map $\partial \colon \fg \to \ZZ$ given by $\partial(x) = jx$ for
$x \in \fg_j$ is a derivation of the semisimple Lie algebra $\fg$.
Since all derivations of a semisimple Lie algebra are inner derivations, there
exists a semisimple element $h_\Gamma \in \fg$ such that $\partial = \ad h_\Gamma$.
\end{proof}

\begin{rem}
This proof can be extended to reductive Lie algebras such as $\fgl_n$.
\end{rem}

We will be interested in the properties of $\ZZ$\gradings\
coming from \sltriples, and so we will give them a special name:
$\ZZ$\gradings\ coming from \sltriples\ by the above procedure are called
\emph{Dynkin gradings}. Dynkin gradings have a number of useful properties:

\begin{enumerate}[label=\textsf{GG\arabic*}.,ref=\textsf{GG\arabic*},labelindent=\parindent,leftmargin=*]
\item \label[property]{GGprop1}
  $e \in \fg_2$,
\item \label[property]{GGprop2}
  $\ad e \colon \fg_j \to \fg_{j+2}$ is injective for $j \le -1$,
\item \label[property]{GGprop3}
  $\ad e \colon \fg_j \to \fg_{j+2}$ is surjective for $j \ge -1$,
\item \label[property]{GGprop4}
  $\liez{e} \subseteq \bigoplus_{j \ge 0} \fg_j$,
\item \label[property]{GGprop5}
  $\killing{\fg_i}{\fg_j} = 0$ unless $i+j=0$,
\item \label[property]{GGprop6}
  $\dim \liez{e} = \dim \fg_0 + \dim \fg_1$.
\end{enumerate}

\Cref*{GGprop1} follows directly from the definition of an
\sltriple, while \cref*{GGprop2,GGprop3} follow from
the fact that $\fg$ has the structure of a finite-dimensional
$\fsl_2$-representation, and hence decomposes as a direct sum of irreducible
$\fsl_2$\modules\ (see \cref{fig:sl2rep}). \Cref*{GGprop4,GGprop5,,GGprop6}, on the other hand, can be
proven from \cref*{GGprop1,GGprop2,,GGprop3} directly.


\begin{figure}[hf]

\tikzstyle{dot}=[circle,fill=black,inner sep=0,minimum size=3mm]

\centering
\begin{tikzpicture}[scale=.5,auto]

  \draw[<->] (-.5,7) -- (-.5,-7);
  \foreach \x in {6,4,...,-6} {
    \draw (-.8,\x) -- (-.5,\x);
    \node (tick \x)  at (-.5,\x) [label=left:$\x$] {};
  }

  \foreach \x in {6,4,...,-6}
    \node at (1,\x) [dot] {};
  \draw (1,6) -- (1,-6);

  \foreach \x in {5,3,...,-5}
    \node at (3,\x) [dot] {};
  \draw (3,5) -- (3,-5);

  \foreach \x in {5,3,...,-5}
    \node at (5,\x) [dot] {};
  \draw (5,5) -- (5,-5);

  \foreach \x in {3,1,-1,-3}
    \node at (7,\x) [dot] {};
  \draw (7,3) -- (7,-3);

  \foreach \x in {2,0,-2}
    \node at (9,\x) [dot] {};
  \draw (9,2) -- (9,-2);
\end{tikzpicture}

\caption{
The decomposition of a finite-dimensional $\fsl_2$-module as a sum of
simple modules. Each column is an irreducible component, and each dot represents
the 1-dimensional weight space of weight given at the left. The action of $h$
corresponds to scaling by the weight, the action of $e$ moves up the string to
the next weight space and the action of $f$ moves down.}

\label{fig:sl2rep}
\end{figure}
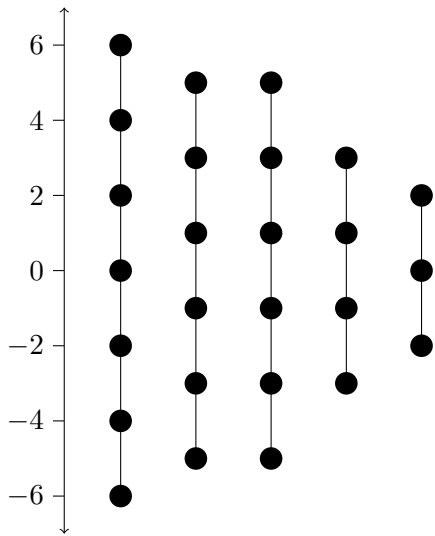

\begin{prop} \label{prop:GGprops}
Any $\ZZ$\grading\ satisfying \cref*{GGprop1,GGprop2,,GGprop3} will also satisfy
\cref*{GGprop4,GGprop5,,GGprop6}.
\end{prop}

\begin{proof}
Note that \cref*{GGprop2} implies that $\liez{e} = \ker \ad e$ must not have any
component in $\fg_j$ for $j \le -1$, whence \cref*{GGprop4}. To prove
\cref*{GGprop5}, we consider the semisimple element $h_\Gamma$ coming from
\cref{lem:ZgradSS}: taking $x \in \fg_i$ and $y \in \fg_j$ we can note that
$\killing{[x,h]}{y} = \killing{x}{[h,y]}$, and hence
$-i \killing{x}{y} = j \killing{x}{y}$. As a result, either $i+j=0$ or
$\killing{x}{y} = 0$, proving \cref*{GGprop5}. \Cref*{GGprop6} follows from the
following sequence, which is short exact by \cref*{GGprop3,GGprop4}:
\begin{equation*}
0 \lra \liez{e} \xrightarrow{\hspace{2.5ex}} \fg_{-1} \oplus \fg_0 \oplus \fg_{>0}
  \xrightarrow{\ad e} \fg_{>0} \lra 0,
\end{equation*}
where $\fg_{>0} = \bigoplus_{j \ge 0} \fg_j$.
Dimension counting then gives that $\dim \liez{e} = \dim \fg_{-1} + \dim \fg_0$,
which gives \cref*{GGprop6} when combined with the fact that
$\ad e \colon \fg_{-1} \to \fg_1$ is a bijection from \cref*{GGprop2,GGprop3}.
\end{proof}

\begin{rem}
\label{rem:GGprop23}
We can actually make a stronger statement. Note that \cref*{GGprop5} holds for
any $\ZZ$\grading\ $\Gamma$; we can therefore use \cref*{GGprop5} and the
non-degeneracy of the Killing form to prove that \cref*{GGprop2,GGprop3} are
equivalent for any $\ZZ$\grading\ $\Gamma$.
\end{rem}

It turns out that we need slightly more flexibility than the Dynkin gradings
are able to provide us, but the above \namecref{prop:GGprops} tells us that many
of the desirable properties of Dynkin gradings can be obtained from arbitrary
$\ZZ$\gradings\ which satisfy \crefrange*{GGprop1}{GGprop3}. This motivates the
following definition.

\begin{defn}
Given a semisimple Lie algebra $\fg$ with chosen nilpotent element $e$, a
$\ZZ$\grading\ $\Gamma \colon \fg = \bigoplus_{j \in \ZZ} \fg_j$ is called a \emph{good
grading} for $e$ if it satisfies \crefrange{GGprop1}{GGprop3}.
\end{defn}

\begin{rem}
\Cref{lem:ZgradSS} implies that any good grading $\Gamma$ for a nilpotent $e \in \fg$
can be expressed as the eigenspaces of a semisimple element $h_\Gamma$ for which
$[h_\Gamma,e] = 2e$; however, this does not mean that any good grading comes from an
\sltriple\ $\set{e,h_\Gamma,f}$. In particular, there may not exist an element
$f$ which completes $\set{e,h_\Gamma}$ to an \sltriple. Hence, though every
Dynkin grading is a good grading, there exist good gradings which are not
Dynkin, as we shall see later in \cref{sec:TypeAPyramids}.
\end{rem}

\begin{note}
A $\ZZ$\grading\ $\Gamma$ of $\fg$ is said to be \emph{good} if there exists a nilpotent
$e \in \fg$ for which it is a good grading; further, such a nilpotent is
\emph{good} for $\Gamma$. The $\ZZ$\grading\ $\Gamma$ is said to be an \emph{even grading}
if $\fg_{2j+1} = \set{0}$ for any integer $j$.
\end{note}

\begin{defn} \label{def:Gammasl2triple}
Let $\Gamma$ be a good grading for the nilpotent element $e \in \fg$. A $\Gamma$\graded\
\sltriple\ is an \sltriple\ $\set{e,h,f}$ such that
$e \in \fg_2$, $h \in \fg_0$ and $f \in \fg_{-2}$.
\end{defn}

\begin{note}
It should again be emphasised that a $\Gamma$\graded\ \sltriple\ is only
compatible with the good grading $\Gamma$ in the sense outlined in the definition. In
particular, the Dynkin grading coming from $\ad h$ is not generally the same as
$\Gamma$.
\end{note}

\begin{lem}
For any non-zero nilpotent $e \in \fg$ and good $\ZZ$\grading\ $\Gamma$, there exists a
$\Gamma$\graded\ \sltriple\ $\set{e,h,f}$.
\end{lem}

\begin{proof}
By the \JMthm, the nilpotent $e$ can be completed to an \sltriple\
$\set{e,h',f'}$. Let $h' = \sum_{j \in \ZZ} h_j$ and $f' = \sum_{j \in \ZZ} f'_j$ be the
decompositions with respect to $\Gamma$, and define $h = h_0$: it follows that
$[h,e] = 2e$ and $h = [e,f'_{-2}]$. Next, construct $\tilde{f}$ as the component
of $f'_{-2}$ which lies in the $-2$ eigenspace of $\ad h$. We now have an
\sltriple\ $\set{e,h,\tilde{f}}$, though it may no longer be the case that
$\tilde{f} \in \fg_{-2}$; however, taking $f$ to be the $-2$ component of
$\tilde{f}$ with respect to $\Gamma$ provides the required \sltriple. In fact
we note that $\liebr[\big]{e,f-\tilde{f}} = 0$, which implies that
$f = \tilde{f}$ by \cref{GGprop2}.
\end{proof}

\subsubsection{The characteristic of a grading}

We will be interested in determining what the possible good gradings are in a
given Lie algebra. The answer is known, though quite complicated in general;
however, we can make some preliminary remarks greatly narrowing down the
possibilities. We begin by defining the \emph{characteristic} of a
$\ZZ$\grading. We note that $\fg_0$ is a reductive subalgebra of $\fg$, and a
Cartan subalgebra $\fh$ of $\fg_0$ is also a Cartan subalgebra of $\fg$; we
consider the root space decomposition $\fg = \fh \oplus \bigoplus_\alpha \fg_\alpha$. Let
$\Delta_0^+$ be a system of positive roots of pure degree in $\fg_0$; the set
$\Delta^+ \coloneqq \Delta_0^+ \cup \set{\alpha \st \fg_\alpha \subseteq \fg_{>0}}$ forms a system of positive roots in
$\fg$. Choose a set $\Phi$ of simple roots in $\Delta^+$ and let $\Phi_j \coloneqq \Phi \cap \fg_j$ for
each $j \ge 0$. 

\begin{defn}
The \emph{characteristic} of a $\ZZ$\grading\ is the decomposition
$\Phi = \bigcup_{j\ge0} \Phi_j$.
\end{defn}

We note that there is a bijection between the set of $\ZZ$\gradings\ on $\fg$ up
to conjugation and the set of all possible characteritics.

\begin{prop}
If $\Gamma$ is a good grading for a nilpotent $e$, then $\Phi = \Phi_0 \cup \Phi_1 \cup \Phi_2$.
\end{prop}

\begin{proof}
Let $\Phi = \set{\alpha_1, \dotsc \alpha_r}$, and assume there exists some simple root
$\alpha_j \notin \Phi_0 \cup \Phi_1 \cup \Phi_2$; $\alpha_j$ must therefore lie in $\Phi_k$ for some $k>2$.
Let $e_\alpha$ be a generator of the weight space $\fg_\alpha$. Since $e \in \fg_2$, it must
lie in the subalgebra generated by $\set{\alpha_i \st i \neq j}$. Hence
$[e,e_{-\alpha_j}] = 0$, which violates \cref{GGprop2}.
\end{proof}

\begin{cor}
If $\Gamma$ is an even good grading for a nilpotent $e$, then $\Phi = \Phi_0 \cup \Phi_2$.
\end{cor}

Thus we can conclude that there are only a finite number of good gradings
possible up to conjugacy: the total number is bounded by $3^{\rank \fg}$.
Unfortunately this bound is not sharp, and there are in general many fewer good
gradings than would be suggested here.

\subsection{A bijection between nilpotent orbits and
\texorpdfstring{$\fsl_2$}{sl₂}-triples}

Our objective in this \namecref{sec:Nilorbits} has been to introduce the tools
necessary for the study of nilpotent orbits in $\nilcone$. The \JMthm\ states
that any nilpotent element $e$ can be completed to an \sltriple, however
it doesn't tell us how many ‘different’ such triples there are: we don't know whether the
triple is ‘essentially unique’, or if there a number of different ‘inequivalent’
\sltriples\ possible. More concretely, we consider two \sltriples\
equivalent if they are conjugate (i.e.\ in the same orbit) under the adjoint
action of $G$. We can then construct a map from the set of equivalence classes
of \sltriples\ to the set of non-zero nilpotent orbits in~$\fg$.
\begin{equation*}
\begin{split}
\Omega \colon \set{\fsl_2\text{-triples in }\fg}/G
  & \to \set{\text{non-zero nilpotent orbits in }\fg} \\
[\set{e,h,f}] & \mapsto \orbit_e
\end{split}
\end{equation*}

\begin{thm}[Kostant]
\label{thm:orbitbij}
The map $\Omega$ is a bijection.
\end{thm}

This \namecref{thm:orbitbij} tells us that when considering nilpotent orbits in
$\fg$, we are completely justified in constructing \sltriples, as there is
a unique conjugacy class of \sltriples\ for each orbit. Any constructions
we make for a nilpotent orbit can be made using a choice of \sltriple\
without worrying about different choices yielding different results.

\begin{proof} \cite[§3.4]{CM:Nilorbits}
The map $\Omega$ is well-defined, as two conjugate \sltriples\ will have
elements $e$ lying in the same nilpotent orbit. The fact that it is surjective
follows directly from the \JMthm. It remains only to show that $\Omega$ is injective.

Assume we have two \sltriples\ with the same image under $\Omega$; without loss
of generality, we can conjugate so the triples have the form $\set{e,h,f}$ and
$\set{e,h',f'}$. Consider the Lie algebra $\fu_e \coloneqq \liez{e} \cap [\fg,e]$: we
note that $\fu_e$ is an $\ad h$-invariant ideal of $\liez{e}$, and that 
$\fu_e = \liez{e}_{>0}$. This second identity follows from \cref{GGprop4},
and the fact that $[\fg,e] \cap \fg_0 = [e,\fg_{-2}]$, and hence does not commute
with $e$ by $\fsl_2$ representation theory (see e.g.~\cref{fig:sl2rep}). That
$\fu_e$ lies in strictly positive degree further implies that it is a
nilpotent ideal of $\liez{e}$.

Let $U_e$ be the connected subgroup of $G$ with Lie algebra $\fu_e$. Since
$\fu_e$ is nilpotent the exponential map is is a diffeomorphism, and the
adjoint action has a particularly simple expression: for $x \in \fu_e$ and
sufficiently large $n$,
\begin{equation*}
\exp(x) \cdot h = h + [x,h] + \frac{1}{2}[x,[x,h]] + \dotsb +
\frac{1}{n!}[x,[x,\dotso,[x,h]]].
\end{equation*}
The fact that $\fu_e$ is $\ad h$-invariant implies that
$U_e \cdot h \subseteq h + \fu_e$. We can show that this inclusion is in fact an
equality either by constructing an appropriate element of $U_e$ directly
(cf.~\cite[Lemma~3.4.7]{CM:Nilorbits}), or by observing that the orbit of the
unipotent group is a Zariski-closed dense subset of $h + \fu_e$, and hence
$h + \fu_e$ itself (cf.~\cite[Lemma~3.7.21]{CG:RepTheoryGeom}).

We can now use the fact that $h-h' \in \fu_e$ to
see that $h' \in U_e \cdot h$, and hence have proven that there exists an element
of $U_e \subseteq G$ which fixes $e$ under the adjoint action and sends $h$ to $h'$.
This allows us to conjugate $\set{e,h',f'}$ to $\set{e,h,f''}$, and the fact
that $[e,f-f''] = 0$ implies that $f = f''$ by \cref{GGprop2}, thus completing
the proof.
\end{proof}

\subsection{Slodowy slices}
\label{sec:SlodowySlices}

For a given nilpotent orbit $\orbit_e$, we will be interested in studying the
structure of certain transverse slices to $\orbit_e$ at a given point. While
$\orbit_e$ has many different transverse slices passing through any individual
point, there is a certain relatively natural class of such slices which we'll be
working with. Consider a nilpotent element $e \in \fg$ and complete it to an
\sltriple\ $\set{e,h,f}$ using the \JMthm.

\begin{defn}
The \emph{Slodowy slice} to $\orbit_e \subseteq \fg$ through $e$ is
$\slodowy_e \coloneqq e + \liez{f}$, where $\liez{f}$ is the centraliser of $f$.
\end{defn}

\begin{rem}
Though the varieties $\orbit_e$ and $\slodowy_e$ defined above lie in the
semisimple Lie algebra $\fg$, we can instead view them in $\fg^*$ using the
duality $\kappa \colon \fg \isoto \fg^*$
induced by the Killing form. In fact, many of the constructions developed later
are more naturally considered in $\fg^*$. Which ambient space we are considering
will generally be clear from the context, but when clarity is required we shall
consider $\chi = \killing{e}{\cdot} \in \fg^*$, $\orbit_\chi = G \cdot \chi \subseteq \fg^*$ its
orbit under the coadjoint action, and
$\slodowy_\chi \coloneqq \kappa(e + \liez{f}) = \chi + \ker \ad^* f$.
\end{rem}

\begin{prop} \label{prop:SContract}
The Slodowy slice $\slodowy_e$ has a contracting $\CC*$\action\ which fixes $e$.
\end{prop}

\begin{proof}
Consider an \sltriple\ containing $e$: this exponentiates to an embedding
$\iota \colon SL_2 \hookrightarrow G$. We can then choose a cocharacter
$\gamma \colon \CC* \to G$ by defining
$\gamma(t) = \iota \begin{psmallmatrix} t&0\\0&t^{-1} \end{psmallmatrix}$. Note that
$\Ad_{\gamma(t)} e = t^2 e$.

Consider the action of $\CC*$ on $\slodowy_e$ given by
$t \cdot (e + x) \coloneqq t^{-2} \Ad_{\gamma(t)} (e + x) = e + t^{-2} \Ad_{\gamma(t)} x$
for $x \in \liez{f}$. We can see from the second equality that this action fixes
$e$, so it just remains to show that $\lim_{t \to \infty} t \cdot (e+x) = e$ for any
$x \in \liez{f}$. This follows from version of \cref{GGprop4} for the nilpotent
$f$, which implies that $\Ad_{\gamma(t)}$ acts on $\liez{f}$ by negative powers of
$t$, and hence $t^{-2} \Ad_{\gamma(t)}$ acts by strictly negative powers.
\end{proof}

\begin{note}
In the course of the above proof, we showed that by completing any nilpotent
element to an \sltriple\ we can express $t^2 e = \Ad_{\gamma(t)} e$ for some
cocharacter $\gamma$. This means that non-zero scalar multiplication of a nilpotent
element preserves $G$-orbits.
\end{note}

\begin{prop} \label{prop:SOTrans}
The Slodowy slice $\slodowy_e$ is a transverse slice to $\orbit_e$ at the point
$e$, and in particular $T_e \slodowy_e \oplus T_e \orbit_e = T_e \fg$. Furthermore,
$e$ is the unique point of intersection:
$\slodowy_e \cap \orbit_e = \set{e}$.
\end{prop}

\begin{proof} \cite[Proposition~3.7.15]{CG:RepTheoryGeom}
To prove that $\slodowy_e$ is a transverse slice to $\orbit_e$, it suffices
to prove that $T_e \slodowy_e \oplus T_e \orbit_e = T_e \fg$; the stronger result in 
a neighbourhood of $e$ follows automatically by exponentiation.

Note that $T_e \slodowy_e = \liez{f}$ and $T_e \orbit_e = [\fg,e]$; that
$\liez{f} \cap [\fg,e] = {0}$ follows directly from $\fsl_2$ representation theory.
To show that they together span $T_e \fg = \fg$, we count dimensions. Consider
the decomposition of $\fg$ into irreducible $\fsl_2$\modules\
$\fg = \bigoplus_{j=1}^n V(\lambda_j)$. The centraliser $\liez{f}$ consists of
precisely those elements of $\fg$ which lie in the lowest weight spaces of the
irreducible components, and $[\fg,e]$ consists of those elements which do not
lie in the lowest weight spaces of the irreducible components (see
\cref{fig:sl2repker}). Thus $\dim \liez{f} + \dim [\fg,e] = \dim \fg$, and so
$T_e \slodowy_e \oplus T_e \orbit_e = T_e \fg$.


\begin{figure}[hf]

\tikzstyle{dot}=[circle,fill=black,inner sep=0,minimum size=3mm]

\centering
\begin{tikzpicture}[scale=.5,auto]

  \draw[<->] (-.5,7) -- (-.5,-7);
  \foreach \x in {6,4,...,-6} {
    \draw (-.8,\x) -- (-.5,\x);
    \node (tick \x)  at (-.5,\x) [label=left:$\x$] {};
  }

  \foreach \x in {6,4,...,-6}
    \node at (1,\x) [dot] {};
  \draw (1,6) -- (1,-6);

  \foreach \x in {5,3,...,-5}
    \node at (3,\x) [dot] {};
  \draw (3,5) -- (3,-5);

  \foreach \x in {5,3,...,-5}
    \node at (5,\x) [dot] {};
  \draw (5,5) -- (5,-5);

  \foreach \x in {3,1,-1,-3}
    \node at (7,\x) [dot] {};
  \draw (7,3) -- (7,-3);

  \foreach \x in {2,0,-2}
    \node at (9,\x) [dot] {};
  \draw (9,2) -- (9,-2);

  \draw (0,-5)                               -- (1,-5)
               .. controls (2,-5) and (2,-4) .. (3,-4)
                                             -- (5,-4)
               .. controls (6,-4) and (6,-2.5) .. (7,-2)
               .. controls (8,-1.5) and (8.5,-1) .. (9,-1)
                                               -- (10,-1);
  \node at (8,-5) {$\liez{f}$};
  \node at (8,5) {$[\fg,e]$};
\end{tikzpicture}

\caption{
The decomposition of a finite-dimensional $\fsl_2$-module as a sum of
simple modules. The dimension of $\fg$ is equal to the sum of the dimensions of
$\liez{f}$ and $[\fg,e]$.}

\label{fig:sl2repker}
\end{figure}
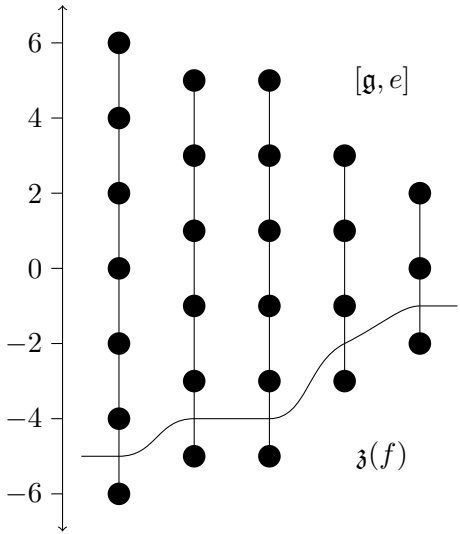

It remains to prove that $\slodowy_e \cap \orbit_e = \set{e}$. We proved above that
$\slodowy_e$ is a transverse slice to $\orbit_e$ in some sufficiently small
neighbourhood, so any other points of intersection must lie outside of that
neighbourhood. We note that the contracting action of \cref{prop:SContract}
preserves $G$-orbits, as it is composed of an honest adjoint action of $G$ followed by a
scaling of the resulting nilpotent, which can also be expressed as a $G$\action\ 
by the note following the \namecref{prop:SContract}. Thus we can contract any
point of $(\slodowy_e \cap \orbit_e) \smallsetminus \set{e}$ to another point in
the same set in an arbitrarily small neighbourhood of $e$, contradicting
transversality of $\slodowy_e$.
\end{proof}

\section{W-algebra basics}
\label{sec:WalgBasics}

We now turn to defining the \Walgebras\ themselves and establishing their basic
properties. Throughout this \namecref{sec:WalgBasics} we shall develop a
procedure for defining the \Walgebra\ $U(\fg,e)$ given a nilpotent $e$ and
good grading $\Gamma$, establish its identity as a non-commutative filtered algebra,
and remark on some of the geometry linking the \Walgebras\ to Slodowy slices. In
the process we shall prove the independence of the isomorphism class of the
\Walgebra\ $U(\fg,e)$ from the various choices our definition entails.

\subsection{Premet subalgebras}

For our definition of \Walgebras, we need to introduce a class of nilpotent
subalgebras compatible with the nilpotent element $e$. These subalgebras are
known as \emph{Premet subalgebras}, and they are almost entirely determined by a
choice of a good grading for $e$. In the case of an even good grading we can
unambiguously define a Premet subalgebra purely from the good grading itself,
but if the grading has non-zero odd component we need to be a bit more subtle.

\begin{lem}
Let $\fg = \bigoplus_{j \in \ZZ} \fg_j$ be a good grading for the nilpotent $e$.
The space $\fg_{-1}$ is a symplectic vector space with symplectic form
$\omega(x,y) \coloneqq \killing{e}{[x,y]}$.
\end{lem}

\begin{proof}
That the form $\omega$ is antisymmetric follows from the antisymmetry of the Lie
bracket. To prove that it is non-degenerate, we show that its radical
$\operatorname{rad} \omega \coloneqq \set{x \in \fg_{-1} \st \omega(x,y) = 0 \;\;\forall \; y \in \fg_{-1}}$,
is zero. Let $x \in \operatorname{rad} \omega$: since
$0 = \omega(x,\cdot) = \killing{e}{[x,\cdot]} = \killing{[e,x]}{\cdot}$,
we check to see when this vanishes as an operator on $\fg_{-1}$. 
By \cref{GGprop2,GGprop5}, $[e,x]$ is a non-zero element of
$\fg_1$ and $\killing{\fg_1}{\fg_j} = 0$ unless $j = -1$, so $[e,x]$ is a
non-zero element of the radical of the Killing form. However, $\fg$ is a
semisimple Lie algebra and so its Killing form is non-degenerate. Therefore $x =
0$, $\operatorname{rad} \omega = \set{0}$, and $\omega$ is non-degenerate.
\end{proof}

\begin{defn} \label{def:Premet}
Let $e$ be a nilpotent element in $\fg$. A \emph{Premet subalgebra} for $e$ is a
subalgebra $\fm \subseteq \fg$ constructed using a choice of a good grading
$\Gamma \colon \fg = \bigoplus_{j \in \ZZ} \fg_j$ along with a Lagrangian subspace
$\fl \subseteq \fg_{-1}$; it is defined as $\fm = \fl \oplus \bigoplus_{j\le-2} \fg_j$.
\end{defn}

Note that this is closed under the Lie bracket by the fact that $\Gamma$ is a Lie
algebra grading.
Premet subalgebras are closely tied to the structure of the nilpotent orbit
$\orbit_e$, and enjoy a number of useful properties.

\begin{prop} \label{prop:PremProp}
Let $\fm \subseteq \fg$ be a Premet subalgebra for a nilpotent $e$; then:
\begin{enumerate}
\item \label[property]{PremDim}
    $\dim \fm = \frac{1}{2} \dim \orbit_e$.
\item \label[property]{PremNil}
    $\fm$ is an ad-nilpotent, and in particular nilpotent, subalgebra of $\fg$.
\item \label[property]{PremChar}
    the linear functional $\chi = \killing{e}{\cdot}$ restricts to a character on
    $\fm$.
\end{enumerate}
\end{prop}

\begin{proof}
\Cref*{PremDim} follows from the orbit--stabiliser theorem and \cref{GGprop6}:
\begin{multline*}
\textstyle
\dim \orbit_e = \dim \fg - \dim \liez{e}
  = \sum_{j \in \ZZ} \fg_j - \dim \fg_0 - \dim \fg_1 = \\
\textstyle
  = \dim \fg_{-1} + \sum_{j \le -2} \dim (\fg_j + \fg_{-j})
  = \dim \fg_{-1} + 2 \sum_{j \le -2} \fg_j = 2 \dim \fm.
\end{multline*}
\Cref*{PremNil} follows because $\fm \subseteq \bigoplus_{j \le -1} \fg_j$, and so
consists entirely of ad-nilpotent elements. For \cref*{PremChar} we note first
that $e \in \fg_2$, and so \cref{GGprop5} implies that $\chi$ vanishes except on
$\fg_{-2}$. Thus we know that $\chi([x,y]) = 0$ unless $x$ and $y$ both lie in
$\fl = \fm \cap \fg_{-1}$; but $\chi([x,y]) = \omega(x,y)$, which vanishes since $\fl$ was
chosen to be Lagrangian with respect to $\omega$.
\end{proof}

\subsubsection{Premet subalgebras for even good gradings}

Fortunately the situation is simpler for even good gradings, and we can give an
intrinsic characterisation of all Premet subalgebras which can be constructed
from an even good grading. We note that if $\Gamma$ is an even good grading, the
Premet subalgebra $\fm = \bigoplus_{j < 0} \fg_j$ is the nilradical of a
parabolic subalgebra $\fp^- = \bigoplus_{j\le0} \fg_j$
(cf.~\cite[Lemma~3.8.4]{CM:Nilorbits}). The converse is true with one additional
condition. Choose a Cartan subalgebra $\fh$ and set of simple roots $\Phi$ such
that $\fp^- = \fp_{-\Theta}$ for some subset $\Theta \subseteq \Phi$. Let $\Delta^+$ be the positive roots
and $\Delta^+_\Theta$ the elements of $\Delta^+$ with exactly one simple summand in $\Theta$.

\begin{thm}[Elashvili--Kac]{\normalfont \cite[Theorem~2.1]{EK:ClassGG}}
If $\fm$ is the nilradical of a parabolic subalgebra $\fp$ in $\fg$, then it
is a Premet subalgebra if and only if there exists a Richardson element --
that is, an element of the open dense orbit in $\fm$ under the adjoint action of
$P$ -- lying in the subspace generated by $\set{e_\alpha \st \alpha \in \Delta^+_\Theta}$.
\end{thm}

\begin{note}
Though Richardson elements exist for any parabolic subalgebra, it is not always
true that there exists one in satisfying the condition of the theorem. Such
parabolics are called \emph{nice parabolic subalgebras}, and their
classification is equivalent to the classification of good gradings.
\end{note}

\subsection{The Whittaker definition of W-algebras}
\label{sec:WalgDef}

We can now present a definition of the \Walgebra\ $U(\fg,e)$ associated to the
nilpotent $e$ using a Premet subalgebra. \Cref{prop:PremProp} states that
$\chi = \killing{e}{\cdot} \colon \fm \to \CC$ is a Lie algebra character, and so
defines a one-dimensional representation $\CC_\chi$. We can induce this
representation of $U(\fm)$ to a representation of $U(\fg)$, the result being
known as the \emph{generalised Gel'fand--Graev module} $Q_\chi$:
\begin{equation*}
Q_\chi \coloneqq U(\fg) \otimes_{U(\fm)} \CC_\chi
\end{equation*}
Note that the kernel of the induced morphism $\chi \colon U(\fm) \to \CC$ is
generated as a two-sided ideal in $U(\fm)$ by the shifted Lie algebra
$\fm_\chi = \set{a - \chi(a) \st a \in \fm}$. Considering the left ideal $U(\fg)\fm_\chi$
in $U(\fg)$, we can see that $Q_\chi \simeq U(\fg) \big/ U(\fg)\fm_\chi$.

\begin{prop}
The Lie algebra $\fm_\chi$ acts locally nilpotently on $Q_\chi$.
\end{prop}

\begin{proof}
Since $Q_\chi$ is a quotient of $U(\fg)$, it suffices to see how $\fm_\chi$ acts on
PBW monomials in $U(\fg)$. Let $\overline{u} \in Q_\chi$ be such that
$u = u_1 u_2 \dotsm u_n$ is a PBW monomial. Since $\fm$ consists of ad-nilpotent
elements, for each $a \in \fm$ we can consider $N \in \NN$ large enough so that
$(\ad a)^N (u_j) = 0$ for all $j$. Then $(\ad a)^{nN} (u) = 0$, and hence large
enough powers of $a$ will commute with $u$; it follows that $a - \chi(a)$ acts
locally nilpotently.
\end{proof}

\begin{defn} \label{def:Walg}
Given a nilpotent $e \in \fg$, good grading $\Gamma$ and Langrangian subspace
$\fl \subseteq \fg_{-1}$, the \emph{(finite) \Walgebra} $U(\fg,e)$ is the space of
Whittaker vectors in $Q_\chi$, that is:
\begin{equation} \label{eq:WalgDef}
U(\fg,e) \coloneqq \pp[\big]{Q_\chi}^{\fm_\chi} =
  \set{\overline{u} \in U(\fg) \big/ U(\fg)\fm_\chi \st (a - \chi(a)) u \in U(\fg)\fm_\chi
  \;\;\forall \;a \in \fm_\chi }.
\end{equation}
\end{defn}

\begin{rem}
In fact, the identity $(a-\chi(a)) u = [a,u] + u (a-\chi(a))$ allows us to further
identify 
\begin{equation} \label{eq:WalgAdDef}
U(\fg,e) = \pp[\big]{Q_\chi}^{\ad \fm} =
  \set{\overline{u} \in U(\fg) \big/ U(\fg)\fm_\chi \st [a,u] \in U(\fg)\fm_\chi \;\;\forall \;a \in
  \fm_\chi }.
\end{equation}
We can observe that $U(\fg,e)$ is not just a vector space, but also inherits an
algebra structure from $U(\fg)$, as
\begin{equation*}
[a, u_1 u_2] = [a, u_1] u_2 + u_1 [a, u_2] \in U(\fg)\fm_\chi u_2 + U(\fg)\fm_\chi =
  U(\fg)\fm_\chi,
\end{equation*}
where the final equality follows from \cref{eq:WalgDef}.
This is due solely from the fact that $\chi$ is a character of $\fm$; the quotient
of a non-commutative algebra by a left ideal does not not inherit an algebra
structure in general.
\end{rem}

There is another way to see the algebra structure on $U(\fg,e)$ by studying the
structure of the endomorphisms of the Gel'fand--Graev module $Q_\chi$. It is a
cyclic module, so to specify an endomorphism it suffices to define the image of
$\overline{1} \in Q_\chi$: as a vector space,
\begin{equation*}
\End_{U(\fg)}(Q_\chi) = \Hom_{U(\fg)}(U(\fg) \big/ U(\fg)\fm_\chi, Q_\chi) =
  \pp[\big]{Q_\chi}^{\fm_\chi}.
\end{equation*}
This isomorphism as vector spaces is furthermore an isomorphism of algebras
\begin{equation} \label{eq:WalgEndDef}
U(\fg,e) \simeq \End_{U(\fg)}(Q_\chi)\op.
\end{equation}
This can be taken as an alternate definition of the \Walgebra.

\begin{rem}
It should be noted that these definitions don't just depend on the nilpotent
$e$, but depend also on the choice of good grading $\Gamma$ and Lagrangian subspace
$\fl$. Fortunately, the resulting algebras are isomorphic for different choices
of $\Gamma$ and $\fl$. We will therefore omit them from the notation. The
independence of Lagrangian subspace $\fl$ will be shown in
\cref{sec:WalgGeometry}, while the independence of $\Gamma$ is shown by Brundan and
Goodwin in \cite{BG:GGPoly}.
\end{rem}

\begin{eg}
If we restrict to the case where $e=0$, the only good grading for $e$ is the
trivial grading $\fg = \fg_0$. We therefore have that $\chi = 0$, $\fm_\chi =
\set{0}$, $Q_\chi = U(\fg)$, and so $U(\fg,0) = U(\fg)$.
\end{eg}

\subsubsection{W-algebras for even good gradings}

The above definition can be significantly simplified in the case that the good
grading $\Gamma$ is even, as the Lagrangian subspace $\fl$ will vanish. Consider the
parabolic subalgebra $\fp \coloneqq \bigoplus_{j \ge 0} \fg_j$ and its opposite
nilradical $\fm = \bigoplus_{j \le -2} \fg_j$; we have a vector space
decomposition $\fg \simeq \fp \oplus \fm$. The PBW theorem then implies that we have an
algebra decomposition $U(\fg) = U(\fp) \oplus U(\fg)\fm_\chi$. This provides us with
a $\fm_\chi$\module\ isormorphism $U(\fg) \big/ U(\fg)\fm_\chi \simeq U(\fp)$, where $\fm_\chi$
acts by the $\chi$-twisted adjoint action. With these observations,
\cref{eq:WalgAdDef} reduces to
\begin{equation} \label{eq:WalgEvenDef}
U(\fg,e) = U(\fp)^{\ad \fm} \coloneqq \set{u \in U(\fp) \st [a,u] \in U(\fg)\fm_\chi
  \;\;\forall \; a \in \fm}.
\end{equation}
As a result, $U(\fg,e)$ can be identified as a subalgebra of $U(\fp)$,
rather than a subquotient of $U(\fg)$: this greatly simplifies calculations in
many examples.

\begin{eg}
A classical result of Kostant \cite{Kos:WhitRep} states that for a regular
nilpotent element $e_{\text{reg}} \in \fg$, $U(\fg,e_{\text{reg}}) \simeq Z(\fg)$. We
can demonstrate this here in an example.

Let $\fg = \fsl_2$,
$e = \begin{psmallmatrix} 0&1\\0&0 \end{psmallmatrix}$,
$h = \begin{psmallmatrix} 1&0\\0&-1 \end{psmallmatrix}$, and
$f = \begin{psmallmatrix} 0&0\\1&0 \end{psmallmatrix}$;
this implies that $\fp = \gen{e,h}$ and $\fm_\chi = \gen{f-1}$. It can be
checked that $\frac{1}{2}h^2 - h + 2e$ lies in $U(\fp)^{\ad \fm}$, and in
fact it freely generates it. We therefore have that
$U(\fg,e) = \Cpoly[\big]{\frac{1}{2}h^2 - h + 2e}$, which is isomorphic to
$Z(\fg) = \Cpoly[\big]{\frac{1}{2}h^2 + ef + fe}$ under the projection
$U(\fg) = U(\fp)\otimes U(\fm) \twoheadrightarrow U(\fp)$.
\end{eg}

\subsection{The Kazhdan filtration}

As it stands the \Walgebra\ is an algebraic construction, however it is closely
related to the geometry of Slodowy slices. In order to discuss this, we need to
introduce a filtration on $U(\fg,e)$.

Recall that $U(\fg)$ has a filtration known as the PBW filtration, where
$U_j(\fg)$ is spanned by the collection of all monomials $x_1 x_2 \dotsm x_i$
for $i \le j$, $x_k \in \fg$.
We would like to modify this filtration to take the good grading $\Gamma$ into
account; consider a semisimple element $h_\Gamma$ from \cref{lem:ZgradSS}.
We can extend the grading on $\fg$ to a grading on $U(\fg)$ by defining
$U(\fg)_i \coloneqq \set{u \in U(\fg) \st \ad h_\Gamma(u) = iu}$, and combine the two by
defining
\begin{equation*}
U_j(\fg)_i \coloneqq U_j(\fg) \cap U(\fg)_i = \set{u \in U_j(\fg) \st \ad h_\Gamma(u) = iu}.
\end{equation*}

\begin{defn}
The \emph{Kazhdan filtration} on $U(\fg)$ associated to the grading $\Gamma$ of $\fg$
is defined by
\begin{equation*}
F_n U(\fg) = \sum_{i + 2j \le n} U_j(\fg)_i.
\end{equation*}
\end{defn}

\begin{rem} \label{rem:KazhdanProps}
The Kazhdan filtration enjoys a number of useful properties:
\begin{enumerate}
\item 
  If we consider $x \in \fg_i$ and $y \in \fg_j$, it follows that $x \in F_{i+2}
  U(\fg)$ and $y \in F_{j+2} U(\fg)$, and therefore $[x,y] \in F_{i+j+2} U(\fg)$.
  Hence the associate graded $\gr U(\fg)$ is commutative, and is thus isomorphic
  to $\Sym(\fg) = \Cpoly{\fg^*}$.

\item
  The left ideal $U(\fg)\fm_\chi$ contains all the elements of $U(\fg)$ of strictly
  negative degree; therefore, the Kazhdan filtration descends to a positive
  filtration on the Gel'fand--Graev module $Q_\chi = U(\fg) \big/ U(\fg)\fm_\chi$.

\item
  Passing to the associated graded algebra, the left ideal $U(\fg)\fm_\chi$ passes
  to the two-sided ideal $\gr U(\fg)\fm_\chi \subseteq \Cpoly{\fg^*}$ consisting of all
  functions which vanish on $\chi + \fm^{*,\bot}$, where
  $\fm^{*,\bot} = \set{\xi \in \fg^* \st \xi(a) = 0 \;\;\forall \; a \in \fm}$.

\item
  The associated graded of the Gel'fand--Graev module is
  $\gr Q_\chi \simeq \gr U(\fg) \big/ \gr U(\fg)\fm_\chi$, which is a positively-graded
  commutative algebra in turn isomorphic to $\Cpoly[\big]{\chi + \fm^{*,\bot}}$.
  Furthermore, the natural map $\gr Q_\chi \to \Cpoly[\big]{\chi + \fm^{*,\bot}}$ is
  an algebra isomorphism.

\item
  The Kazhdan filtration descends to a positive filtration of the \Walgebra\ 
  $U(\fg,e) = \pp[\big]{Q_\chi}^{\fm_\chi}$, and furthermore $F_0 U(\fg,e) = \CC$.
\end{enumerate}
\end{rem}

\subsection{A geometric interpretation of W-algebras}
\label{sec:WalgGeometry}

These properties suggest a connection between the filtered algebras and
Gel'fand--Graev module associated to the \Walgebra\ and the functions on a number
of subvarieties of $\fg^*$. The question then remains: what is the associated
graded algebra of the \Walgebra\ $U(\fg,e)$? We shall show that it corresponds to
the ring of functions on the Slodowy slice $\Cpoly[\big]{\slodowy_\chi}$ (recall
\cref{sec:SlodowySlices}), however it will take a little effort to do so. The
relationship between the associated graded algebras can be summarised in the
following \namecref{thm:WalgQuant}. Recall the definition of a $\Gamma$\graded\ 
\sltriple\ (\cref{def:Gammasl2triple}).

\begin{thm}[Gan--Ginzburg]{\normalfont \cite{GG:QuantSlod}}
\label{thm:WalgQuant}
Consider a nilpotent $e \in \fg$ with an associated good grading~$\Gamma$ and Premet
subalgebra~$\fm$. Let $\set{e,h,f}$ be a $\Gamma$\graded\ \sltriple, and
$\slodowy_e \subseteq \fg$ and $\slodowy_\chi \subseteq \fg^*$ the corresponding Slodowy slices.
Then the following diagram of commutative algebras commutes:


\begin{equation*}
\begin{tikzcd}[]
\gr U(\fg)                         \rar[-, double equal sign distance]                \dar
  & \Cpoly{\fg^*}                  \rar[leftrightarrow]{\sim} \dar
    & \Cpoly{\fg}                                             \dar \\
\gr Q_\chi                            \rar[-, double equal sign distance]
  & \Cpoly[\big]{\chi + \fm^{*,\bot}} \rar[leftrightarrow]{\sim} \dar
    & \Cpoly[\big]{e + \fm^\bot}                              \dar \\
\gr U(\fg,e)                       \rar{\sim}                 \uar
  & \Cpoly[\big]{\slodowy_\chi}       \rar[leftrightarrow]{\sim}
    & \Cpoly[\big]{\slodowy_e}                                     \\
\end{tikzcd}
\end{equation*}

\end{thm}

\noindent
We first summarise the arrows we already know.
\begin{itemize}
\item
  The maps in the first column arise from the fact that
  $Q_\chi \coloneqq U(\fg) \big/ U(\fg)\fm_\chi$ is a quotient module and
  $U(\fg,e) \coloneqq \pp[\big]{Q_\chi}^{\fm_\chi}$ is a submodule.  
\item
  The equalities between the first and second columns follow from
  \cref{rem:KazhdanProps}. 
\item
  The isomorphisms between the second an third columns are induced by the
  isomorphism $\kappa \colon \fg \isoto \fg^*$ coming from the Killing form.
\item
  The maps from the first to the second row in the second and third columns are
  both restriction of functions.
\end{itemize}
There are three arrows left to define: the bottom maps in the second and third
columns, and the first map in the third row. However, it suffices to define the
map $\Cpoly[\big]{e + \fm^\bot} \to \Cpoly[\big]{\slodowy_e}$, and the remaining
two follow:
\begin{itemize}
\item
  The map $\Cpoly[\big]{\chi + \fm^{*,\bot}} \to \Cpoly[\big]{\slodowy_\chi}$ shall be
  defined by passing through the isomorphism $\kappa$.
\item
   The final map $\gr U(\fg,e) \to \Cpoly[\big]{\slodowy_\chi}$ shall be defined as
   the composition of the arrows
   $\gr U(\fg,e) \to \gr Q_\chi \to \Cpoly[\big]{\chi + \fm^{*,\bot}} \to
     \Cpoly[\big]{\slodowy_\chi}$.
\end{itemize}
Therefore, after defining the map
$\Cpoly[\big]{e + \fm^\bot} \to \Cpoly[\big]{\slodowy_e}$,
to complete the proof of \cref{thm:WalgQuant} it only remains to show that the
map $\gr U(\fg,e) \to \Cpoly[\big]{\slodowy_\chi}$ is an isomorphism. Our proof
shall follow the exposition of Wang \cite{Wan:NilOrbWAlg}.

\begin{rem}
We will temporarily relax some of our definitions slightly to allow a bit more
flexibility than strictly necessary to explain \cref{thm:WalgQuant}. This does
not complicate the proof, and it allows us to prove the independence of choice
of Lagrangian subspace $\fl \subseteq \fg_{-1}$ in the definition of $U(\fg,e)$ as a
side benefit.

Let $\fl$ be an \emph{isotropic subspace} of $\fg_{-1}$ (i.e.~$\omega(\fl,\fl) = 0$),
rather than the stronger condition of being a Lagrangian subspace; we continue
to define the Premet subalgebra $\fm \coloneqq \fl \oplus \bigoplus_{j\le-2} \fg_j$. We 
introduce the new definitions:
\begin{align}
\label{eq:LagrangianPrime}
\fl' & \coloneqq \fl^{\bot_\omega} = \set{x \in \fg_{-1} \st \omega(x,\fl) = 0}, \\
\label{eq:PremetPrime}
\fm' & \coloneqq \fl' \oplus \bigoplus_{j\le-2} \fg_j.
\end{align}
Note that it we choose $\fl$ to be a Lagrangian (and hence isotropic) subspace,
then it follows from the definition that $\fl = \fl'$ and $\fm = \fm'$, so
everything we say for these spaces applies equally well with our original
definitions.

We continue to define $Q_\chi \coloneqq U(\fg) \big/ U(\fg)\fm_\chi$, however we may now denote it
$Q_\fl$ to emphasise the dependence on the choice of $\fl$. Finally, we define
\begin{equation}
U(\fg,e)_\fl \coloneqq (Q_\fl)^{\ad \fm'}.
\label{eq:WalgPrime}
\end{equation}
This retains an algebra structure by the same argument as for $U(\fg,e)$, and
the Kazhdan filtration is also defined for $Q_\fl$ and $U(\fg,e)_\fl$. Further,
all the maps of \cref{thm:WalgQuant} exist and satisfy the same properties when
substituting our new definitions $\fm$, $Q_\fl$ and $U(\fg,e)_\fl$,
\emph{mutatis mutandis}.
\end{rem}

\begin{lem} \label{lem:SlodowyAdjTangentIso}
For any $\Gamma$\graded\ \sltriple\ $\set{e,h,f}$, it follows that
$\fm^\bot = [\fm',e]\oplus \liez{f}$.
\end{lem}

\begin{proof}
We note first that $\liez{f} \subseteq \fm^\bot$, as
$\fm^\bot \supseteq \bigoplus_{j\le0} \fg_j \supseteq \liez{f}$, where the first inclusion follows
from \cref{GGprop5} and the second is a version of version of \cref{GGprop4} for
$f$. We also have that $[\fm',e] \subseteq \fm^\bot$, as
$\killing[\big]{\fm}{[\fm',e]} = \killing[\big]{[\fm,\fm']}{e} = 0$. In
addition, $[\fm',e] \cap \liez{f} = \set{0}$, which follows from $\fsl_2$
representation theory (see \cref{fig:sl2rep}). Finally, note that 
$\dim \fm^\bot \cap \fg_1 = \dim \fl'$, and so
$\dim \fm^\bot = \dim \fm' + \dim \fg_0 + \dim \fg_{-1}$; then
$\dim \fm' = \dim [\fm',e]$ by \cref{GGprop2}, and
$\dim \fg_0 + \dim \fg_{-1} = \dim \liez{f}$ by the version of \cref{GGprop6}
for $f$.
\end{proof}

\begin{rem}
Note that this \namecref{lem:SlodowyAdjTangentIso} shows that
$\slodowy_e = e + \liez{f}$ is a subvariety of $e + \fm^\bot$, and hence the
remaining map $\Cpoly[\big]{e + \fm^\bot} \to \Cpoly{\slodowy_e}$ in
\cref{thm:WalgQuant} can be defined as restriction of functions.
\Cref*{thm:WalgQuant} therefore follows from the following
\namecref{thm:WalgQuantIso}.
\end{rem}

\begin{thm} \label{thm:WalgQuantIso}
The map $\nu \colon \gr U(\fg,e)_\fl \to \Cpoly[\big]{\slodowy_\chi}$, defined as the
composition
$\gr U(\fg,e)_\fl \to \gr Q_\fl \to \Cpoly[\big]{e + \fm^{*,\bot}} \to
  \Cpoly[\big]{\slodowy_\chi}$, is an isomorphism.
\end{thm}

\subsubsection{Proof of \cref*{thm:WalgQuantIso}}

Recall from \cref{prop:SContract} the contracting $\CC*$\action\ on
$\slodowy_e$, denoted $\rho(t) \coloneqq t^{-2} \Ad_{\gamma(t)}$. We note that this acts not
just on $\slodowy_e$ but also on the whole of $\fg$, and that furthermore it
stabilises not only $\liez{f}$ but also $\fm^\bot$. In addition, the action of
$\CC*$ on $e + \fm^\bot$ is contracting for the same reasons as in
\cref{prop:SContract}.

Let $M'$ be the closed subgroup of $G$ such that $\Lie M' = \fm'$. We define a
$\CC*$\action\ on the variety
$M' \times \slodowy_e = M' \times \pp[\big]{e + \liez{f}}$ by the equation
\begin{equation}
t \cdot (g, e+x) = \pp[\big]{\gamma(t)g\gamma(t^{-1}), \rho(t)(e+x)}.
\label{eq:MSContract}
\end{equation}
Note that the action is still a contracting action, and in particular
\begin{equation}
\lim_{t \to \infty} t \cdot (g, e+x) = (1, e).
\label{eq:MSContractLim}
\end{equation}

\begin{lem}
\label{lem:WalgQuantAdjDecomp}
The adjoint action map $\alpha \colon M' \times \slodowy_e \to e + \fm^\bot$ is a
$\CC*$-equivariant isomorphism of affine varieties.
\end{lem}

\begin{proof}
The proof that the map is $\CC*$-equivariant is a direct computation. Note that
the map induces an isomorphism
$T_{(1,e)} \pp{M' \times \slodowy_e} \isoto T_e \pp[\big]{e + \fm^\bot}$ by
\cref{lem:SlodowyAdjTangentIso}.
The lemma then follows from the following general result: any
$\CC*$\nobreakdash-equivariant map of smooth affine varieties with contracting
$\CC*$\actions\ which induces an isomorphism on tangent spaces at the fixed
points must be an isomorphism.
\end{proof}

To complete the proof, we note that $U(\fg)$ and $Q_\fl$ are $\fm'$\modules\ via
the adjoint action, and that the map $U(\fg) \to Q_\fl$ is an $\fm'$\module\
homomorphism. The adjoint $\fm'$\action\ preserves the Kazhdan filtration, and so
the associated graded map $\gr U(\fg) \to \gr Q_\fl$ is also an $\fm'$\module\
homomorphism. Noting that
$U(\fg,e)_\fl = \pp{Q_\fl}^{\fm'} = H^0(\fm', Q_\fl)$, we can therefore
reformulate \cref{thm:WalgQuantIso} in terms of Lie algebra cohomology.

\begin{thm}
\label{thm:WalgQuantLieCoh}
The map $\nu \colon \gr U(\fg,e)_\fl \to \Cpoly[\big]{\slodowy_\chi}$ can be
decomposed as
\begin{equation*}
\gr H^0(\fm', Q_\fl) \xrightarrow{\nu_1} H^0(\fm', \gr Q_\fl)
  \xrightarrow{\nu_2} \Cpoly[\big]{\slodowy_\chi},
\end{equation*}
where $\nu_1$ and $\nu_2$ are isomorphisms. Furthermore,
$H^i (\fm', Q_\fl) = H^i (\fm', \gr Q_\fl) = 0$ for $i > 0$.
\end{thm}

\begin{proof}
Transferring \cref{lem:WalgQuantAdjDecomp} through the isomorphism $\kappa \colon \fg
\isoto \fg^*$ allows us identify the $\fm'$\module\ isomorphisms
\begin{equation*}
\gr Q_\fl \simeq \Cpoly[\big]{\chi + \fm^{*,\bot}} \simeq \Cpoly{M'} \otimes \Cpoly[\big]{\slodowy_\chi},
\end{equation*}
where the $\fm'$\action\ on $\Cpoly{M'} \otimes \Cpoly[\big]{\slodowy_\chi}$ comes by
identifying $\fm' = T_1 M'$ as the derivations of $\Cpoly{M'}$ and acting on the
first tensor component.

Recall the standard co-chain complex for calculating Lie algebra cohomology
$H^i(\fm', X)$:
\begin{equation}
0 \to X \to \fm'^* \otimes X \to \extp^2 \fm'^* \otimes X \to \dotsb \to
    \extp^k \fm'^* \otimes X \to \dotsb
\label{eq:LieAlgCohom}
\end{equation}
In the case that $X = \Cpoly{M'}$ this is just the de Rham complex
$\Omega^\bullet M'$, and so we've reduced to calculating the de Rham cohomology of
$M'$. But since $M'$ is a unipotent group and therefore an affine space, we
simply have that $H^0 \pp[\big]{\fm', \Cpoly{M'}} = \CC$ and
all higher cohomology groups vanish. It follows that
\begin{equation*}
H^0 (\fm', \gr Q_\fl) \simeq
  H^0 \pp[\big]{\fm', \Cpoly{M'}} \otimes \Cpoly[\big]{\slodowy_\chi} =
  \Cpoly[\big]{\slodowy_\chi},
\end{equation*}
and hence the map $\nu_2$ is an isomorphism and all higher cohomology groups
$H^k \pp[\big]{\fm', \gr Q_\fl}$ vanish for $k > 0$.

\begin{rem}
\label{rem:HamRedIntro}
Note that we could further remark that
\begin{equation*}
H^0 (\fm', \gr Q_\fl) \simeq H^0 \pp[\big]{\fm', \Cpoly{\chi + \fm^{*,\bot}}}
  = \Cpoly[\big]{\chi + \fm^{*,\bot}}^{M'},
\end{equation*}
that is to say the Slodowy slice is isomorphic to the quotient variety
$\slodowy_\chi \simeq \pp[\big]{\chi + \fm^{*,\bot}} \big/ M'$. This interpretation will be
very important for our future considerations.
\end{rem}

We now define a filtration on the complex \labelcref{eq:LieAlgCohom} for
$X = Q_\fl$. Note that $\fm'$ is a strictly-negatively graded subalgebra of
$\fg$ under the grading $\Gamma$, and so $\fm'^*$ is strictly-positively graded; we
write this grading as $\fm'^* = \bigoplus_{j\ge1} \fm'^*_j$. We can combine this
with the Kazhdan filtration on $Q_\fl$ to define the filtration
$F_p \pp[\Big]{\extp^k \fm'^* \otimes Q_\fl}$ as spanned by all
$(x_1 \wedge \dotsb \wedge x_k) \otimes v$ for which $x_j \in \fm'^*_{i_j}$, $v \in F_n Q_\fl$
and $n + \sum_{j = 1}^k i_j \le p$.
Taking the associated graded of this filtered complex gives us the standard
complex for computing the cohomology of $\gr Q_\fl$, as $\fm'^*$ is already a
graded Lie algebra.

Consider the spectral sequence with
$E_0^{p,q} \coloneqq F_p \pp[\Big]{\extp^{p+q} \fm'^* \otimes Q_\fl} \big/ 
              F_{p-1} \pp[\Big]{\extp^{p+q} \fm'^* \otimes Q_\fl}$.
We note that the first page of this spectral sequence is just
$E_1^{p,q} = H^{p+q} \pp[\big]{\fm', \gr_p Q_\fl}$, and since this is zero
for $p+q \neq 0$, the spectral sequence degenerates at the first page. Hence the
spectral sequence will converge to $H^{p+q} \pp[\big]{\fm', \gr_p Q_\fl}$.
However, by general results this spectral sequence, if it converges, must
converge to
$E_\infty^{p,q} = F_p H^{p+q} (\fm', Q_\fl) \big/ F_{p-1} H^{p+q} (\fm', Q_\fl)$
(see a standard reference, e.g.~\cite{McC:UserGuideSS}). Hence
it must be that $H^i (\fm', \gr Q_\fl) \simeq \gr H^i (\fm', Q_\fl)$, completing the
proof of \cref{thm:WalgQuantLieCoh}, and therefore proving
\cref{thm:WalgQuantIso,thm:WalgQuant}.
\end{proof}

\begin{cor}
The \Walgebra\ $U(\fg,e)$ does not depend up to isomorphism on the choice of
Lagrangian (or isotropic) subspace $\fl$.
\end{cor}

\begin{proof}
Let $\fl \subseteq \fl'$ be two isotropic subspaces of $\fg_{-1}$. The inclusion of $\fl$
into $\fl'$ induces a map\linebreak $Q_\fl \to Q_{\fl'}$, which itself induces a map
$U(\fg,e)_\fl \to U(\fg,e)_{\fl'}$. By \cref{thm:WalgQuantIso} this map
descends to an isomorphism at the associated graded level, and is therefore an
isomorphism itself. Choosing $\fl = \set{0}$ we get an isomorphism 
$U(\fg,e)_\fl \isoto U(\fg,e)_{\fl'}$ for any isotropic subspace $\fl'$ (and
therefore any Lagrangian subspace).
\end{proof}

\section{Hamiltonian reduction and its relation to W-algebras}
\label{sec:HamReductionWalg}

The observation made in \cref{rem:HamRedIntro} gives us an important insight
into the geometry of the Slodowy slice $\slodowy_\chi$: it can be expressed as a
quotient of a certain affine subvariety of $\fg^*$. This recalls the
construction of Hamiltonian reduction, where a Poisson variety with a
Hamiltonian group action can be reduced at a regular value of the moment map. In
fact, both $\fg^*$ and $\slodowy_\chi$ have the structure of a Poisson variety, and 
we can describe a Hamiltonian group action on $\fg^*$ so that $\slodowy_\chi$ is
the Hamiltonian reduction by this action. To develop this point of view, we need
to describe the Poisson structures on $\fg^*$ and $\slodowy_\chi$.

\subsection{The Slodowy slice as a Poisson variety}

For any Lie algebra $\fg$, the dual space $\fg^*$ has the natural structure of a
Poisson variety coming from the Lie bracket on $\fg$. To define this, we first
note that for any function $f \in \Cpoly{\fg^*}$, its differential
$\d f \in T^*(\fg^*)$ can be viewed at any point $\xi \in \fg^*$ as lying naturally in
the Lie algebra~$\fg$; this comes from the natural identification
$T^*_\xi (\fg^*) = (\fg^*)^* \simeq \fg$.

\begin{defn}
\label{def:LPBracket}
For a Lie algebra $\fg$, the ring of functions $\Cpoly{\fg^*}$ has a natural
Poisson bracket, called the \emph{Lie--Poisson} or \emph{Kostant--Kirillov}
bracket. Given $f,g \in \Cpoly{\fg^*}$, the Poisson bracket is defined as
$\poibr{f,g} \in \Cpoly{\fg^*}$, where
$\poibr{f,g}(\xi) \coloneqq \xi \pp[\big]{ \liebr{\d f_\xi, \d g_\xi} }$.
\end{defn}

\begin{note}
The Lie--Poisson bracket has an extremely simple expression. Any two elements
$x,y \in \fg$ we can be interpreted as functions in $\Cpoly{\fg^*}$ using the
evaluation map: the Poisson bracket~$\poibr{x,y}$ is then just $[x,y]$. This
can be extended to all polynomials in $\Cpoly{\fg^*}$ using the Leibniz rule.
\end{note}

\begin{thm}
The symplectic leaves of $\fg^*$ with the Lie--Poisson bracket are the co-adjoint
orbits.
\end{thm}

\begin{proof}
Choosing a co-adjoint orbit $\orbit \subseteq \fg*$, we need to show that the Poisson
bracket restricts to a symplectic form on $\orbit$. To do this, we only need to
show that for each $\alpha \in \orbit$ the induced Poisson bracket on $T_\alpha \orbit$ is
non-degenerate. 

We note that for $x \in \fg$, viewed as an element of $\Cpoly{\fg^*}$,
\begin{equation*}
\poibr{x,\cdot}(\xi) = \xi([x,\cdot]) = \xi \pp[\big]{\ad_x(\cdot)}
  = \ad^*_x(\xi) (\cdot).
\end{equation*}
So at the point $\xi \in \fg^*$, the radical of the Poisson bracket is the set of
all $x \in \fg$ which annihilate~$\xi$ under the co-adjoint action. This is the
tangent space of $G^\xi$, the stabiliser of $\xi$ in $G$. Since $\orbit_\xi \simeq G / G^\xi$
by the orbit--stabiliser theorem, the co-adjoint orbits are the maximal
subvarieties on which the Poisson bracket is non-degenerate, and are hence the
symplectic leaves.
\end{proof}

\begin{thm}{\normalfont \cite{GG:QuantSlod}}
\label{thm:SlodowyIndPoisson}
The Slodowy slice $\slodowy_\chi$ inherits a Poisson bracket from $\fg^*$.
\end{thm}

\begin{proof}
We recall from \cref{prop:SOTrans} that the Slodowy slice $\slodowy_\chi$ is
transverse to $\orbit_\chi$ at~$\chi$. By standard results
(cf.~\cite[Proposition~3.10]{Vai:LecGeoPoisson}), it suffices to show that for
any co-adjoint orbit $\orbit$ and $\xi \in \orbit \cap \slodowy_\chi$, the restriction of
the symplectic form on $T_\xi \orbit$ to $T_\xi \pp[\big]{\orbit \cap \slodowy_\chi}$
is non-degenerate.

We will work in $\fg$, using the Killing isomorphism $\kappa \colon \fg \isoto \fg^*$
to pass to $\fg^*$ when necessary. As
$T_\xi \pp[\big]{\orbit \cap \slodowy_\chi} = T_\xi \orbit \cap T_\xi \slodowy_\chi$, we will
examine the spaces $T_\xi \orbit$ and $T_\xi \slodowy_\chi$. Since $\orbit$ is a
co-adjoint orbit, its tangent space is simply the image of
$\ad^*_\xi$, which expressed in $\fg$ is $[\kappa^{-1}(\xi), \fg]$, the image of
$\ad_{\kappa^{-1}(\xi)}$. The tangent space of $\slodowy_\chi = \kappa(e + \liez{f})$ can be
seen to be $\kappa(\liez{f})$.

The symplectic form at $\xi$ can be seen to be $\omega_\xi (x,y) = \xi([x,y])$, so we need
to determine the radical of $\omega_\xi$ restricted to
$T_\xi \pp[\big]{\orbit \cap \slodowy_\chi}$. The annihilator of $T_\xi \slodowy_\chi$ in
$\fg$ is $[f,\fg]$, which can be seen from associativity of the Killing form
$\killing[\big]{\liez{f}}{[f,\fg]} = \killing[\big]{[\liez{f},f]}{\fg}$, and so
$\operatorname{rad} \omega$ is
\begin{equation*}
\kappa \pp*{ \liebr[\big]{ \kappa^{-1}(\xi), [f,\fg] } \cap \liez{f} }
\end{equation*}
Since $\kappa^{-1}(\xi) \in e + \liez{f}$, $\fsl_2$ representation theory tells us that
this space is $\set{0}$ (see \cref{fig:sl2rep}).
\end{proof}

\subsection{Slodowy slices and Hamiltonian reduction}

Since we now know that the space $X = \fg^*$ is a Poisson variety, we can ask
whether the co\nobreakdash-adjoint action of $G$ on $X$ is a Hamiltonian action. The answer
is yes, and in order to demonstrate this we will exhibit a co-moment map
$\mu^* \colon \fg \to \Cpoly{X}$, where $\fg = \Lie(G)$. Note that this is
equivalent to exhibiting a moment map $\mu \colon X \to \fg^*$ by the
equation $\inn{\mu(\xi)}{x} = \mu^*(x)(\xi)$, where $x \in \fg$, $\xi \in X$ and
$\inn{\cdot}{\cdot}$ is the canonical pairing of $\fg^*$ with $\fg$.

Note that $G$ acts on $X$ by the co-adjoint action $\Ad^*$, which induces an
action of $\fg$ on $X$, the co-adjoint action $\ad^*$. This in turn induces an
action on $\Cpoly{X} = \Sym(\fg)$, which is just the adjoint action of $\fg$ on
$\Sym(\fg)$. To produce a co-moment map, we need a function
$\mu^* \colon \fg \to \Cpoly{X}$ such that $\poibr{\mu^*(x),\cdot} = \ad_x$;
we claim that the map given by $\mu^*(x) = x$ suffices. To check this, note
that $\poibr{\mu^*(x), \cdot}(\alpha) = \inn{[x, \cdot]}{\alpha} = \inn{\ad_x(\cdot)}{\alpha}$,
and hence $\poibr{\mu^*(x), \cdot} = \ad_x$. The corresponding moment map
$\mu \colon \fg^* \to \fg^*$ is the identity map.

As a special case of the above, note that given a choice of Premet subalgebra
$\fm$, its corresponding algebraic group $M \subseteq G$ acts on $\fg^*$ by the
co-adjoint action; the resulting moment map $\mu \colon \fg^* \to \fm^*$ is
restriction of functions. Recall that in \cref{rem:HamRedIntro}, we noted that
$\slodowy_\chi \simeq \pp[\big]{\chi + \fm^{*,\bot}}/M$. There are three important
remarks to be made about this equation:
\begin{itemize}
\item that $\chi$ is a character of $\fm^*$ implies that it is fixed under the
      co-adjoint action of $M$;
\item the space $\chi + \fm^{*,\bot}$ is the pre-image of $\chi$ under the moment map
      $\mu$; and
\item the character $\chi$ is furthermore a regular value of the moment map $\mu$, as
      $\fm^* \subseteq \fm^{*,\bot}$.
\end{itemize}
This proves the following \namecref{thm:SlodowyHamRed}, which is fundamental to
our further work.

\begin{thm}
\label{thm:SlodowyHamRed}
The Slodowy slice $\slodowy_\chi \subseteq \fg^*$ can be expressed as a Hamiltonian
reduction of $\fg^*$ at the regular value $\chi \in \fm^*$ by the co-adjoint action
of $M$. Concretely,
\begin{align*}
\slodowy_\chi & \simeq \mu^{-1} (\chi) \big/ M  & \text{and} \quad\quad\quad
\Cpoly[\big]{\slodowy_\chi} & \simeq \pp*{ \Cpoly{\fg^*} \big/ I\pp[\big]{\mu^{-1} (\chi)} }^{M}.
\end{align*}
\end{thm}

As a Hamiltonian reduction of a Poisson variety, the Slodowy slice $\slodowy_\chi$
inherits a Poisson bracket from $\fg^*$, defined as follows. Consider the
following natural maps:
\begin{align*}
\iota & \colon \pp*{ \Cpoly{\fg^*} \big/ I \pp[\big] {\mu^{-1} (\chi)} } ^M
  \hookrightarrow \Cpoly{\fg^*} \big/ I \pp[\big] {\mu^{-1} (\chi)} \\
\pi     & \colon \Cpoly{\fg^*}
  \twoheadrightarrow \Cpoly{\fg^*} \big/ I \pp[\big] {\mu^{-1} (\chi)}
\end{align*}
Considering $f,g \in \Cpoly[\big]{\slodowy_\chi}$, we define the Poisson bracket
$\poibr{f,g}$ by lifting $\iota(f)$ and $\iota(g)$ to functions
$\tilde{f}, \tilde{g} \in \Cpoly{\fg^*}$, and requiring that
$\iota \pp[\big]{\poibr{f,g}} = \pi \pp[\big]{\poibr{\tilde{f},\tilde{g}}}$.
That this is well-defined follows from the conditions of Hamiltonian reduction.
This Poisson bracket agrees with that of \cref{thm:SlodowyIndPoisson}.

\subsection{Quantum Hamiltonian reduction}

Up to this point, we have been developing two related but distinct threads:
Slodowy slices and \Walgebras. In fact, these two subjects are much more
intimately related than has been presented so far, and their correspondence
forms the backbone of this thesis. In short, it can be stated that the
\Walgebra\ $U(\fg,e)$ is a quantisation of the ring of functions on the Slodowy
slice~$\slodowy_\chi$, and every statement about Slodowy slices has a corresponding
quantisation which can be applied to \Walgebras.

Recalling the Whittaker definition of \Walgebras\ from \cref{eq:WalgAdDef}, the
similarity to the Hamil\-tonian reduction equation of \cref{thm:SlodowyHamRed} can
be seen:
\begin{align*}
U(\fg,e) & = \pp[\big]{U(\fg) \big/ U(\fg)\fm_\chi}^\fm    &
\Cpoly[\big]{\slodowy_\chi} & \simeq \pp*{ \Cpoly{\fg^*} \big/ I\pp[\big]{\mu^{-1} (\chi)} }^{M}
\end{align*}
In the definition of the \Walgebra, we consider the invariants under the adjoint
action of $\fm$ in the quotient of a non-commutative filtered algebra by a left ideal.
This corresponds exactly in the Hamiltonian reduction expression to taking the
invariants under the adjoint action of $M$ in the quotient of the corresponding
associated graded spaces.

This observation is more than a curiosity: there is a very precise sense in which
we can translate concepts of Hamiltonian reduction of Poisson varieties to apply
to the \Walgebra\ $U(\fg,e)$. This will allow us to express $U(\fg,e)$ as a
\emph{quantum Hamiltonian reduction} of the universal enveloping algebra
$U(\fg)$. This can be expressed using the formalism of \emph{deformation
quantisation}, which encodes the structure of a Poisson algebra in a
non-commutative filtered algebra. Many of the structures of Poisson geometry
have quantum analogues which can be applied to the W-algebraic context,
including a quantum co-moment map which quantises the classical co-moment map.
This furnishes us with all the tools necessary to define a quantum Hamiltonian
reduction which descends gracefully to classical Hamiltonian reduction.

The technical details of this correspondence shall be addressed in detail in
\cref{chap:CatOWalg}, however for current purposes we can use this formalism as
a motivation. This point of view will allow for the application of the
techniques of Poisson geometry to \Walgebras, providing us with powerful tools
for solving problems.

\chapter{W-algebras in \texorpdfstring{\typeA*}{type A}}
\label{chap:TypeA}

Up to this point we've discussed the background for \Walgebras\ in an arbitrary
semisimple Lie algebra $\fg$. Much of what has been discussed has a much simpler
and more concrete realisation in the \emph{classical} Lie algebras, and in
particular in the Lie algebras of \typeA. We shall discuss what this background
looks like when we restrict our attention to the Lie algebras of \typeA.

In this chapter, we continue working over an algebraically closed field of
characteristic zero, which can be taken to be $\CC$. We shall fix a number
$n \in \NN_+$, and shall fix~$\fg$ to be a simple Lie algebra of \typeA{n-1}. This
is to say we shall let $\fg = \fsl_n$, the Lie algebra of $n \times n$ traceless
matrices with the commutator Lie bracket $[x,y] \coloneqq xy - yx$.

\section{Nilpotent orbits}

The nilpotent orbits in \typeA\ have a particularly simple characterisation.
Since every matrix in $\Mat_n(\CC)$ has a Jordan canonical form,
every element of $\fg$ can be conjugated to some matrix consisting of Jordan
blocks on the diagonal. In particular, for a nilpotent element $e \in \fg$, every
generalised eigenvalue is $0$, so its conjugacy class is entirely determined by
the sizes of the Jordan blocks. As a result, the nilpotent orbits in $\fg$ are
in bijection with the set of \emph{partitions} of $n$, i.e.~tuples
$\lambda = (\lambda_1, \dotsc, \lambda_k)$ satisfying $\sum_{i=1}^k \lambda_i = n$ and $\lambda_1 \ge \dotsb \ge \lambda_k$.
Partitions can be indicated by \emph{Young diagrams}.
\begin{align*}
e & = \begin{psmallmatrix} 0&1& & &  \\
                           0&0& & &  \\
                            & &0&1&  \\
                            & &0&0&  \\
                            & & & &0
      \end{psmallmatrix}  &
e & = \begin{psmallmatrix} 0&1&0& & &  \\
                           0&0&1& & &  \\
                           0&0&0& & &  \\
                            & & &0&1&  \\
                            & & &0&0&  \\
                            & & & & &0
    \end{psmallmatrix}  \\
\text{Partition:} & \qquad\, (2,2,1) &
                  & \qquad\;\, (3,2,1) \\
\text{Young diagram:} & \qquad \ydiagram{2,2,1} &
                      & \qquad \ydiagram{3,2,1}
\end{align*}

The set of nilpotent orbits has a natural partial ordering, where
$\orbit' \le \orbit$ if and only if $\orbit' \subseteq \overline{\orbit}$. Since
nilpotent orbits correspond to partitions, this also imposes a partial ordering
on the partitions of $n$. However, there already exists a well-known ordering on
the set of partitions of $n$: the \emph{dominance ordering}.

\begin{defn}
Let $\lambda = (\lambda_1, \dotsc, \lambda_k)$ and $\mu = (\mu_1, \dotsc, \mu_\ell)$ be two partitions
of $n$ satisfying $\lambda_1 \ge \dotsb \ge \lambda_k > 0$ and $\mu_1 \ge \dotsb \ge \mu_\ell > 0$. The
\emph{dominance ordering} is the partial ordering on the set of partitions of
$n$ where $\lambda \ge \mu$ if and only if
\begin{equation*}
\sum_{i=1}^j \lambda_i \ge \sum_{i=1}^j \mu_i
\end{equation*}
for every $j$ between $1$ and $\max(k,\ell)$ (here we declare $\lambda_i = \mu_j = 0$
for all $i > k$ and $j > \ell$). In this case, we say that $\lambda$ \emph{dominates}
$\mu$.
\end{defn}

\begin{thm}[Gerstenhaber, Hesselink]{\normalfont \cite[Theorem~6.2.5]{CM:Nilorbits}}
The partial ordering on orbits corresponds to the dominance ordering under the
equivalence between nilpotent orbits and partitions.
\end{thm}

Recall that in a partial ordering $\ge$, we say that $\lambda$ \emph{covers} $\mu$ if
$\lambda > \mu$ and there exists no element $\nu$ such that $\lambda > \nu > \mu$, i.e.~$\lambda$ is a
minimal element satisfying $\lambda > \mu$. We will need the following result on the
fine structure of the dominance ordering.

\begin{prop}[Gerstenhaber]{\normalfont \cite[Lemma~6.2.4]{CM:Nilorbits}}
\label{prop:OrbitCover}
Let $\lambda = (\lambda_1, \dotsc, \lambda_k)$ and $\mu$ be two partitions of $n$. The partition $\lambda$
covers $\mu$ if and only if $\mu$ can be obtained from $\lambda$ by the following
procedure. Let $i$ be an index and $j > i$ be the smallest index such that
$0 \le \lambda_j < \lambda_i - 1$, where we again declare $\lambda_j = 0$ for $j > k$. Assume that
either $\lambda_j = \lambda_i - 2$ or $\lambda_\ell = \lambda_i$ whenever $i < \ell < j$. Then the
components of $\mu$ are obtained from the components of $\lambda$ by replacing $\lambda_i$ and
$\lambda_j$ by $\lambda_i - 1$ and $\lambda_j + 1$, respectively, and re-arranging if necessary.
\end{prop}

\begin{eg}
\hspace{0pt} 
\begin{enumerate}[label=(\alph*)]
\item The partition $(3,1)$ covers the partition $(2,2)$.
\item The partition $(3,3)$ covers the partition $(3,2,1)$.
\item The partition $(3,2,1)$ covers \emph{both} the partitions $(2,2,2)$ and
      $(3,1,1,1)$.
\item The partition $(4,3,1)$ \emph{does not} cover the partition $(3,3,2)$.
      Taking $i = 1$ and $j = 3$ satisfies neither of the two conditions
      $\lambda_j = \lambda_i - 2$ and $\lambda_\ell = \lambda_i$ for all $i < \ell < j$.
      Instead, it can be seen that:
      \begin{enumerate}[label=(\roman*),topsep=0pt]
      \item The partition $(4,3,1)$ covers the partition $(4,2,2)$.
      \item The partition $(4,2,2)$ covers the partition $(3,3,2)$.
      \end{enumerate}
\end{enumerate}
\end{eg}

\begin{proof}[Proof of \cref*{prop:OrbitCover}]
By construction, $\lambda$ will cover $\mu$ for any $\lambda$ and $\mu$ related by the procedure
in the \namecref{prop:OrbitCover}, so it remains to show the converse. Assume
that $\lambda$ covers $\mu$; let $i$ be the least integer such that $\lambda_\ell > \mu_\ell$,
and $j$ be as in the \namecref{prop:OrbitCover}. Applying the procedure we
obtain a new partition $\nu$ satisfying $\lambda > \nu \ge \mu$, and hence $\nu = \mu$ since $\lambda$
covers $\mu$.
\end{proof}

Furthermore, recall from \cref{eg:TypeAsltriple} that there is an algorithm for
constructing \sltriples\ from a nilpotent $e$ in \typeA. This not only allows us
to obtain ordinary \sltriples, but combining it with the characterisation of
good gradings $\Gamma$ from the following \namecref{sec:TypeAPyramids} will allow us
to find $\Gamma$\graded\ \sltriples\ as well. Hence all the concepts whose existence
was proven in \cref{sec:Nilorbits} have concrete constructions in \typeA. This
shall be used extensively within this chapter.

\section{Pyramids}
\label{sec:TypeAPyramids}

We are interested in determining what all possible good gradings are in \typeA.
This has been accomplished by Elashvili and Kac, who have furthermore given a
complete classification of all good gradings in every simple Lie algebra
\cite{EK:ClassGG}. In the classical types, this is accomplished using a
combinatorial structure known as a \emph{pyramid}, which is closely related to a
Young diagram with additional information to encode the data of the grading.
In \typeA\ there is a particularly simple description of the pyramids.

\begin{defn}
\label{def:typeApyramid}
Let $\lambda = (\lambda_1, \lambda_2, \dotsc, \lambda_k)$ be a partition of $n$, that is to say a
sequence of strictly positive integers $\lambda_1 \ge \lambda_2 \ge \dotsb \ge \lambda_k > 0$ where
$\sum_{j=1}^k \lambda_j = n$. A \emph{pyramid of shape $\lambda$} is a collection of $n$~boxes
of size~$2 \times 2$ centred at integer points $(i,j)$ of the plane, satisfying
the following conditions:
\begin{enumerate}
\item the number of boxes in the $\ell$th row (which corresponds to having
      second co-ordinate $2\ell-1$) is $\lambda_\ell$ for each $1 \le \ell \le k$.
\item the first co-ordinates of the boxes in the $\ell$th row form an arithmetic
      sequence of difference~$2$, i.e.~$f_\ell, f_\ell + 2, \dotsc, F_\ell$.
\item the first row is centred at $0$, i.e.~$f_1 = - F_1$.
\item \label[property]{typeAPyramidCond}
      the first co-ordinates of the first and last boxes in each row form
      increasing and decreasing sequences, respectively,
      i.e.~$f_1 \le \dotsb \le f_k$ and $F_1 \ge \dotsb \ge F_k$.
\end{enumerate}

\noindent
More generally, a \emph{pyramid of size $n$} is a pyramid of shape $\lambda$ for some
partition $\lambda$ of $n$.
\end{defn}

\begin{rem}
We shall illustrate this definition by constructing all possible pyramids of
shape $\lambda = (4,3)$.
We begin by constructing a row of $\lambda_1 = 4$ boxes, each of width 2 and centred
at integer values, such that the whole row is centred at zero.
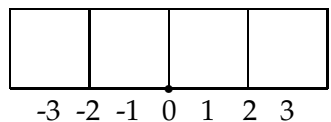
\begin{equation*}
\begin{picture}(120,42)
\put(0,12) {\line(1,0){120}} \put(0,42)  {\line(1,0){120}}
\put(0,12) {\line(0,1){30}}  \put(120,12){\line(0,1){30}}
\put(30,12){\line(0,1){30}} \put(60,12){\line(0,1){30}}
\put(90,12){\line(0,1){30}}
\put(60,12){\circle*{3}}\put(57.5,0){0}
\put(10,0){-3} \put(40,0){-1} \put(72,0){1} \put(102,0){3}
\put(25,0){-2} \put(88,0){2}
\end{picture}
\end{equation*}

We now place a second row of $\lambda_2 = 3$ boxes on top of the first row (again with
each box centred on an integer) according to the following rule: the leftmost
box of the second row cannot be further left than the leftmost box of the first
row, and the rightmost box of the second row cannot be further right than the
rightmost box of the first row. In our example, this gives us three pyramids:
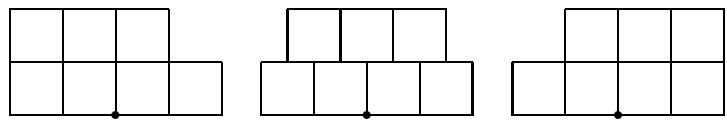
\begin{equation*}
\begin{picture}(270,40)
\put(0,0) {\line(1,0){80}} \put(0,20){\line(1,0){80}}
\put(0,40){\line(1,0){60}}
\put(0,0) {\line(0,1){40}} \put(20,0){\line(0,1){40}}
\put(40,0){\line(0,1){40}} \put(60,0){\line(0,1){40}}
\put(80,0){\line(0,1){20}}
\put(40,0){\circle*{3}}

\put(95,0)  {\line(1,0){80}} \put(95,20) {\line(1,0){80}}
\put(105,40){\line(1,0){60}}
\put(95,0)  {\line(0,1){20}} \put(115,0) {\line(0,1){20}}
\put(135,0) {\line(0,1){20}} \put(155,0) {\line(0,1){20}}
\put(175,0) {\line(0,1){20}}
\put(105,20){\line(0,1){20}} \put(125,20){\line(0,1){20}}
\put(145,20){\line(0,1){20}} \put(165,20){\line(0,1){20}}
\put(135,0) {\circle*{3}}

\put(190,0) {\line(1,0){80}} \put(190,20){\line(1,0){80}}
\put(210,40){\line(1,0){60}}
\put(190,0) {\line(0,1){20}} \put(210,0) {\line(0,1){40}}
\put(230,0) {\line(0,1){40}} \put(250,0) {\line(0,1){40}}
\put(270,0) {\line(0,1){40}}
\put(230,0){\circle*{3}}
\end{picture}
\end{equation*}
If there were further $\lambda_j$ in the partition, we would repeat this rule until
all the rows were constructed.
\end{rem}

Given a pyramid $P$ with $n$ boxes we can construct a \emph{filling} of the
pyramid by labelling each box with one of the numbers $\set{1,\dotsc,n}$ such
that there are no repeated labels. Most often we shall choose the labelling so
that it increases first up columns and then left to right.

\begin{equation}
\begin{picture}(60,60)
\put(0,0)  {\line(1,0){60}} \put(0,20) {\line(1,0){60}}
\put(10,40){\line(1,0){40}} \put(10,60){\line(1,0){40}}
\put(0,0)  {\line(0,1){20}} \put(20,0) {\line(0,1){20}}
\put(40,0) {\line(0,1){20}} \put(60,0) {\line(0,1){20}}
\put(10,20){\line(0,1){40}} \put(30,20){\line(0,1){40}}
\put(50,20){\line(0,1){40}}
 \put(30,0){\circle*{3}}
 \put(10,10){\makebox(0,0){{1}}}
 \put(30,10){\makebox(0,0){{4}}}
 \put(50,10){\makebox(0,0){{7}}}
 \put(20,30){\makebox(0,0){{2}}}
 \put(40,30){\makebox(0,0){{5}}}
 \put(20,50){\makebox(0,0){{3}}}
 \put(40,50){\makebox(0,0){{6}}}
\end{picture}
\label{eq:FilledPyramid}
\end{equation}

Let $P$ be a pyramid with a filling consisting of the numbers
$\set{1,\dotsc,n}$. If the box labelled $k$ is centred at the point
$(i,2j-1)$, then we define:
\begin{itemize}
\item $\col(k) = i$, the column number of the box labelled $k$.
\item $\row(k) = j$, the row number of the box labelled $k$.
\end{itemize}
Furthermore, we say $\ell$ is right-adjacent to $k$, denoted $k \to \ell$,
if the box labelled $\ell$ lies in the same row as and immediately adjacent to
the right of the box labelled $k$, i.e.~$\row(\ell) = \row(k)$ and
$\col(\ell) = \col(k) + 2$.

\begin{eg}
In the above pyramid \labelcref{eq:FilledPyramid}:
$\row(2) = 2$, $\col(2) = -1$, $\row(7) = 1$, $\col(7) = 2$, while
$1 \to 4$,\; $4 \to 7$,\; $2 \to 5$ and $3 \to 6$.
\end{eg}

\subsection{A bijection between pyramids and good gradings}

To any filled pyramid $P$ of size $n$, one can associate a nilpotent element
$e_P$ and a $\ZZ$\grading\ $\Gamma_P$ by the following construction. Fix a standard
basis of $\CC^n$, and let $E_{ij}$ be the matrix which maps the $i$th standard
basis vector to the $j$th and maps all other basis vectors to zero. Note that
this is the matrix with $(j,i)$-entry $1$ and all other entries $0$.
Define the nilpotent~$e_P$ and the $\ZZ$\grading~$\Gamma_P$ by declaring
\begin{align*}
e_P & \coloneqq \sum_{i \to j} E_{ij}  & \text{and}\qquad\quad
\deg E_{ij} & \coloneqq \col(j) - \col(i).
\end{align*}

One can check that the element $e_P$ is nilpotent, as
\begin{equation*}
e_P^k = \sum_{i_1 \to \dotsb \to i_k} E_{i_1 i_k},
\end{equation*}
which vanishes for large enough $k$ as every row is of finite length. A
straightforward calculation shows that $\Gamma_P$ is a Lie algebra grading.

\begin{note}
The element $e_P$ and the $\ZZ$\grading~$\Gamma_P$ can equally well be viewed in the
context of the Lie algebra $\fsl_n$ and the Lie algebra $\fgl_n$. In fact, the
following \namecref{prop:PyramidGood} holds equally well for both $\fsl_n$ and
$\fgl_n$. The proofs in the two cases are virtually identical, so we shall
proceed assuming the Lie algebra is $\fgl_n$ for computational simplicity. The
only substantive difference arises in determining the dimension of $\liez{e_P}$,
which includes the centre of $\fgl_n$, but is of dimension one less when
considering the centreless $\fsl_n$. Correcting for this discrepancy, the proof
remains identical.
\end{note}

\begin{prop}[Elashvili--Kac]{\normalfont \cite[Theorem~4.1]{EK:ClassGG}}]
\label{prop:PyramidGood}
The grading $\Gamma_P$ is good for $e_P$.
\end{prop}

\begin{proof}
It is clear by construction that $\deg e_P = 2$, so one only needs to show that
$\ad e_P \colon \fg_j \to \fg_{j+2}$ is injective for $j \le -1$ by
\cref{rem:GGprop23}. Since $\ker \ad e_P = \liez{e_P}$, to show that $\Gamma_P$ is
good for the nilpotent $e_P$, it suffices to prove the following claim.

\begin{claim}
Every element of $\fg = \fgl_n$ which commutes with $e_P$ lies in the subspace
$\bigoplus_{j\ge0} \fg_j$.
\end{claim}

It will be useful to express elements of $\fgl_n$, i.e.~endomorphisms of
$\CC^n$, as arrow diagrams in the pyramid~$P$. To the endomorphism $E_{ij}$ we
will assign the diagram consisting of the pyramid~$P$ with an arrow originating
from the box labelled $i$ and ending at the box labelled $j$. The diagram
representing the nilpotent $e_P$ is shown in \cref{fig:ggnilpotent}.


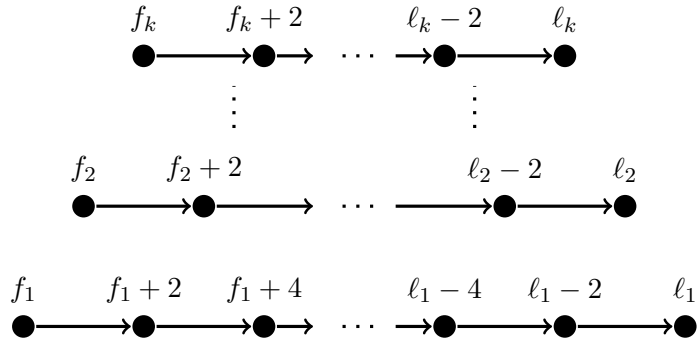
\begin{figure}[hf]

\tikzstyle{dot}=[circle,fill=black,inner sep=0,minimum size=3mm]

\centering
\begin{tikzpicture}[very thick,scale=.8]

  \path node (n1) at (0,0)  [dot,label=above:$f_1  $] {}
        node (n2) at (2,0)  [dot,label=above:$f_1+2$] {}
          edge [<-] (n1)
        node (n3) at (4,0)  [dot,label=above:$f_1+4$] {}
          edge [<-] (n2)
        node (d1) at (5,0)  {} edge [<-] (n3)
        node (d2) at (6,0)  {} edge [loosely dotted] (d1)
        node (n4) at (7,0)  [dot,label=above:$\ell_1-4$] {}
          edge [<-] (d2)
        node (n5) at (9,0)  [dot,label=above:$\ell_1-2$] {}
          edge [<-] (n4)
        node (n6) at (11,0) [dot,label=above:$\ell_1$] {}
          edge [<-] (n5);

  \path node (p1) at (1,2)  [dot,label=above:$f_2  $] {}
        node (p2) at (3,2)  [dot,label=above:$f_2+2$] {}
          edge [<-] (p1)
        node (c1) at (5,2)  {} edge [<-] (p2)
        node (c2) at (6,2)  {} edge [loosely dotted] (c1)
        node (p4) at (8,2)  [dot,label=above:$\ell_2-2$] {}
          edge [<-] (c2)
        node (p5) at (10,2) [dot,label=above:$\ell_2$] {}
          edge [<-] (p4);

  \draw [loosely dotted] (3.5,3.3) -- (3.5,4.0);
  \draw [loosely dotted] (7.5,3.3) -- (7.5,4.0);

  \path node (q1) at (2,4.5)  [dot,label=above:$f_k  $] {}
        node (q2) at (4,4.5)  [dot,label=above:$f_k+2$] {}
          edge [<-] (q1)
        node (d1) at (5,4.5)  {} edge [<-] (q2)
        node (d2) at (6,4.5)  {} edge [loosely dotted] (d1)
        node (q3) at (7,4.5)  [dot,label=above:$\ell_k-2$] {}
          edge [<-] (d2)
        node (q4) at (9,4.5)  [dot,label=above:$\ell_k$] {}
          edge [<-] (q3);

\end{tikzpicture}

\caption{The nilpotent map $e_P$, where an arrow from $f_i+2p$ to $f_j+2q$
denotes that the standard basis vector corresponding to $f_i+2p$ (i.e.~the
vector $e_m$, where $m$ is the label of the box centred at $(f_i+2p,2i-1)$) is
mapped to the basis vector corresponding to $f_j+2q$.}

\label{fig:ggnilpotent}
\end{figure}


\begin{figure}[p]
\centering

\tikzstyle{dot}=[circle,fill=black,inner sep=0,minimum size=3mm]
\tikzstyle{e}=[color=gray!50]

\begin{subfigure}[b]{\textwidth}

\begin{tikzpicture}[very thick,bend angle=45,scale=.8]

  \path node (n1) at (0,0)  [dot,label=below:$f_j  $] {}
        node (n2) at (2,0)  [dot,label=below:$f_j+2$] {}
          edge [<-,e] (n1)
        node (d1) at (3,0)  {}
          edge [<-,e] (n2)
        node (d2) at (4,0)  {} edge [loosely dotted] (d1)
        node (n3) at (5,0)  [dot,label=below:$f_j+2r$] {}
          edge [<-,e] (d2)
          edge [<-, bend right] (n1)
        node (n4) at (7,0)  [dot,label=below:$f_j+2r+2$] {}
          edge [<-,e] (n3)
          edge [<-, bend right] (n2)
        node (d3) at (8,0)  {}
          edge [<-,e] (n4)
        node (d4) at (11,0) {} edge [loosely dotted] (d3)
        node (n5) at (12,0) [dot,label=below:$\ell_j-2r-2$] {}
          edge [<-,e] (d4)
        node (n6) at (14,0) [dot,label=below:$\ell_j-2r$] {}
          edge [<-,e] (n5)
        node (d5) at (15,0) {}
          edge [<-,e] (n6)
        node (d6) at (16,0) {} edge [loosely dotted] (d5)
        node (n7) at (17,0) [dot,label=below:$\ell_j-2$] {}
          edge [<-,e] (d6)
          edge [<-, bend right] (n5)
        node (n8) at (19,0) [dot,label=below:$\ell_j$] {}
          edge [<-,e] (n7)
          edge [<-, bend right] (n6);

  \begin{scope}
    \clip (n3) rectangle +(2.3,1.2);
    \draw [->] (n3) to [bend left] +(5,0);
  \end{scope}

  \begin{scope}
    \clip (n4) rectangle +(1.5,1);
    \draw [->] (n4) to [bend left] +(5,0);
  \end{scope}

  \begin{scope}
    \clip (n5) rectangle +(-1.5,1);
    \draw [<-] (n5) to [bend right] +(-5,0);
  \end{scope}

  \begin{scope}
    \clip (n6) rectangle +(-2.3,1.2);
    \draw [<-] (n6) to [bend right] +(-5,0);
  \end{scope}

\end{tikzpicture}

\caption{The endomorphism $E_j[2r]$ on row $j$ with rightward shift $2r$ for
         $r\ge0$. Note that $e_P = \sum_j E_j[2]$.}
\label{fig:ggnilcommute1}
\end{subfigure}
\\
\begin{subfigure}[b]{\textwidth}

\begin{tikzpicture}[very thick,scale=1.2]

  \path node (n1) at (0,0)  [dot,label=below:$f_i  $] {}
        node (n2) at (2.6,0)  [dot,label=below:$f_i+2$] {}
          edge [<-,e] (n1)
        node (d1) at (3.6,0)  {}
          edge [<-,e] (n2)
        node (d2) at (4.6,0)  {} edge [loosely dotted] (d1)
        node (n3) at (5.6,0)  [dot,label=below:$f_i+\ell_j-f_j-2r-2$] {}
          edge [<-,e] (d2)
        node (n4) at (8.2,0)  [dot,label=below:$f_i+\ell_j-f_j-2r$] {}
          edge [<-,e] (n3)
        node (d3) at (9.2,0)  {}
          edge [<-,e] (n4)
        node (d4) at (11.6,0) {} edge [loosely dotted] (d3)
        node (n5) at (12.6,0) [dot,label=below:$\ell_i$] {}
          edge [<-,e] (d4);

  \path node (m1) at (.7,2)  [dot,label=above:$f_j  $] {}
        node (c1) at (1.7,2)  {}
          edge [<-,e] (m1)
        node (c2) at (2.7,2)  {} edge [loosely dotted] (c1)
        node (m2) at (3.7,2)  [dot,label=above:$f_j+2r$] {}
          edge [<-,e] (c2)
          edge [<-] (n1)
        node (m3) at (6.3,2)  [dot,label=above:$f_j+2r+2$] {}
          edge [<-,e] (m2)
          edge [<-] (n2)
        node (c3) at (7.3,2)  {}
          edge [<-,e] (m3)
        node (c4) at (8.3,2) {} edge [loosely dotted] (c3)
        node (m4) at (9.3,2)  [dot,label=above:$\ell_j-2$] {}
          edge [<-,e] (c4)
          edge [<-] (n3)
        node (m5) at (11.9,2) [dot,label=above:$\ell_j$] {}
          edge [<-,e] (m4)
          edge [<-] (n4);

\end{tikzpicture}

\caption{The endomorphism $E_i^j[2r]$ from row~$i$ to row~$j$ for $j>i$ with
rightward shift~$2r$ for $r\ge0$.} 
\label{fig:ggnilcommute2}
\end{subfigure}
\\
\begin{subfigure}[b]{\textwidth}

\begin{tikzpicture}[very thick,scale=1.2]

  \path node (n1) at (0,0)  [dot,label=below:$f_i  $] {}
        node (d1) at (1,0)  {}
          edge [<-,e] (n1)
        node (d2) at (3.4,0)  {} edge [loosely dotted] (d1)
        node (n2) at (4.4,0)  [dot,label=below:$\ell_i+f_j-\ell_j+2r$] {}
          edge [<-,e] (d2)
        node (n3) at (7,0)  [dot,label=below:$\ell_i+f_j-\ell_j+2r+2$] {}
          edge [<-,e] (n2)
        node (d3) at (8,0)  {}
          edge [<-,e] (n3)
        node (d4) at (9,0) {} edge [loosely dotted] (d3)
        node (n4) at (10,0)  [dot,label=below:$\ell_i-2$] {}
          edge [<-,e] (d4)
        node (n5) at (12.6,0)  [dot,label=below:$\ell_i$] {}
          edge [<-,e] (n4);

  \path node (m1) at (.7,2)  [dot,label=above:$f_j  $] {}
          edge [->] (n2)
        node (m2) at (3.3,2)  [dot,label=above:$f_j+2$] {}
          edge [<-,e] (m1)
          edge [->] (n3)
        node (c1) at (4.3,2)  {}
          edge [<-,e] (m2)
        node (c2) at (5.3,2)  {} edge [loosely dotted] (c1)
        node (m3) at (6.3,2)  [dot,label=above:$\ell_j-2r-2$] {}
          edge [<-,e] (c2)
          edge [->] (n4)
        node (m4) at (8.9,2)  [dot,label=above:$\ell_j-2r$] {}
          edge [<-,e] (m3)
          edge [->] (n5)
        node (c3) at (9.9,2)  {}
          edge [<-,e] (m4)
        node (c4) at (10.9,2) {} edge [loosely dotted] (c3)
        node (m5) at (11.9,2) [dot,label=above:$\ell_j$] {}
          edge [<-,e] (c4);

\end{tikzpicture}

\caption{The endomorphism $E_j^i[2r]$ from row~$j$ to row~$i$ for $j>i$ with
rightward shift~$2r$ for $r\ge0$.} 
\label{fig:ggnilcommute3}
\end{subfigure}

\begin{note}
The endomorphism $e_P$ is also shown in each diagram, consisting of the
lighter horizontal lines. It can be checked from the diagram that all such
endomorphisms commute with $e_P$.
\end{note}

\caption{Endomorphisms of $\CC^n$ commuting with $e_P$.}
\label{fig:ggnilcommute}

\end{figure}

\Cref{fig:ggnilcommute} contains a collection of endomorphisms, represented by
arrow diagrams, which commute with $e_P$. The endomorphisms in
\cref{fig:ggnilcommute1} are positively graded by construction, and
\cref{typeAPyramidCond} of \cref{def:typeApyramid}, i.e.~the pyramid condition
of the definition, ensures that the endomorphisms in
\cref{fig:ggnilcommute2,fig:ggnilcommute3} are also positively graded.
The remainder of this proof consists of showing these endomorphisms form a basis
of $\liez{e_P}$. They are linearly independent by construction, so it remains
only to show they span the kernel.

Let $\lambda = (\lambda_1,\dotsc,\lambda_k)$ be the shape of $P$, and let
$\lambda^* = (\lambda^*_1,\dotsc,\lambda^*_\ell)$ be the dual partition. Recall that the dimension
of $\liez{e_P}$ is $\sum_{i=1}^\ell \lambda^*_i{}^2$, as any endomorphism which
commutes with $e_P$ has a simultaneous Jordan basis, and hence is related to
$e_P$ by a collection of endomorphisms of the spaces generated by the $i$th
elements of the Jordan strings.

Counting the endomorphisms of \cref{fig:ggnilcommute}, we can see that there are
$\sum_{i=1}^k \lambda_i$ of those in \cref{fig:ggnilcommute1}, and
$\sum_{i=1}^k \sum_{j=i+1}^k \lambda_j = \sum_{i=1}^k (i-1)\lambda_i$ each of those in
\cref{fig:ggnilcommute2,fig:ggnilcommute3}. Hence the dimension of the space
spanned by the endomorphisms is
\begin{equation*}
\sum_{i=1}^k \lambda_i + 2 \sum_{i=1}^k (i-1)\lambda_i = n + 2\sum_{i=1}^k (i-1)\lambda_i
  = \sum_{i=1}^{\lambda_1} \lambda^*_i{}^2,
\end{equation*}
where the second equality follows by counting cubes in a tower of blocks
associated to $\lambda$ as in \cref{fig:blocktower}. Thus the endomorphisms of
\cref{fig:ggnilcommute} span, and hence form a basis for, $\liez{e_P}$. It
therefore follows that $\Gamma_P$ is a good grading for $e_P$.
\end{proof}


\begin{figure}[htbp]

\tikzstyle{top}=[yslant=0.5,xslant=-1]
\tikzstyle{left}=[yslant=-0.5]
\tikzstyle{right}=[yslant=0.5]
\tikzstyle{shade}=[fill=gray!50]

\centering
\begin{tikzpicture}[on grid,scale=.8]

\draw[top]    (-3,-3) grid +(3,3);
\draw[left]   (-3,-4) grid +(3,1);
\draw[right]  ( 0,-4) grid +(3,1);
\draw[top,shade]   (-1,-1) rectangle +(1,1);
\draw[top,shade]   (-2,-2) rectangle +(1,1);
\draw[top,shade]   (-3,-3) rectangle +(1,1);
\draw[left,shade]  (-1,-4) rectangle +(1,1);
\draw[right,shade] ( 0,-4) rectangle +(1,1);

\draw[->,very thick,left] (-4,-3.7) -- +(.8,0) node[above,midway] {$\lambda^*_\ell$};

\begin{scope}[xshift=20,yshift=10]
\draw[top]    (-7,-3) grid +(2,3);
\draw[left]   (-7,-8) grid +(3,1);
\draw[right]  (-4,-4) grid +(2,1);
\end{scope}

\begin{scope}[yshift=20]
\draw[top]    (-7,-7)  grid +(2,2);
\draw[left]   (-2,-8) grid +(2,1);
\draw[right]  ( 0,-8) grid +(2,1);
\draw[top,shade]   (-6,-6) rectangle +(1,1);
\draw[top,shade]   (-7,-7) rectangle +(1,1);
\draw[left,shade]  (-1,-8) rectangle +(1,1);
\draw[right,shade] ( 0,-8) rectangle +(1,1);
\end{scope}

\begin{scope}[xshift=-20,yshift=10]
\draw[top]    (-3,-7) grid +(3,2);
\draw[left]   ( 2,-4) grid +(2,1);
\draw[right]  ( 4,-8) grid +(3,1);

\draw[->,very thick,left] (6.7,1) -- +(0,-.8) node[auto,swap,near start] {$\lambda_k$};

\draw[<->,very thick,top] (-6,-8.5) -- +(4.5,0) node[auto,swap,midway] {$k-1$};

\end{scope}

\draw[very thick,loosely dotted] (-2,-.7) -- (-2,.8);
\draw[very thick,loosely dotted] ( 2,-.7) -- ( 2,.8);

\begin{scope}[yshift=5]

\draw[left]  (-3, 0) grid +(3,1);
\draw[top]   ( 1, 1) grid +(1,3);
\draw[right] ( 0, 0) grid +(3,1);
\draw[top]   ( 1, 1) grid +(3,1);
\draw[top,shade]   ( 1, 1) rectangle +(1,1);
\draw[left,shade]  (-1, 0) rectangle +(1,1);
\draw[right,shade] ( 0, 0) rectangle +(1,1);

\draw[left]   (-2, 2) grid +(2,2);
\draw[right]  ( 0, 2) grid +(2,2);
\draw[top]    ( 4, 4) grid +(2,2);
\draw[top,shade]   ( 4, 4) rectangle +(1,1);
\draw[top,shade]   ( 5, 5) rectangle +(1,1);
\draw[left,shade]  (-1, 2) rectangle +(1,1);
\draw[left,shade]  (-1, 3) rectangle +(1,1);
\draw[right,shade] ( 0, 2) rectangle +(1,1);
\draw[right,shade] ( 0, 3) rectangle +(1,1);

\draw[->,very thick,left] (-3,3.3) -- +(.8,0) node[above,midway] {$\lambda^*_1$};
\draw[->,very thick,left] (-3,2.2) -- +(.8,0) node[above,midway] {$\lambda^*_2$};
\draw[->,very thick,left] (-4,0.3) -- +(.8,0) node[above,midway] {$\lambda^*_3$};

\draw[->,very thick,left] (0.7,7) -- +(0,-.8) node[auto,swap,near start] {$\lambda_1$};
\draw[->,very thick,left] (1.7,7) -- +(0,-.8) node[auto,swap,near start] {$\lambda_2$};
\draw[->,very thick,left] (2.3,5) -- +(0,-.8) node[auto,near end] {$\lambda_3$};
\end{scope}

\end{tikzpicture}
\begin{equation}
n + 2 \sum_{i=1}^k (i-1)\lambda_i = \sum_{i=1}^\ell \lambda^*_i{}^2
\label{eq:blocktower}
\end{equation}

\caption{A tower of blocks associated to the partition
$\lambda = \seq{\lambda_1,\dotsc,\lambda_k}$ of $n$, with dual partition
$\lambda^* = \seq{\lambda^*_1,\dotsc,\lambda^*_\ell}$. This demonstrates the identity
\cref{eq:blocktower}: the number of blocks in the $i$th horizontal plane is
$\lambda^*_i{}^2$, the number of shaded blocks is $n$, and the number of blocks in the
vertical planes on each side is $(i-1)\lambda_i$.}

\label{fig:blocktower}
\end{figure}

\begin{eg}
For pyramid \cref{eq:FilledPyramid}, the corresponding nilpotent and good
grading are
\begin{align*}
e_P & =
\begin{pmatrix}
0 & 0 & 0 & 1 & 0 & 0 & 0 \\
0 & 0 & 0 & 0 & 1 & 0 & 0 \\
0 & 0 & 0 & 0 & 0 & 1 & 0 \\
0 & 0 & 0 & 0 & 0 & 0 & 1 \\
0 & 0 & 0 & 0 & 0 & 0 & 0 \\
0 & 0 & 0 & 0 & 0 & 0 & 0 \\
0 & 0 & 0 & 0 & 0 & 0 & 0 \\
\end{pmatrix}  &  \text{and}\quad
\Gamma_P & \colon
\begin{pmatrix}
 0 & 1  &  1 &  2 &  3 &  3 & 4 \\
-1 & 0  &  0 &  1 &  2 &  2 & 3 \\
-1 & 0  &  0 &  1 &  2 &  2 & 3 \\
-2 & -1 & -1 &  0 &  1 &  1 & 2 \\
-3 & -2 & -2 & -1 &  0 &  0 & 1 \\
-3 & -2 & -2 & -1 &  0 &  0 & 1 \\
-4 & -3 & -3 & -2 & -1 & -1 & 0 \\
\end{pmatrix},
\end{align*}
where we here denote the grading $\Gamma_P$ using a matrix whose $(i,j)$-entry is
$\deg E_{ij}$.
\end{eg}

It therefore follows that every pyramid corresponds, using some filling and
choice of standard basis, to a pair of a nilpotent element and good grading
$(e,\Gamma)$. Choosing a different filling or a different standard basis will produce
a different pair $(e',\Gamma')$, however the two will be conjugate under some change
of basis. Hence a pyramid corresponds to a conjugacy class of pairs $(e,\Gamma)$
under the adjoint action of $G$. It remains to ask whether every conjugacy class
of pairs $(e,\Gamma)$ with $\Gamma$ good for $e$ comes from some pyramid. This question
was also answered affirmatively by Elashvili and Kac.

\begin{thm}[Elashvili--Kac]{\normalfont \cite[Theorem~4.2]{EK:ClassGG}}
There is a bijection between the pyramids of size $n$ and the set of pairs
$(e, \Gamma)$ up to conjugacy, where $e \in \fgl_n$ is a nilpotent element and $\Gamma$ is a
good grading for $e$.
\begin{align*}
\set{\text{Pyramids of size $n$}} & \leftrightarrow
 \set{(e,\Gamma) \st \Gamma \text{ is a good grading for } e}/GL_n \\
P & \mapsto (e_P, \Gamma_P)
\end{align*}
This holds equally well for $\fsl_n$.
\end{thm}

\begin{rem}
The pyramids which are symmetric under reflection about the zero column are
called \emph{symmetric pyramids}, and correspond to Dynkin gradings under this
bijection. \emph{Even pyramids}, those for which $\col(j) - \col(i)$ is even for
every pair of boxes $j$ and~$i$, correspond to even good gradings. Since not
every pyramid is symmetric, this demonstrates that not every good grading is
Dynkin. Furthermore, since there always exists an even pyramid of given shape
(for example, choosing the pyramid for which all rows are right-aligned), there
always exists an even good grading for any nilpotent~$e$.
\end{rem}

\section{Hamiltonian reduction by stages}
\label{sec:ReductionStages}

In \cref{sec:HamReductionWalg}, we discussed the relationship between Slodowy
slices and Hamiltonian reduction. In particular it was shown that, for any good
grading $\Gamma$, the nilpotent $e \in \fg$ can be completed to a $\Gamma$\graded\ 
\sltriple\ $\set{e,h,f}$, and that the Slodowy slice $\slodowy_e \coloneqq e + \liez{f}$
can be expressed as a Hamiltonian reduction of the Poisson variety $\fg^*$ under
the co-adjoint action of the unipotent group~$M \subseteq G$, after identifying $\fg$
with $\fg^*$. However, $\fg^*$ is itself a Slodowy slice $\slodowy_0$, taking
the trivial grading $\Gamma$ with $\fg_0 = \fg$ and the $\Gamma$\graded\ triple
$h = e = f = 0$. Given two different Slodowy slices $\slodowy_e$ and
$\slodowy_{e'}$ with associated good gradings and \sltriples, there is therefore
a pair of reductions:
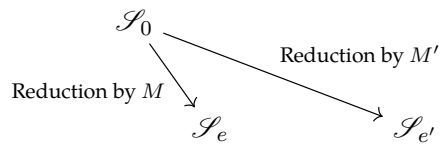
\begin{equation*}
\begin{tikzcd}[column sep=tiny]
\slodowy_0
  \ar{dr}[swap]{\text{Reduction by } M}
  \ar{drrr}{\text{Reduction by } M'} \\
 & \slodowy_e 
 & \hspace{3em} & \slodowy_{e'}
\end{tikzcd}
\end{equation*}

We shall see that this diagram can actually be completed in \typeA\ under
certain conditions on the nilpotents $e$ and~$e'$, with $\slodowy_{e'}$
expressible as a Hamiltonian reduction of $\slodowy_e$ by the action of some
unipotent group $U$:

\begin{equation}
\begin{tikzcd}[column sep=tiny]
\slodowy_0
  \ar{dr}[swap]{\text{Reduction by } M}
  \ar{drrr}{\text{Reduction by } M'} \\
 & \slodowy_e \ar[dashed]{rr}[swap]{\text{Reduction by } U}
 & \hspace{5em} & \slodowy_{e'}
\end{tikzcd}
\label{eq:SlodowyReductions}
\end{equation}

This type of procedure, decomposing a Hamiltonian reduction into a sequence of
smaller reductions, is known as \emph{Hamiltonian reduction by stages}. This is
a general technique which applies in the context of any Poisson variety, and in
particular does not require us to assume that our Poisson variety is $\fg^*$,
for $\fg$ a simple Lie algebra of \typeA. A general reference for this material
can be found in \cite{MMO:HamRed}.

\subsection{Semidirect products and reduction by stages}

Let $G$ be an algebraic group which can be expressed as a semidirect product
$G \simeq H \rtimes K$, where $H$ and $K$ are closed subgroups and $H$ is normal in
$G$. Let $X$ be a Poisson variety on which $G$ acts, and assume that the action
is Hamiltonian with equivariant moment map $\mu \colon X \to \fg^*$.
Under these circumstances, one can consider the Hamiltonian reduction of $X$ by
the action of~$G$ at a regular value $\gamma \in \fg^*$ of $\mu$; this shall be denoted
$X \qq{\gamma} G \coloneqq \mu^{-1}(\gamma) / G$.

Since $G$ decomposes as a semidirect product $H \rtimes K$, the closed $H$ also
acts on $X$ by the inclusion of $H$ into $G$. This action is also Hamiltonian,
and its moment map $\mu_H \colon X \to \fh^*$ is the composition of the moment
map~$\mu$ with the restriction of functions $j \colon \fg^* \to \fh^*$. The
subgroup~$K$ also has this property, but this is not important for our purposes.
Note that $\mu_H$ is equivariant not only with respect to the action of $H$, but
also with the respect to the action of $G$, as both $\mu$ and~$j$ are
$G$\equivariant\ maps.

Identifying $\fg^*$ with $\fh^* \times \fk^*$, the regular value $\gamma \in \fg^*$
can be decomposed as $\gamma = (\eta,\kappa)$, and further $\eta \in \fh^*$ is a regular value of
the moment map $\mu_H$. This allows one to consider the Hamiltonian reduction of
$X$ by $H$ at the regular value $\eta$, $X \qq{\eta} H \coloneqq \mu_H^{-1}(\eta) / H$.
We would like to relate $X \qq{\gamma} G$ and $X \qq{\eta} H$, and in particular
would like to establish a Hamiltonian action of $K$ on $X \qq{\eta} H$ so that the
subsequent reduction by this action produces a Poisson variety isomorphic to
$X \qq{\gamma} G$. This can be done under certain conditions on the group $K$ and the
values $\eta$ and $\kappa$.

By the definition of the semidirect product, there is an action of $K$ on $H$, and
hence an induced action of $K$ on $\fh^*$. We shall assume that the action of
$K$ stabilises $\eta \in \fh^*$. In this case, there is an induced action of $K$ on
$X \qq{\eta} H$, where the well-definedness of the action follows from the
normality of $H$, the $G$\nobreakdash-equivariance of $\mu_H$, and the fact that
$K$ stabilises $\eta$. Furthermore, this action is Hamiltonian with induced moment
map $\mu_K \colon X \qq{\eta} H \to \fk^*$ defined by $\mu_K([x]) = \mu(x)$. We
finally assume that $\kappa \in \fk^*$ is a regular value of $\mu_K$.

\begin{thm}
\label{thm:2StageReduction}
Let $X$ be a Poisson variety with a Hamiltonian action by the algebraic group
$G \simeq H \rtimes K$ satisfying the above hypotheses. There is a Poisson
isomorphism between the space $X \qq{\gamma} G$ and the two-stage reduction
$\pp{X \qq{\eta} H} \qq{\kappa} K$.
\end{thm}

\begin{note}
There are many versions of this \namecref{thm:2StageReduction} with varying sets
of hypotheses. In particular, it holds in much greater generality, as in
\cite[Theorem~5.2.9]{MMO:HamRed}. Though presented there for symplectic
varieties, the proof for Poisson varieties follows identically making the
necessary changes. Since we're making the simplifying assumptions that $G$ is a
semidirect product $H \rtimes K$ for which $K$ stabilises $\eta$, most of the
details of the theorem simplify; in particular the \emph{stages hypothesis} is
automatically satisfied and the moment map of the action of $K$ on $X \qq{\eta} H$
is obtained by lifting to $X$ and applying $\mu$.
\end{note}

\section{Reduction by stages for Slodowy slices}
\label{sec:SlodowyReduction}

As has been previously established in \cref{thm:SlodowyHamRed}, the Slodowy
slices $\slodowy_\chi$ can be expressed as Hamiltonian reductions of the dual Lie
algebra $\fg^*$. This can and will be equivalently stated in the Lie algebra
itself, rather than its dual, by applying the Killing isomorphism: $\slodowy_e$
can be expressed as the Hamiltonian reduction of $\fg$. Since $\fg$ is the Slodowy
slice through the zero nilpotent, one might ask whether different Slodowy slices
can be expressed as Hamiltonian reductions of other Slodowy slices. The
objective of this \namecref{sec:SlodowyReduction} is to provide a conjecture
regarding under which conditions this can be done, and to provide a construction
to accomplish this.

\begin{obj}
\label{thm:SlodowyReductionStages}
Let $\fg$ be a Lie algebra of \typeA, and $e_1$ and $e_2$ be two nilpotent
elements of $\fg$ such that~$\orbit_{e_1} < \orbit_{e_2}$, with Slodowy slices
$\slodowy_{e_1}$ and $\slodowy_{e_2}$, respectively. Then we would like to
exhibit an algebraic group $K$ with a Hamiltonian action on
$\slodowy_{e_1}$, along with a regular value $\kappa$ of the moment map
$\mu \colon \slodowy_{e_1} \to \fk^*$, such that $\slodowy_{e_2}$ can be expressed
as a Hamiltonian reduction of $\slodowy_{e_1}$,
i.e.~$\slodowy_{e_2} \simeq \slodowy_{e_1} \qq{\kappa} K$.
\end{obj}

This will produce a collection of commuting reductions of Slodowy slices for
every edge in the Hasse diagram of the partial ordering on nilpotent orbits, as
in \cref{fig:sl6Hasse}.
Since every pair of nilpotent orbits $\orbit_1 < \orbit_2$ can be filled in by a
sequence of covering relations $\orbit_1 < \dotsb < \orbit_2$, it will suffice
to construct the reductions assuming that $\orbit_{e_1}$ is covered by
$\orbit_{e_2}$. \Cref{thm:2StageReduction} provides the tools needed to
construct this. For any pair of nilpotent orbits with $\orbit_1$ covering
$\orbit_2$, we will choose
\begin{itemize}
\item nilpotent elements $e_1 \in \orbit_1$ and $e_2 \in \orbit_2$ with Killing
      duals $\chi_1$ and $\chi_2$, respectively;
\item a good grading $\Gamma_1$ for $e_1$ with Premet subalgebra $\fm_1$ and
      algebraic group $M_1$; and
\item a subalgebra $\fm_2 \supseteq \fm_1$ with corresponding algebraic group $M_2$
\end{itemize}
which satisfy the following conditions:
\begin{enumerate}[label=\textsf{SR\arabic*}.,ref=\textsf{SR\arabic*},labelindent=\parindent,leftmargin=*]
\item \label[condition]{cond:SubalgDecomp}
      the subalgebra $\fm_2$ decomposes as a semidirect product
      $\fm_2 = \fm_1 \rtimes \fk$;
\item \label[condition]{cond:SubalgCharacter}
      the functional $\chi_2$ restricts to a character of $\fm_2$ and decomposes as
      $(\chi_1, \kappa)$ in the above decomposition;
\item \label[condition]{cond:KStabilise}
      the subalgebra $\fk$ annihilates $\chi_1$; and
\item \label[condition]{cond:KRegular}
      the value $\kappa \in \fk^*$ is a regular value of the moment map
      $\mu_K \colon \fg^* \qq{\chi_1} M_1 \to \fk^*$ of the action of $K$, the
      algebraic group corresponding to $\fk$, on the Hamiltonian reduction
      $\fg^* \qq{\chi_1} M_1$, and the action of $K$ on $\mu_K^{-1}(\kappa)$ is free and
      proper.
\end{enumerate}

With these choices it follows that $M_2$ decomposes as a semidirect product
$M_2 = M_1 \rtimes K$, and there is a corresponding reduction by stages
construction
\begin{equation*}
X = \fg^* \qq{\chi_2} M_2 \simeq \pp{\fg^* \qq{\chi_1} M_1} \qq{\kappa} K \simeq \slodowy_{\chi_1} \qq{\kappa} K
\end{equation*}
provided by \cref{thm:SlodowyHamRed,thm:2StageReduction}. We will therefore
provide a construction satisfying these conditions, and conjecture that the
Poisson variety $X$ obtained is isomorphic to the Slodowy slice
$\slodowy_{\chi_2}$. Provided this conjecture holds, this will accomplish
\cref{thm:SlodowyReductionStages}.

\begin{prop}
\label{prop:Premet2Reduction}
If $\fm_2$ is a Premet subalgebra for a good grading of $e_2$ then
$X \simeq \slodowy_{\chi_2}$. Furthermore, \cref*{cond:SubalgDecomp,cond:SubalgCharacter}
imply \cref*{cond:KStabilise,cond:KRegular}. \Cref*{thm:SlodowyReductionStages}
therefore follows under these conditions.
\end{prop}

\begin{proof}
Since Slodowy slices can be expressed as Hamiltonian reductions of $\fg^*$ by
Premet subgroups, that $\slodowy_{\chi_2} \simeq \fg^* \qq{\chi_2} M_2$ follows directly
from \cref{thm:SlodowyHamRed}.

Recall that by \cref{prop:PremProp}, $\chi_1$ and $\chi_2$ are characters of $\fm_1$
and $\fm_2$, respectively. Therefore $\chi_2$ vanishes on $\liebr{\fm_2,\fm_2}$,
and so is annihilated by $\fm_2$ in general, and $\fk \subseteq \fm_2$ in particular,
establishing \labelcref*{cond:KStabilise}.

To prove \labelcref*{cond:KRegular}, note that
$\mu_K \colon \slodowy_{\chi_1} \to \fk^*$ has tangent map
\begin{equation*}
T_{[\xi]} \mu_K \colon T_{[\xi]} \pp[\big]{ \fg^* \qq{\chi_1} M_1 }
  \lra T_{\res*{\xi}{\fk}} \fk^*.
\end{equation*}
This map, expressed more concretely, is the restriction of functions
$\operatorname{res} \colon \fm_1^{*,\bot} / \fm_1^* \to \fk^*$. This is
well-defined as $\fk \subseteq \fm_2$, and therefore
$\fk^* \subseteq \fm_2^{*,\bot} \subseteq \fm_1^{*,\bot}$, and surjective as 
$\fk \cap \fm_1 = \set{0}$. That the action is locally free and proper follows from
the freeness and properness of the action of $M_2$.
\end{proof}

\begin{note}
Keeping in mind the parallels between Slodowy slices and \Walgebras, these
constructions should be applicable not only to Hamiltonian reduction of Slodowy
slices, but also to quantum Hamiltonian reduction of \Walgebras.  We shall
therefore work with Lie algebras instead of algebraic groups, and the
corresponding statements for groups follow directly by exponentiation.
\end{note}

\subsection{A construction in examples}

Before providing the general construction satisfying 
\crefrange*{cond:SubalgDecomp}{cond:KRegular}, it will be  useful to see the
construction in a number of examples. The main features can be seen in some
concrete cases, and the general construction follows in a straightforward manner
from these examples.

\subsubsection{Reducing from a subregular nilpotent in $\fsl_3$}
\label{eg:sl3subreg}
Let $\fg = \fsl_3$, and let $\orbit_1$ and $\orbit_2$ be the subregular and
regular nilpotent orbits, respectively. First, we construct a right-aligned
pyramid for the subregular element in Jordan canonical form. Note that a
right-aligned pyramid is automatically even, and so uniquely specifies a Premet
subalgebra.
\begin{align*}
P_1 & = \raisebox{0.2em}{\large \ytableaushort{\none3,12}} &
e_1 & =
\begin{pmatrix}
0 & 1 & 0 \\
0 & 0 & 0 \\
0 & 0 & 0
\end{pmatrix}
&
\fm_1 & =
\begin{pmatrix}
0 & 0 & 0 \\
* & 0 & 0 \\
* & 0 & 0
\end{pmatrix}
\intertext{If box~3 were moved to the bottom row, the resulting pyramid would be
the standard pyramid for the regular nilpotent element in $\fg$ in Jordan
canonical form, $e_2$.}
P_2 & = \raisebox{-0.5em}{\large \ytableaushort{123}} &
e_2 & =
\begin{pmatrix}
0 & 1 & 0 \\
0 & 0 & 1 \\
0 & 0 & 0
\end{pmatrix}
&
\fm_2 & =
\begin{pmatrix}
0 & 0 & 0 \\
* & 0 & 0 \\
* & * & 0
\end{pmatrix}
\end{align*}

Since $\fm_2$ is a Premet subalgebra, it suffices to check 
\cref*{cond:SubalgDecomp,cond:SubalgCharacter}. One can check that
$\fm_1$ is an ideal of $\fm_2$, and the complementary subalgebra $\fk$ can be
chosen to be
\begin{equation*}
\fk = 
  \begin{pmatrix}
  0 & 0 & 0 \\
  0 & 0 & 0 \\
  0 & * & 0
  \end{pmatrix},
\end{equation*}
verifying \cref*{cond:SubalgDecomp}. \Cref*{cond:SubalgCharacter} can be checked
by observing that $\res{\chi_2}{\fm_1} = \chi_1$.

\subsubsection{Reducing from the middle nilpotent in $\fsl_4$}
\label{eg:sl4middle}

Let $\fg = \fsl_4$ and let $\orbit_1$ be the middle nilpotent orbit, i.e.~the
orbit consisting of all nilpotent elements of type~(2,2). This covers the subregular
nilpotent orbit $\orbit_2$. We again construct a right-aligned pyramid for a
middle nilpotent, specifying a unique Premet subalgebra.
\begin{align*}
P_1 & = \raisebox{0.35em}{\large \ytableaushort{24,13}} &
e_1 & = E_{13} + E_{24} = 
\begin{pmatrix}
0 & 0 & 1 & 0 \\
0 & 0 & 0 & 1 \\
0 & 0 & 0 & 0 \\
0 & 0 & 0 & 0
\end{pmatrix}
&
\fm_1 & =
\begin{pmatrix}
0 & 0 & 0 & 0 \\
0 & 0 & 0 & 0 \\
* & * & 0 & 0 \\
* & * & 0 & 0
\end{pmatrix}
\end{align*}
By sliding box 4 to the bottom row we would obtain a subregular nilpotent of
type~(3,1), however instead of choosing this pyramid we shall construct a
subregular nilpotent and good grading as follows:
\begin{itemize}
\item Let the subregular nilpotent element $e_2$ be the sum of the original
  nilpotent $e_1$ and all matrices $E_{ij}$ where $i$ is a box in the first row
  and $j$ is the box immediately above it in the second row,
  i.e.~$\row(i) = 1$, $\row(j) = 2$ and $\col(i) = \col(j)$.
  \begin{equation*}
    e_2 = e_1 + E_{12} + E_{34} = 
      \begin{pmatrix}
      0 & 1 & 1 & 0 \\
      0 & 0 & 0 & 1 \\
      0 & 0 & 0 & 1 \\
      0 & 0 & 0 & 0 \\
      \end{pmatrix}.
  \end{equation*}
\item Let $\fm_2$ be the Lie algebra generated by $\fm_1$ with the additional
  generator $E_{21} + E_{43}$, that is the sum of all $E_{ij}$ such that
  $\row(i) = 2$, $\row(j) = 1$ and $\col(i) = \col(j)$.
\begin{equation*}
\fm_2 = \set*{
\begin{pmatrix}
0 & 0 & 0 & 0 \\
a & 0 & 0 & 0 \\
* & * & 0 & 0 \\
* & * & a & 0
\end{pmatrix}
\st a \in \CC }.
\end{equation*}
\end{itemize}

A direct computation shows that \crefrange*{cond:SubalgDecomp}{cond:KRegular}
are satisfied, with the choice of complementary subalgebra
\begin{equation*}
\fk = \gen[\big]{E_{21} + E_{43}} =
\set*{
\begin{pmatrix}
0 & 0 & 0 & 0 \\
a & 0 & 0 & 0 \\
0 & 0 & 0 & 0 \\
0 & 0 & a & 0
\end{pmatrix}
\st a \in \CC }.
\end{equation*}

\subsubsection{Reducing from the zero nilpotent in $\fsl_3$}
\label{eq:sl3zero}
Let $\fg = \fsl_3$ and let $\orbit_1 = \set{0}$ be the zero orbit and $\orbit_2$
be the minimal orbit (which is also the subregular orbit).
\begin{align*}
P_1 & = \raisebox{1.0em}{\large \ytableaushort{3,2,1}} &
e_1 & =
\begin{pmatrix}
0 & 0 & 0 \\
0 & 0 & 0 \\
0 & 0 & 0
\end{pmatrix}
&
\fm_1 & =
\begin{pmatrix}
0 & 0 & 0 \\
0 & 0 & 0 \\
0 & 0 & 0
\end{pmatrix}
\end{align*}
If box~3 were moved to the bottom row, the result would be a pyramid
corresponding to a minimal nilpotent element. The nilpotent $e_2$ and subalgebra
$\fm_2$ shall be chosen as follows:
\begin{itemize}
\item Let $e_2$ be the sum of the original nilpotent $e_1$ and all matrices
  $E_{ij}$ where $i$ is in the first row and $j$ is the box immediately above it
  in the third row, i.e.~$\row(i) = 1$, $\row(j) = 3$ and $\col(i) = \col(j)$.
  \begin{equation*}
    e_2 = e_1 + E_{13} = 
      \begin{pmatrix}
      0 & 0 & 1 \\
      0 & 0 & 0 \\
      0 & 0 & 0
      \end{pmatrix}.
  \end{equation*}
\item Let $\fm_2$ be the Lie algebra generated by $\fm_1$ with the additional
  generators $E_{21} + E_{32}$ and~$E_{31}$. These are the generators
  corresponding to the integers $k = 1,2$, where the generator corresponding to
  $k$ consists of the sum of all $E_{ij}$ such that
  $\row(i) \le 3$, $\row(j) \ge 1$, $\col(i) = \col(j)$ and $\row(i)-\row(j) = k$.
\begin{equation*}
\fm_2 = \gen[\big]{E_{21} + E_{32}, E_{31}} =
\set*{
\begin{pmatrix}
0 & 0 & 0 \\
a & 0 & 0 \\
b & a & 0 \\
\end{pmatrix}
\st a,b \in \CC }.
\end{equation*}
\end{itemize}

\subsection{The general construction}
\label{sec:GeneralConstruction}

Let $\fg = \fsl_n$ be the Lie algebra of \typeA{n-1}. The conjugacy classes of
nilpotent elements correspond to Jordan types, and are hence indexed by
partitions of $n$. Consider a pair of nilpotent conjugacy classes indexed by
$\lambda = (\lambda_1, \lambda_2, \dotsc, \lambda_m)$ and $\mu = (\mu_1, \dotsc, \mu_{m+1})$ where $\lambda$
covers~$\mu$, i.e.~$\lambda>\mu$ and no partition lies intermediate of the two. Let $i$
and~$j$ be the integers used for obtaining $\mu$ from $\lambda$ as in
\cref{prop:OrbitCover}.

Construct a right-aligned pyramid for $\mu$ in the usual way, numbering the
boxes from bottom to top and left to right. This determines a nilpotent
$e_1 \in \orbit_\mu$ and $\fm_1 \subseteq \fg$, and hence determines a Slodowy
slice $\slodowy_e$. What remains is to choose $e_2 \in \orbit_{\lambda}$ and
$\fm_2$ which satisfy \crefrange{cond:SubalgDecomp}{cond:KRegular}.

\begin{thm} \label{thm:typeAred}
In the above circumstances, let $e_2$ and $\fm_2 \subseteq \fg$ be as follows:
\begin{align*}
e_2 & = e_1 + \mspace{-15mu}
  \sum_{\substack{
    \row(k) = i, \row(\ell) = j \\
    \col(k) = \col(\ell)}}
      \mspace{-35mu} E_{k\ell}
\qquad\quad & \text{and} \qquad\quad \fm_2 & = \fm_1 + \gen[\big]{ E_m }_{m=1}^{j-i},
& \text{where} \quad
E_m & = \mspace{-30mu}
  \sum_{\substack{
  i \leq \row(k) < \row(\ell) \leq j \\
  \row(\ell) - \row(k) = m           \\
  \col(k) = \col(\ell)}}
    \mspace{-30mu} E_{\ell k}.
\end{align*}
Further, let $\chi_k = \killing{e_k}{\cdot}$ for $k=1,2$ and
$\fk = \gen[\big]{E_m}_{m=1}^{j-i}$. Then $e_2$ is a nilpotent element of Jordan
type~$\lambda$ in~$\orbit_\lambda$, and \crefrange{cond:SubalgDecomp}{cond:KRegular} hold.
\end{thm}

\begin{rem}
\label{rem:NilpGens}
Note that the nilpotent $e_2$ and the generators $E_m$ can be rewritten using
the generators of $\liez{e_1}$ as in \cref{fig:ggnilcommute3}:
\begin{align}
e_2 & = e_1 + E^j_i[0] &
E_m & = \sum_{i \le k \le j-m} E^k_{k+m} [0].
\label{eq:GenNilcommute}
\end{align}
In particular, it can be seen from the figure that the generators $E_m$ commute
amongst themselves, and hence $\fk$ is an abelian Lie algebra. To see that
$\fm_2$ is indeed closed under the Lie bracket, observe that, using the good
grading from the pyramid for $\mu$,
\begin{equation*}
\liebr{\fm_2, \fm_2} = \liebr{\fm_2, \fm_1} \subseteq
  \liebr[\big]{\bigoplus_{k\le0} \fg_k, \bigoplus_{\ell\le-2} \fg_\ell} \subseteq
  \bigoplus_{k\le-2} \fg_k = \fm_1 \subseteq \fm_2.
\end{equation*}
\end{rem}

\begin{proof}
To prove that $e_2$ has the correct Jordan type, it suffices to exhibit a Jordan
basis. Note that a Jordan basis can be read off the rows of the pyramid,
proceeding from left to right. Let the Jordan basis for $e_1$ in row $i$ be
given by
\begin{equation*}
e_{i_1} \mapsto e_{i_2} \mapsto \dotsb \mapsto e_{i_{\mu_i}} \mapsto 0.
\end{equation*}
The Jordan basis for $e_2$ is identical to that of $e_1$ except for those
strings corresponding to rows $i$ and~$j$. Specifically, we have the two strings
\begin{gather*}
\pp[\big]{e_{i_1}} \mapsto \dotsb \mapsto \pp[\big]{e_{i_{\mu_i-\mu_j}}}
  \mapsto \pp[\big]{e_{i_{\mu_i-\mu_j+1}} + e_{j_1}} \mapsto \dotsb 
  \mapsto \pp[\big]{e_{i_{\mu_i}} + (\mu_j-1)e_{j_{\mu_j-1}}} \mapsto \pp[\big]{\mu_je_{j_{\mu_j}}} \mapsto 0, \\
\pp[\big]{e_{i_{\mu_i-\mu_j+1}} - (\mu_j-1)e_{j_1}} \mapsto \pp[\big]{e_{i_{\mu_i-\mu_j+2}} - (\mu_j-2)e_{j_2}}
  \mapsto \dotsb \mapsto \pp[\big]{e_{i_{\mu_i}} - e_{j_{\mu_j-1}}} \mapsto 0,
\end{gather*}
of lengths $\mu_i+1 = \lambda_i$ and $\mu_j-1 = \lambda_j$, respectively.

\begin{eg}
Consider partitions $\mu = (3,2)$ and $\lambda = (4,1)$. Choosing a right-aligned
pyramid for $\mu$, one obtains the pyramid and Jordan basis of type $\mu$
\begin{align*}
& \smash{\raisebox{-.75em}{\ytableaushort{\none35,124}}}  &
   e_1 \mapsto e_2 \mapsto e_4 & \mapsto 0 \\
&&  e_3 \mapsto e_5 & \mapsto 0.
\end{align*}
The corresponding Jordan basis of type $\lambda$ is
\begin{gather*}
(e_1) \mapsto (e_2) \mapsto (e_4 + e_3) \mapsto (2e_5) \mapsto 0 \\
(e_4 - e_3) \mapsto 0.
\end{gather*}
\end{eg}

\Cref{cond:SubalgDecomp}, that $\fm_2 \simeq \fm_1 \rtimes \fk$, follows by showing
that $\fm_1$ is a Lie ideal of $\fm_2$ and $\fk$ is a complementary subalgebra,
both of which are shown in \cref{rem:NilpGens}.

To check \cref{cond:SubalgCharacter}, it can be seen that
$\res{\chi_2}{\fm_1} = \chi_1$ by construction. To confirm that $\chi_2$ is a character
of $\fm_2$, note that $\liebr{\fm_2,\fm_2} = \liebr{\fm_1 + \fk, \fm_1}$ by
\cref{rem:NilpGens}. However, since
$\chi_2 \pp[\big]{\liebr{\fm_1,\fm_1}} = \chi_1 \pp[\big]{\liebr{\fm_1,\fm_1}} = 0$,
it remains only to check that $\chi_2(\liebr{\fk, \fm_1}) = 0$. We shall check this
on a basis
$\set[\big]{\liebr{E_m, E_{\ell k}} \st
  \col(k) < \col(\ell), 1 \leq m \leq j-i}$.
First, note that
\begin{equation*}
\chi_2 \pp[\big]{ \liebr{E_m, E_{\ell k}} }
  = \inn[\big]{e_2}{\liebr{E_m, E_{\ell k}}}
  = \inn[\big]{\liebr{e_2, E_m}}{E_{\ell k}}.
\end{equation*}
This vanishes, as \cref{eq:GenNilcommute} implies that
\begin{equation*}
\liebr{e_2, E_m} = \liebr[\Big]{e_1 + E^j_i[0], \sum_{i\le k\le j-m} E^k_{k+m}[0]} = 0.
\end{equation*}
This further establishes the stronger claim that $\fk$ annihilates $\chi_2$, and
also $\chi_1$: hence \cref{cond:KStabilise} also holds. \Cref{cond:KRegular}
follows by the same argument as in \cref{prop:Premet2Reduction}.\qedhere
\end{proof}

\begin{conj}
\label{conj:typeAred}
For nilpotents $e_1, e_2 \in \fg$ and subalgebras $\fm_1, \fm_2 \subseteq \fg$ as defined
in \cref{thm:typeAred}, the reduced space $\fg^* \qq{\chi_2} M_2$ is isomorphic to
the Slodowy slice $\slodowy_{\chi_2}$ as a Poisson variety.
\end{conj}

\begin{rem}
This conjecture is a special case of a more general conjecture due to Premet,
based on his work in \cite{Pre:TransSlice}
(cf.~\cite[Question~1]{Sad:PairesAdmissibles}). Specifically, Premet conjectures
that for any ad-nilpotent subalgebra $\fm \subseteq \fg$ satisfying the conditions:
\begin{enumerate}
\item \label[condition]{cond:PreConjChar}
      $\chi(\fm,\fm) = 0$, i.e.~$\chi$ is a character of $\fm$;
\item \label[condition]{cond:PreConjCentre}
      $\fm \cap \liez[\fg]{e} = \set{0}$; and
\item \label[condition]{cond:PreConjDim}
      $\dim \fm = (\dim G \cdot e) / 2$;
\end{enumerate}
the \Walgebras\ $U(\fg,e)$, and therefore the classical reduced spaces
$\fg \qq{\chi} M$, are all isomorphic.

Concretely, the algebra $\fm_2$ defined in \cref{thm:typeAred} satisfies
\cref*{cond:PreConjChar,cond:PreConjDim} by construction, and it can be checked
that it satisfies \cref*{cond:PreConjCentre} by modifying the diagrams in
\cref{fig:ggnilcommute}.
Premet's conjecture would therefore imply that the
Hamiltonian reduction by $\fm_2$ at $\chi_2$ is isomorphic to the Slodowy slice
$\slodowy_{\chi_2}$

In fact, Premet has proven this conjecture in the case that the base field is of
characteristic~$p$ \cite{Pre:TransSlice}. Hence \cref*{conj:typeAred}, and the
construction presented here in general, is actually a theorem when working over
a field of non-zero characteristic. 
\end{rem}

\begin{prop}
\label{prop:typeAsubreg}
\Cref{conj:typeAred} holds for $e_1$ a subregular nilpotent and $e_2$ a regular
nilpotent.
\end{prop}

\begin{proof}
The subalgebra $\fm_2$ constructed is simply the maximal nilpotent subalgebra of
lower-triangular matrices $\fn_-$. This is a Premet subalgebra for $e_2$.
\end{proof}

\begin{rem}
The construction detailed in this \namecref{sec:GeneralConstruction} can be
modified slightly to give a stronger version of \cref*{prop:typeAsubreg}.
Instead of choosing a right-aligned pyramid of shape $\mu$, one can choose a
pyramid which is right-aligned but for a leftward shift of 1 row~$i$ and another
leftward shift of 1 at row~$j+1$. This necessitates a choice of Lagrangian
$\fl \subseteq \fg_{-1}$; this choice can be made so that the resulting Premet
subalgebra can be extended to a Premet subalgebra for a pyramid of shape $\lambda$
which is right-aligned but for a leftward shift of 1 at row~$i+1$ and another
leftward shift of 2 at row~$j$.
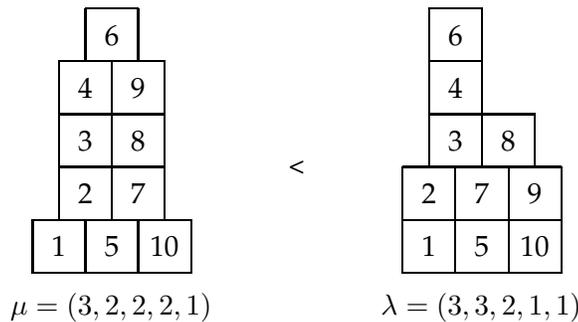
\begin{equation*}
\begin{picture}(200,113)(0,-13)
\put(0,0)  {\line(1,0){60}} \put(0,20) {\line(1,0){60}}
\put(10,40){\line(1,0){40}} \put(10,60){\line(1,0){40}}
\put(10,80){\line(1,0){40}} \put(20,100){\line(1,0){20}}
\put(0,0)  {\line(0,1){20}} \put(20,0) {\line(0,1){20}}
\put(40,0) {\line(0,1){20}} \put(60,0) {\line(0,1){20}}
\put(10,20){\line(0,1){60}} \put(30,20){\line(0,1){60}}
\put(50,20){\line(0,1){60}}
\put(20,80){\line(0,1){20}} \put(40,80){\line(0,1){20}}
 \put(10,10){\makebox(0,0){{1}}}
 \put(30,10){\makebox(0,0){{5}}}
 \put(50,10){\makebox(0,0){{10}}}
 \put(20,30){\makebox(0,0){{2}}}
 \put(40,30){\makebox(0,0){{7}}}
 \put(20,50){\makebox(0,0){{3}}}
 \put(40,50){\makebox(0,0){{8}}}
 \put(20,70){\makebox(0,0){{4}}}
 \put(40,70){\makebox(0,0){{9}}}
 \put(30,90){\makebox(0,0){{6}}}
 \put(30,-13){\makebox(0,0){{$\mu = (3,2,2,2,1)$}}}

\put(100,40){\makebox(0,0){{<}}}

\put(140,0)  {\line(1,0){60}} \put(140,20) {\line(1,0){60}}
\put(140,40) {\line(1,0){60}} \put(150,60){\line(1,0){40}}
\put(150,80){\line(1,0){20}} \put(150,100){\line(1,0){20}}
\put(140,0)  {\line(0,1){40}} \put(160,0) {\line(0,1){40}}
\put(180,0) {\line(0,1){40}} \put(200,0) {\line(0,1){40}}
\put(150,40){\line(0,1){60}} \put(170,40){\line(0,1){60}}
\put(190,40){\line(0,1){20}}
 \put(150,10){\makebox(0,0){{1}}}
 \put(170,10){\makebox(0,0){{5}}}
 \put(190,10){\makebox(0,0){{10}}}
 \put(150,30){\makebox(0,0){{2}}}
 \put(170,30){\makebox(0,0){{7}}}
 \put(190,30){\makebox(0,0){{9}}}
 \put(160,50){\makebox(0,0){{3}}}
 \put(180,50){\makebox(0,0){{8}}}
 \put(160,70){\makebox(0,0){{4}}}
 \put(160,90){\makebox(0,0){{6}}}
 \put(170,-13){\makebox(0,0){{$\lambda = (3,3,2,1,1)$}}}
\end{picture}
\label{eq:AlternateConstruction}
\end{equation*}

For this new pyramid and compatible choice of Lagrangian, \cref*{thm:typeAred}
remains true. Furthermore, \cref*{prop:typeAsubreg} and its proof hold not only
for $e_1$ a subregular nilpotent and $e_2$ a regular nilpotent, but more
generally for any pair of nilpotent elements $e_1$ and $e_2$ of Jordan types
$\mu = (\mu_1,\dotsc,\mu_k,1)$ and $\lambda = (\mu_1,\dotsc,\mu_k+1)$, respectively.
\end{rem}

\begin{eg}
Let $\fg = \fsl_4$, $\orbit_1$ be the middle nilpotent orbit and $\orbit_2$ be
the subregular nilpotent orbit as in page~\pageref{eg:sl4middle}. The Slodowy slice
$\slodowy_{\chi_2}$ and a presentation for the reduced space $\fg^* \qq{\chi_2} M_2$
thus obtained are:
\begin{align*}
\slodowy_{\chi_2} & = \set*{
\begin{pmatrix}
a       & 1      & 0 & 0 \\
b-3a^2  & a      & 1 & 0 \\
c+20a^3 & b-3a^2 & a & d \\
f       & 0      & 0 & -3a
\end{pmatrix}
\st a,b,c,d,f \in \CC}, &
\Cpoly[\big]{\slodowy_{\chi_2}} & = \Cpoly{a,b,c,d,f} \\
\fg^* \qq{\chi_2} M_2 & \simeq \set*{
\begin{pmatrix}
0                  & 1              & 1                 & 0 \\
x + \frac{u+v}{4}  & 0              & 0                 & 1 \\
\frac{-3u+v}{4}    & -2y            & 0                 & 1 \\
z + \frac{u+v}{2}y & \frac{u-3v}{4} & x + \frac{u+v}{4} & 0
\end{pmatrix}
\st u,v,x,y,z \in \CC}, &
\Cpoly[\big]{\fg^* \qq{\chi_2} M_2} & = \Cpoly{u,v,x,y,z}.
\end{align*}
The non-zero Poisson brackets are given by the formulae:
\begin{align*}
\poibr{a,d} & = \tfrac{-1}{24} d     & \poibr{c,d} & = \tfrac{1}{6} bd  & &&
  \poibr{u,y} & = \tfrac{1}{8}  (u + x + y^2) &
  \poibr{u,z} & = \tfrac{1}{4} x (u + x + y^2) \\
\poibr{a,f} & = \tfrac{1}{24} f      & \poibr{c,f} & = \tfrac{-1}{6} bf & &&
  \poibr{v,y} & = \tfrac{-1}{8} (v + x + y^2) &
  \poibr{v,z} & = \tfrac{-1}{4} x (v + x + y^2) \\
\poibr{d,f} & = \mathrlap{\tfrac{-27}{2} a^3 + ab - \tfrac{1}{8} c} & & & &&
  \poibr{u,v} & = \mathrlap{\tfrac{-1}{4} \pp[\big]{ z + xy + 2(u+v)y }}.
\end{align*}

\noindent
Consider the ring homomorphism
$\varphi \colon \Cpoly[\big]{\slodowy_{\chi_2}} \to \Cpoly[\big]{\fg^* \qq{\chi_2} M_2}$
defined on generators by
\begin{align*}
\varphi(a) & = \tfrac{-1}{3} y   & \varphi(b) & = x   & \varphi(c) & = 2z - \tfrac{8}{3} xy   &
\varphi(d) & = v + x + y^2       & \varphi(f) & = -u - x - y^2.
\end{align*}
It can be checked that this map is a ring isomorphism and also preserves the
Poisson bracket; it hence induces an isomorphism of the Poisson varieties
$\slodowy_{\chi_2} \simeq \fg^* \qq{\chi_2} M_2$. Furthermore, this map preserves the
characteristic polynomial, and is the unique map satisfying all these
properties, up to the automorphism $\alpha$ of $\Cpoly[\big]{\slodowy_{\chi_2}}$
($\alpha(d) = -d$, $\alpha(f) = -f$, and all other generators fixed).
\end{eg}

\section{Quantum Hamiltonian reduction by stages for W-algebras}
\label{sec:QHRStages}

The construction of \cref{sec:GeneralConstruction} gives a method for
constructing semidirect product decompositons $M_2 \simeq M_1 \rtimes K$ suitable for
constructing Slodowy slices as Hamiltonian reductions of Slodowy slices
corresponding to more singular nilpotents. However, it has been presented in
such a way that its generalisation to the quantum case and \Walgebras\ is
straightforward. In particular, we shall construct a notion of \emph{quantum
Hamiltonian reduction by stages} which generalises the classical Hamiltonian
reduction by stages, and present a construction expressing \Walgebras\
as quantum Hamiltonian reductions of \Walgebras\ corresponding to more singular
nilpotent elements. What follows is a quantum version of
\cref{thm:2StageReduction}, the Hamiltonian reduction by stages theorem.

\begin{thm}
\label{thm:Quantum2StageReduction}
Let $\fg$ be a Lie algebra with universal enveloping algebra $U(\fg)$, and let
$\fm_1$, $\fm_2$ and $\fk$ be ad-nilpotent subalgebras of $\fg$. Furthermore,
let $\fm_2 = \fm_1 \rtimes \fk$, and let $\chi_2$ be a character of $\fm_2$ which
decomposes as $\chi_2 = \pp{\chi_1,\kappa}$. Denoting by $\fm_{1,\chi_1}$, $\fm_{2,\chi_2}$ and
$\fk_\kappa$ the corresponding shifted Lie algebras, define the quantum Hamiltonian
reductions by
\begin{align*}
U_1 \coloneqq U(\fg) \qq{\chi_1} \fm_1 & = \pp[\big]{U(\fg) \big/ U(\fg) \fm_{1,\chi_1}}^{\fm_1} &
\text{and} \quad\quad
U_2 \coloneqq U(\fg) \qq{\chi_2} \fm_2 & = \pp[\big]{U(\fg) \big/ U(\fg) \fm_{2,\chi_2}}^{\fm_2},
\end{align*}
where the invariants are equivalently either taken with respect to the adjoint
action or left multiplication by the shifted Lie algebras. Then $U_2$ can be
expressed as a quantum Hamiltonian reduction of $U_1$:
\begin{equation*}
U_2 \simeq \pp[\big]{ U_1 \big/ U_1\fk_\kappa }^\fk.
\end{equation*}
\end{thm}

With this \namecref{thm:Quantum2StageReduction}, it is now possible to approach
the problem of expressing \Walgebras\ as intermediate quantum Hamiltonian
reductions. Specifically, since the construction of
\cref{sec:GeneralConstruction} is phrased in terms of Lie algebras and
characters, one can apply this \namecref{thm:Quantum2StageReduction} directly to
the construction to obtain a quantum Hamiltonian reduction of the \Walgebra\
$U(\fg,e_1)$. Choosing $\fm_1$ to be a Premet subalgebra for $e_1$, the algebra
$U_1$ is just the \Walgebra\ $U(\fg,e_1)$, and the construction gives a
nilpotent $e_2$ with a subalgebra $\fm_2$ satisfying
\begin{equation*}
U(\fg,e_1) \qq{\kappa} \fk \simeq U(\fg) \qq{\chi_2} \fm_2.
\end{equation*}

\begin{conj}
\label{conj:typeAQred}
The reduced space $U(\fg) \qq{\chi_2} \fm_2$ is isomorphic to the \Walgebra\
$U(\fg,e_2)$.
\end{conj}

This \namecref{conj:typeAQred} is a quantum version of \cref{conj:typeAred}, and
its veracity is closely related. Specifically, since the \Walgebras\
$U(\fg,e_i)$ are filtered algebras whose associated graded algebras are
$\Cpoly[\big]{\slodowy_{\chi_i}}$, to prove that
$U(\fg,e_2) \simeq U(\fg) \qq{\chi_2} \fm_2$ would require lifting an isomorphism
$\varphi \colon \fg \qq{\chi_2} M_2 \to \slodowy_{\chi_2}$ to a homomorphism $\widetilde{\varphi}
\colon U(\fg,e_2) \to U(\fg) \qq{\chi_2} \fm_2$, which would then automatically be
an isomorphism by the general properties of filtered algebras.

\chapter[Category~\texorpdfstring{$\cO$}{O} for W-algebras]
{Category~\texorpdfstring{$\mathbfcal{O}$}{O} for W-algebras}
\label{chap:CatOWalg}

The BGG category~$\cO$, hereafter referred to just as \emph{category $\cO$}, is
a certain full subcategory of the category of all $\fg$\modules\ which satisfies
three conditions. Given a triangular decomposition $\fg = \fn_- \oplus \fh \oplus \fn_+$,
category~$\cO$ consists of all modules in $\Mod{\fg}$ which:
\begin{enumerate}
\item \label[condition]{cond:CatOfg}
      are finitely-generated,
\item \label[condition]{cond:CatOss}
      are acted upon semi-simply by $\fh$,
\item \label[condition]{cond:CatOln}
      are acted upon locally nilpotently by $\fn_+$.
\end{enumerate}
This category was originally defined by Bernstein, Gel'fand and Gel'fand in
\cite{BGG:CatO}, and has since proven extremely useful to mathematicians, in
particular for its relation to a number of categorification constructions
(cf.~\cite{BFK:CatTLAlg,KMS:CatSpecht}).

A similar subcategory exists in the category of all modules over a \Walgebra\
$U(\fg,e)$. As $U(\fg,e)$ is a subquotient of $U(\fg)$, it can be equipped with
a similar triangular decomposition, and analogues of
\crefrange*{cond:CatOfg}{cond:CatOln} can be formulated.
In \cite{Los:CatOWAlg}, Losev investigates the structure of this analogue of
category~$\cO$ for \Walgebras, and constructs an equivalence between it and a
certain subcategory of $\Mod{\fg}$.

The objective of this \namecref{chap:CatOWalg} is to prove a similar
equivalence in type~A, relating the categories~$\cO$ for \Walgebras\ to the
categories~$\cO$ for \Walgebras\ for more singular nilpotent elements, using the
quantum Hamiltonian reduction by stages construction developed in
\cref{sec:QHRStages}. In particular, this \namecref{chap:CatOWalg} shall assume
that \cref{conj:typeAQred} holds.

\section{The Skryabin equivalence}
\label{sec:Skyrabin}

To work with category~$\cO$ for a \Walgebra, it is first necessary to consider
the category of all modules over a \Walgebra: $\Mod{U(\fg,e)}$. Although
\Walgebras\ are difficult to work with on their own, the situation can be
significantly simplified by constructing an equivalence between $\Mod{U(\fg,e)}$
and a certain subcategory of $\Mod{\fg}$ corresponding to what are known as
\emph{Whittaker modules}.

Recall that the \Walgebra\ $U(\fg,e)$ can be defined as the invariants under the
adjoint action of $\fm$ in the \emph{generalised Gel'fand--Graev module}
$Q_\chi \coloneqq U(\fg) \big/ U(\fg)\fm_\chi$.
\begin{equation*}
U(\fg,e) = \pp[\big]{Q_\chi}^\fm.
\end{equation*}
The definition of $Q_\chi$ allows one to note that it both has the usual left
action of $U(\fg)$ by multiplication, but it further has a well-defined right
action of $U(\fg,e)$. Hence, $Q_\chi$ can be seen to be a $\pp{U(\fg),
U(\fg,e)}$\bimodule, and therefore induces an adjoint pair of functors between
the associated module categories:
\begin{align*}
 Q_\chi \otimes_{U(\fg,e)} \text{--} & \colon \Mod{U(\fg,e)} \to \Mod{U(\fg)} \\
 \Hom_{U(\fg)} \pp[\big]{Q_\chi,\text{--}} & \colon \Mod{U(\fg)} \to \Mod{U(\fg,e)}
\end{align*}

\begin{defn}
The functor $Q_\chi \otimes_{U(\fg,e)} \text{--}$ is called the \emph{Skryabin functor}.
\end{defn}

The essential image of the Skryabin functor can be described in intrinsic terms:
it is the full subcategory of $\Mod{U(\fg)}$ consisting of all \emph{Whittaker
modules} for the pair $(\fm,\chi)$.

\begin{defn}
A module $M \in \Mod{U(\fg)}$ is called a \emph{Whittaker module} for $(\fm,\chi)$ if
the Lie algebra $\fm$ acts by the generalised character $\chi$, or equivalently if
the shifted Lie algebra $\fm_\chi$ acts locally nilpotently on $M$. The category of
Whittaker modules is denoted $\Whit_{\fm,\chi}(\fg)$.
\begin{equation*}
\Whit_{\fm,\chi} (\fg) \coloneqq
 \set[\big]{ M \in \Mod{U(\fg)} \st \forall \; m \in M, y \in \fm, \;\exists \; n > 0
               \text{ such that } (y - \chi(y))^n m = 0 }
\end{equation*}

\noindent
A \emph{Whittaker vector} in a Whittaker module $M$ is a vector $m \in M$ on which
$\fm$ acts strictly by the character $\chi$. The space of Whittaker vectors of $M$
is denoted $\Wh(M)$.
\begin{equation*}
\Wh(M) \coloneqq \set{ m \in M \st \forall \; y \in \fm, (y - \chi(y)) m = 0 }
\end{equation*}
\end{defn}

\begin{lem}
For any $M \in \Mod{U(\fg,e)}$, $Q_\chi \otimes_{U(\fg,e)} M$ is a Whittaker module for
$(\fm,\chi)$. Furthermore, for a Whittaker module $M \in \Whit_{\fm,\chi}(\fg)$, the
space of Whittaker vectors $\Wh(M)$ is naturally a $U(\fg,e)$\module.
\end{lem}

\begin{proof} To prove the first claim, it suffices to check that $y - \chi(y)$
acts locally nilpotently on $Q_\chi$ for all $y \in \fm$. Since $\fm$ is strictly
negatively graded with respect to the good grading $\Gamma$, it acts locally
nilpotently on the Lie algebra $\fg$. The induced adjoint action on $U(\fg)$ is
also locally nilpotent, which can be seen by induction on the length in the PBW
filtration. Hence for a sufficiently large $n > 0$, one can commute
$(y - \chi(y))^n u = u (y - \chi(y))^n \in U(\fg) \fm_\chi$, and so $\fm_\chi$ acts locally
nilpotently on $Q_\chi$.

To prove the second claim, it suffices to note that for $m \in \Wh(M)$,
\begin{equation*}
(y + U(\fg)\fm_\chi) \cdot m = y \cdot m + U(\fg) \fm_\chi \cdot m = y \cdot m
\end{equation*}
and so lifting $y + U(\fg)\fm_\chi \in U(\fg,e)$ to $y \in U(\fg)$ results in a
well-defined action of $U(\fg,e)$ on~$M$.
\end{proof}

\begin{thm}[The Skryabin equivalence]{\normalfont \cite[Appendix~1, due to Skryabin]{Pre:TransSlice}}
\label{thm:SkryEquiv}
\hspace*{\fill}\\
The Skryabin functor is an equivalence of categories
\begin{align*}
Q_\chi \otimes_{U(\fg,e)} \text{--} & \colon \Mod{U(\fg,e)} \to \Whit_{\fm,\chi}(\fg) \\
\shortintertext{with quasi-inverse functor}
\Wh & \colon \Whit_{\fm,\chi}(\fg) \to \Mod{U(\fg,e)}.
\end{align*}
\end{thm}

\begin{proof} \cite[Theorem~6.1]{GG:QuantSlod}
We begin by showing that $\Wh \pp[\big]{Q_\chi \otimes_{U(\fg,e)} M} \simeq M$ for all modules
$M \in \Mod{U(\fg,e)}$, and to this end note that
$\Wh \pp[\big]{Q_\chi \otimes_{U(\fg,e)} M} \simeq H^0 \pp[\big]{\fm, \pp{Q_\chi \otimes_{U(\fg,e)} M}}$,
where $Q_\chi \otimes_{U(\fg,e)} M$ is regarded as a $\fm$-module with action twisted
by~$\chi$. The proof then proceeds in a similar manner as the proof of
\cref{thm:WalgQuantIso}: by demonstrating an isomorphism on the associated
graded modules.

\begin{claim}
Assume that $V$ is generated by a finite-dimensional subspace $V_0$, and that
$V$ is therefore a filtered $U(\fg,e)$\module\ with filtration
$F_n V \coloneqq (F_n U(\fg,e)) \cdot V_0$. Then 
\begin{equation*}
\gr H^0 \pp[\big]{\fm, Q_\chi \otimes_{U(\fg,e)} V } \simeq
  H^0 \pp[\big]{\fm, \gr \pp[\big]{Q_\chi \otimes_{U(\fg,e)} V} } \simeq \gr V,
\end{equation*}
and all higher cohomology vanishes,
i.e.~$H^i \pp[\big]{\fm, Q_\chi \otimes_{U(\fg,e)} V } = 0$ for all $i > 0$.
\end{claim}

The second isomorphism follows from \cref{thm:WalgQuant,lem:WalgQuantAdjDecomp}:
\begin{equation*}
\gr \pp[\big]{Q_\chi \otimes_{U(\fg,e)} V} \simeq \gr Q_\chi \otimes_{\gr U(\fg,e)} \gr V \simeq
  \pp[\big]{ \Cpoly{M'} \otimes \gr U(\fg,e) } \otimes_{\gr U(\fg,e)} \gr V \simeq
  \Cpoly{M'} \otimes \gr V.
\end{equation*}
The first isomorphism follows directly by noting that $Q_\chi \otimes_{U(\fg,e)} V$ is
positively-graded and $\fm$ is strictly-negatively graded, and then following
through the proof of \cref{thm:WalgQuantLieCoh} \emph{mutatis mutandis}.

To prove the \namecref{thm:SkryEquiv}, it show to prove that for any
$V \in \Whit_{\fm,\chi}$, the canonical map
\begin{align*}
\gamma & \colon Q_\chi \otimes_{U(\fg,e)} \Wh(V) \to V,  &
\pp[\big]{u + U(\fg)\fm_\chi} \otimes m & \mapsto u \cdot m
\end{align*}
is an isomorphism. Consider an exact sequence of $U(\fg)$\modules
\begin{equation}
0 \lra V' \lra Q_\chi \otimes_{U(\fg,e)} \Wh(V) \overset{\gamma}{\lra} V \lra V'' \lra 0.
\label{eq:SkryLES}
\end{equation}

To see $\gamma$ is injective, note that
$\Wh(V') = V' \cap \Wh \pp[\big]{ Q_\chi \otimes_{U(\fg,e)} \Wh(V) } = V' \cap \Wh(V) = 0$,
where the first equality follows from the fact that $V'$ is a submodule, the
second equality follows from the first part of this \namecref{thm:SkryEquiv},
and the third follows from the definition of $\gamma$. But $\Wh(V)$ cannot be zero
without $V$ also being zero, hence $\gamma$ is injective.

To prove $\gamma$ is surjective, we aim to show that $V'' = 0$. The short exact
sequence (in view of the fact that $V' = 0$) \cref*{eq:SkryLES} gives rise to a
long exact sequence on cohomology
\begin{equation*}
0 \lra H^0 \pp[\big]{\fm, Q_\chi \otimes_{U(\fg,e)} \Wh(V) } \overset{\gamma^*}{\lra}
  H^0(\fm, V) \lra H^0(\fm, V'') \lra H^1 \pp[\big]{\fm, Q_\chi \otimes_{U(\fg,e)} V } = 0,
\end{equation*}
where the last equality follows from the above claim. Re-writing this sequence
as an exact sequence of Whittaker vectors yields
\begin{equation*}
0 \lra \Wh \pp[\big]{Q_\chi \otimes_{U(\fg,e)} \Wh(V) } \overset{\gamma^*}{\lra} \Wh(V)
  \lra \Wh(V'') \lra 0.
\end{equation*}
But the first two terms are equal, and hence $\gamma^*$ is an isomorphism and
$V'' = 0$, completing the proof that $\gamma$ is surjective and therefore an
isomorphism.
\end{proof}

\section[The definition of category~\texorpdfstring{$\cO$}{O} for W-algebras]
{The definition of category~\texorpdfstring{$\mathbfcal{O}$}{O} for W-algebras}

We now seek to define an analogue of category~$\cO$ for the
\Walgebras~$U(\fg,e)$. To do this, we need to make a choice of parabolic
subalgebra~$\fp$ satisfying certain conditions. Let $\fp$ be a parabolic
subalgebra of~$\fg$ satisfying the two conditions:
\begin{enumerate}
\item $e$ is a distinguished nilpotent in the Levi subalgebra~$\fl$ of~$\fp$,
      i.e.~the centraliser $\liez[\fl]{e}$ contains no semisimple elements not
      lying in $\liez{\fl}$;
\item $\fp$ contains a fixed maximal torus $\ft$ of the centraliser
      $\liez[\fg]{e}$.
\end{enumerate}

The choice of parabolic $\fp$ allows one to define a pre-order on the weights of
$\ft$ as follows: $\lambda \ge \mu$ if and only if $\lambda - \mu$ is a linear combination of the
weights of $\ft$ acting on $\fp$.

\begin{eg}
When $e = 0$, then $\fp$ can be taken to be a Borel subalgebra $\fb$ with Levi
a fixed Cartan subalgebra $\fh$. As the centraliser $\liez[\fg]{e} = \fg$, the
maximal torus $\ft$ can be taken to be~$\fh$. The pre-order is the classical
partial order on weights: $\lambda \ge \mu$ if and only if $\lambda - \mu$ is a positive linear
combination of simple roots.
\end{eg}

\begin{eg}
When $e$ is a regular nilpotent in $\fg$, then $\fp$ can be taken to be $\fg$
itself with Levi subalgebra $\fg$. The maximal torus is then just
$\ft = \set{0}$. The only weight $\ft$ is the zero weight, and the pre-order is
trivial.
\end{eg}

\begin{eg}
Let $\fg = \fsl_3$ and
$e = \begin{psmallmatrix} 0&1&0 \\ 0&0&0 \\ 0&0&0 \end{psmallmatrix}$. The
parabolic subalgebra $\fp$, Levi subalgebra $\fl$ and maximal torus $\ft$ can be
taken to be
\begin{align*}
\fp & = \begin{psmallmatrix} *&*&* \\ *&*&* \\ 0&0&* \end{psmallmatrix} &
\fl & = \begin{psmallmatrix} *&*&0 \\ *&*&0 \\ 0&0&* \end{psmallmatrix} &
\ft & = \set*{\begin{psmallmatrix} a&0&0 \\ 0&a&0 \\ 0&0&-2a \end{psmallmatrix}
                \st a \in \CC}
\end{align*}
As $\ft$ is a one-dimensional, its weights are just complex numbers. Further,
the set of weights of~$\ft$ acting on~$\fp$ is the set $\set{0,3}$.
\end{eg}

There is furthermore an embedding $U(\ft) \hookrightarrow U(\fg,e)$, and hence
any $U(\fg,e)$\module\ can be decomposed into its generalised
weight spaces with respect to~$\ft$. This can be seen either by choosing a
$\ft$\invariant\ Premet subalgebra $\fm$, and hence the image of $U(\ft)$ in
$U(\fg) \big/ U(\fg)\fm_\chi$ can be seen to be $\ad \fm$\invariant, or one can prove
it for any choice of Premet subalgebra by a more careful argument, as in
\cite[Theorem~3.3]{BGK:HWTWAlg}. This allows one to formulate a version of
category~$\cO$ for \Walgebras.

\begin{defn}
Given a nilpotent element $e \in \fg$ with a compatible choice of parabolic
subalgebra $\fp$ and maximal torus $\ft$ in $\liez{e}$ as above, we define a
number of versions of \emph{category~$\cO$} for the \Walgebra\ $U(\fg,e)$. Each
is a full subcategory of the finitely-generated modules in $\Mod{U(\fg,e)}$
satisfying certain conditions. The weakest version is denoted
$\wO~*{e,\fp}$, and consists of those modules $M$ satisfying
\begin{enumerate}
\item \label[condition]{cond:CatOWln}
      the weights of $M$ are contained in a finite union of of sets of the form
      $\set{\mu \st \mu \le \lambda}$.
\end{enumerate}
Another subcategory, denoted $\wO*{e,\fp}$, is defined as consisting of modules
which further satisfy the condition
\begin{enumerate}
\setcounter{enumi}{1}
\item \label[condition]{cond:CatOWss}
      $\ft$ acts on $M$ with finite-dimensional generalised weight spaces.
\end{enumerate}

For each of these versions of category~$\cO$, we define a further subcategory
consisting of those modules on which $\ft$ acts semi-simply, and denote them
$\wO~{e,\fp}$ and $\wO{e,\fp}$, respectively. Note that
$\wO{e,\fp} = \wO~{e,\fp} \cap \wO*{e,\fp}$.
\end{defn}

\begin{rem}
This forms an analogue of the classical BGG category~$\cO$ defined at the
beginning of this \namecref{chap:CatOWalg}. \Cref*{cond:CatOWss} here
corresponds to \cref{cond:CatOss} in that definition, and \cref*{cond:CatOWln}
here corresponds to both \cref{cond:CatOfg,cond:CatOln}.
\end{rem}

\begin{eg}
If $e = 0$ and $\fb$ is a chosen Borel subalgebra, then $\wO{0,\fb}$ is just the
classical category~$\cO$ for the chosen Borel.
\end{eg}

\begin{eg}
If $e_\text{reg} \in \fg$ is a regular nilpotent, then $\fp = \fg$ and
$\wO{e_\text{reg},\fg}$ is just the category of all finite-dimensional modules
over $U(\fg,e_\text{reg}) \simeq Z(\fg)$:
\begin{equation*}
\wO{e_{\text{reg}}, \fg} \simeq \FMod{Z(\fg)}.
\end{equation*}
\end{eg}

\section{Generalised Whittaker modules and a category equivalence}

The Skryabin equivalence gives an equivalence between the category of all
modules for a \Walgebra\ and the category of Whittaker modules for the Premet
subalgebra~$\fm$ and character~$\chi$. In \cite[Theorem~4.1]{Los:CatOWAlg}, Losev
introduces the concept of a \emph{generalised Whittaker module}, and proves an
equivalence between the categories~$\cO$ for \Walgebras\ and appropriate
categories of generalised Whittaker modules.

Let $e \in \fg$ be a nilpotent with good grading
$\fg = \bigoplus_{i \in \ZZ} \fg_i$. Given a choice of parabolic $\fp$ with $e$
regular in the Levi $\fl$, one can construct an analogue of the Premet
subalgebra $\fm \subseteq \fg$. Let $\fl = \bigoplus_{i \in \ZZ} \fl_i$ be the induced
grading on $\fl$ coming from $\fg$, and choose a Premet subalgebra
$\underline{\fm} \subseteq \fl$ in the usual way. Note that this allows one to define
the \Walgebra\ 
$U(\fl,e) \coloneqq \pp[\big]{ U(\fl) \big/ U(\fl)\underline{\fm}_\chi }^{\ad \underline{\fm}}$.
One can then consider the nilradical~$\fu$ of~$\fp$, where $\fp = \fl \oplus \fu$,
and define the new subalgebra $\widetilde{\fm} \coloneqq \underline{\fm} \oplus \fu$, along
with its shift $\widetilde{\fm}_\chi$.

\begin{defn}
A $U(\fg)$\module~M is a \emph{generalised Whittaker module} for the
pair~$(e, \fp)$ if $\widetilde{\fm}$ acts on $M$ by generalised character~$\chi$,
or equivalently if $\widetilde{\fm}_\chi$ acts on $M$ by locally nilpotent
endomorphisms.
The full subcategory of all generalised Whittaker modules in
$\Mod{U(\fg)}$ is denoted~$\wWhit~*{e,\fp}$, while the
full subcategory of $\wWhit~*{e,\fp}$ consisting of those modules on
which $\ft$ acts semi-simply is denoted~$\smash{\wWhit~{e,\fp}}$.
\end{defn}

The functor $\smash{\pp{\text{--}}^{\widetilde{\fm}_\chi}}$, taking
$\widetilde{\fm}_\chi$\invariants\ of the module $M$, gives a natural functor
from $\wWhit~*{e,\fp}$ to $\Mod{U(\fl,e)}$. This can be seen as
$U(\fl)$ acts naturally on $M^{\fu}$ for any $M \in \wWhit~*{e,\fp}$,
and $M^{\widetilde{\fm}_\chi} = \smash{\pp{M^\fu}^{\underline{\fm}_\chi}}$, so there
is a well-defined action of $U(\fl) \big/ U(\fl) \underline{\fm}_\chi$.
A module $M \in \wWhit~*{e,\fp}$ is said to be \emph{of finite type} if
$M^{\widetilde{\fm}_\chi}$ is finite-dimensional when viewed as a
$U(\fl,e)$\module.

\begin{prop}{\normalfont \cite[Proposition~4.2]{Los:CatOWAlg}}
If $e \in \fg$ is a \emph{regular nilpotent} in $\fl$, then a generalised
Whittaker module $M \in \wWhit~*{e,\fp}$ is of finite type if and only
if the action of $Z(\fg) \subseteq U(\fg)$ on $M$ is locally finite. In particular, this
holds for any nilpotent in \typeA.
\end{prop}

\begin{defn}
The category of all finite-type modules in $\wWhit~*{e,\fp}$ is
denoted~$\wWhit*{e,\fp}$, and the subcategory of finite-type modules on which
$\ft$ acts semi-simply is denoted~$\wWhit{e,\fp}$.
\end{defn}

\begin{thm}{\normalfont \cite[Theorem~4.1]{Los:CatOWAlg}}
\label{thm:CatOEquiv}
There is an equivalence of categories
\begin{align*}
\fK & \colon \wWhit~*{e,\fp} \to \wO~*{e,\fp}.
\intertext{This furthermore induces equivalences of subcategories}
\fK & \colon \wWhit~{e,\fp}  \to \wO~{e,\fp} \\
\fK & \colon \wWhit*{e,\fp}  \to \wO*{e,\fp} \\
\fK & \colon \wWhit{e,\fp}   \to \wO{e,\fp}.
\end{align*}
\end{thm}

The proof of \cref*{thm:CatOEquiv} uses Losev's techniques of completions of
deformation quantisations of Poisson algebras, developed in,
e.g.~\cite{Los:QuantSymplWalg}. The second main result of this thesis is a
generalisation of the above \namecref{thm:CatOEquiv}, so we shall develop
Losev's machinery here, with an eye to applying it in the proof of this
\namecref{thm:CatOEquiv} and its generalisation.

\subsection{Deformation quantisations of Poisson algebras}

Losev's results are geometric in nature, and rely on the interpretation of
\Walgebras\ as \emph{deformation quantisations} of certain commutative Poisson
algebras.

\begin{defn}
\label{def:DefQuant}
Given a Poisson algebra $A$ with Poisson bracket $\poibr{\cdot,\cdot}$, a
\emph{deformation quantisation} of $A$ is an associative unital product
$\star \colon A \otimes A \to \power{A}{\hbar}$ (often called a \emph{star
product}). Writing $f \star g = \sum_{k\ge0} D_k(f,g) \hbar^{2k}$, and extending
$\Cpower{\hbar}$-bilinearly to a product on $\power{A}{\hbar}$, a deformation
quantisation $\star$ is one which satisfies the following conditions for
$f,g \in A$ (viewed as elements of $\power{A}{\hbar}$ by the natural inclusion):
\begin{enumerate}
\item \label[condition]{cond:DefQuantProd}
      $\star$ is an associative binary product on $\power{A}{\hbar}$, continuous
      in the $\hbar$-adic topology;
\item \label[condition]{cond:DefQuantMult}
      $f \star g = fg + O(\hbar^2)$ (i.e.\ $D_0(f,g) = fg$), which implies that
      $\star$ degenerates to the usual multiplication in $A$ when $\hbar = 0$;
\item \label[condition]{cond:DefQuantPoisson}
      $[f,g] \coloneqq f \star g - g \star f = \hbar^2 \poibr{f,g} + O(\hbar^4)$ (i.e.\
      $D_1(f,g) - D_1(g,f) = \poibr{f,g}$), which means the
      $\hbar^2$\nobreakdash-term of $\star$ encodes the Poisson bracket of $A$.
\end{enumerate}
We shall also require the stronger condition that $\star$ is a
\emph{differential} deformation quantisation, namely that:
\begin{enumerate}
\setcounter{enumi}{3}
\item For each $k$, $D_k(\cdot,\cdot)$ is a bidifferential operator of order at
      most $k$ in each variable.
\end{enumerate}
\end{defn}

\begin{rem}
Usually, the star product is expanded as $f \star g \coloneqq \sum_{k\ge0} D_k(f,g) \hbar^k$,
and the powers of $\hbar$ in \cref{cond:DefQuantMult,cond:DefQuantPoisson} are
halved (i.e.\ $\hbar$ in place of $\hbar^2$, and $\hbar^2$ in place of
$\hbar^4$). The conventions here are used for better compatibility with the
Kazhdan filtration.
\end{rem}

\begin{rem}
The star product on the algebra $\power{A}{\hbar}$ can be used to introduce a
new product on the Poisson algebra $A$, induced by the projection
$\power{A}{\hbar} \to A$, $\hbar \mapsto 1$. Concretely, define the new
associative product $\circ \colon A \otimes A \to A$ by $f \circ g \coloneqq \sum_{k\ge0} D_k(f,g)$.
Denote the algebra $A$ with this new algebra structure by $\cA$.
\end{rem}

\begin{prop}{\normalfont \cite[Corollary~3.3.3]{Los:QuantSymplWalg}}
The Rees algebra of the \Walgebra\ $U(\fg,e)$, viewed as a filtered algebra with
the Kazhdan filtration, is the unique deformation quantisation of the ring of functions
on the Slodowy slice~$\slodowy_\chi$, up to isomorphism.
\end{prop}

\subsection{Completions of quantum algebras and Losev's machinery}
\label{sec:LosevPaper}

To prove \cref{thm:CatOEquiv}, we need to introduce a technical theorem of
Losev \cite[Proposition~5.1]{Los:CatOWAlg}, and to that end introduce some
needed notation. In addition, the machinery of this proposition shall be needed
to prove the second main result of this thesis.

Assume that all of the following hold:
\begin{itemize}
\item $\fv = \bigoplus_{i \in \ZZ} \fv_i$ is a graded finite-dimensional vector
      space on which a torus $T$ acting by preserving the grading.
\item $A \coloneqq \Sym(\fv)$, with an induced grading $A = \bigoplus_{i \in \ZZ} A_i$ and
      induced action by $T$.
\item $\cA$ is an algebra with the same underlying vector space as $A$, where
      the algebra structure comes from a $T$-invariant homogeneous star
      product.
\item $\omega_1$ is a symplectic form on $\fv_1$ (where $\omega_1(u,v)$ is the constant
      term of the commutator in $\cA$), and $\fy$ is a lagrangian subspace
      of $\fv_1$.
\item $\fm \coloneqq \fy \oplus \bigoplus_{i\le0} \fv_i$.
\item $v_1, v_2, \dotsc, v_n$ is a homogeneous basis of $\fv$ such that
      $v_1, v_2, \dotsc, v_m$ form a basis of $\fm$. Further, let $d_i$ be the
      degree of $v_i$ and assume that $d_1, d_2, \dotsc, d_m$ are increasing and
      that all $v_i$ are $T$-semi-invariant.
\item $A^\heartsuit$ is the subalgebra of $\Cpower{\fv^*}$ consisting of elements
      of the form $\sum_{i\le n} f_i$ for some $n$, where $f_i$ is a homogeneous power
      series of degree $i$.
\item $\cA^\heartsuit$ is the algebra $A^\heartsuit$ with multiplication coming
      from $\cA$. Any element of $\cA^\heartsuit$ can be written as an infinite
      linear combination of monomials $v_{i_1} v_{i_2} \dotsm v_{i_\ell}$ where
      $i_1 \ge \dotsb \ge i_\ell$, and where $\sum_{j=1}^\ell d_{i_j} \le c$ for some
      $c$. This gives us a filtration $F_c \cA^\heartsuit$.
\item $\theta$ is a co-character of $T$, and $\fv_{\ge0}$ and $\fv_{>0}$ are the sums
      of positive and strictly-positive (respectively) $\ad \theta$-eigenspaces of
      $\fv$. We require that $\fv_{>0} \subseteq \fm \subseteq \fv_{\ge0}$. \emph{Note that beyond
      this condition, these spaces are unrelated to the $\ZZ$\grading\ on the
      vector space $\fv$.}
\item $\cA_{\ge0}, \cA_{>0}, \cA^\heartsuit_{\ge0}, \cA^\heartsuit_{>0}$ are all
      defined analogously.
\item $\cA^\wedge \coloneqq \lim \cA / \cA\fm^k$, which has an injective algebra
      homomorphism $\cA^\heartsuit \to \cA^\wedge$.
\end{itemize}

\begin{prop}{\normalfont \cite[Proposition~5.1]{Los:CatOWAlg}}
\label{prop:los5.1}
Let $\fv$ and $A$ be as above, and $\cA$ and $\cA'$ be two different algebras
coming from $A$ as above with two different star products. Suppose there is a
subspace $\fy \subseteq \fv(1)$ which is Lagrangian for both symplectic forms, and
every element of $A$ can be written as a finite sum of monomials in both
$\cA$ and $\cA'$. Then any homogeneous $T$-equivariant isomorphism
$\Phi \colon \cA^\heartsuit \to \cA'^\heartsuit$ satisfying
$\Phi(v_i) - v_i \in F_{d_i-2} \cA + (F_{d_i} \cA \cap \fv^2 \cA)$
extends uniquely to a topological algebra isomorphism
$\Phi \colon \cA^\wedge \to \cA'^\wedge$ with
$\Phi(\cA^\wedge \fm) = \cA'^\wedge \fm$.
\end{prop}

\begin{cor}{\normalfont \cite[Corollary~5.2]{Los:CatOWAlg}}
\label{cor:los5.2}
The equivalence $\Phi \colon \cA^\wedge \to \cA'^\wedge$ of the above
\namecref{prop:los5.1} induces equivalences
$\Phi_* \colon \wWhit*{\fm}(\cA) \to \wWhit*{\fm}(\cA')$ and
$\Phi_* \colon \wWhit{\fm}(\cA) \to \wWhit{\fm}(\cA')$, where again
$\wWhit*{\fm}(\cA)$ is the category of all $\cA$\modules\ on which $\fm$ acts by
locally nilpotent endomorphisms, and $\wWhit{\fm}(\cA)$ is the subcategory on
which $\ft$ acts semi-simply. Furthermore, this functor commutes with taking
$\fm$\invariants, i.e.~$\Phi_*(M^\fm) = \Phi_*(M)^\fm$ for all
$M \in \wWhit*{\fm}(\cA)$.
\end{cor}

Losev uses these technical results to construct an isomorphism between
the completed universal enveloping algebra $U(\fg)^\wedge_\chi$ and
$U(\fg,e)^\wedge_\chi \hat{\otimes} \Weyl{V}^\wedge_0$, the completed tensor
product of the \Walgebra\ and the completed Weyl algebra $\Weyl{V}\wedge_0$,
where $V$ is the symplectic leaf through $e$. He does this by making the
following choices for the above objects.
Recall that we start with a semisimple Lie algebra $\fg$ with distinguished
nilpotent element $e \in \fg$ forming part of an \sltriple\ $\set{e,h,f}$.

\begin{enumerate}
\item $\fv \coloneqq \fg_\chi = \set{\xi - \chi(\xi) \st \xi \in \fg}$.
\item $T$ is a connected maximal torus in the centraliser
      $Z_\fg\pp[\big]{\set{e,h,f}}$ with $h \in \ft$; $\theta$ is an arbitrary
      (generic) co-character $\theta \in \Hom(\CC*,T)$. This determines a parabolic
      $\fp$ as the span of the positive eigenspaces of $\ad \theta$.
\item Choosing $m > 2 + 2d$, where $d$ is the maximum eigenvalue of $\ad h$ in
      $\fg$, define the grading on $\fv$ by
      $\fv_i = \set{\xi \in \fg \st (h-m\theta)\xi = (i-2)\xi}$. This is just the Kazhdan
      grading with a shift by $m$ so that everything in the nilradical of $\fp$
      has negative grading. The algebra $\fm$ corresponds to a choice of
      $\widetilde{\fm}_\chi$.
\item $\cA \coloneqq U(\fg)$ and $\cA' \coloneqq U(\fg,e) \otimes \Weyl{V}$, where $V = [\fg,f]$
      is the symplectic leaf in $\fg$ through $e$,
      and $\Weyl{V}$ is its Weyl algebra.
\end{enumerate}

The map $\Phi \colon U(\fg)^\heartsuit \to \pp[\big]{U(\fg,e) \otimes \Weyl{V}}^\heartsuit$
exists by \cite[Lemma~5.3]{Los:CatOWAlg}, and comes from an application of the
Luna slice theorem, and the identification of these spaces as the quotients of
the $\CC*$-finite parts of the quantum algebras
\begin{equation*}
\power{\Cpoly{\fg^*}^\wedge_\chi}{\hbar}
\qquad \text{and} \qquad
\power{\Cpoly{\slodowy_\chi}^\wedge_\chi}{\hbar} \widehat{\otimes}_{\Cpower{\hbar}}
  \power{\Cpoly{V^*}^\wedge_0}{\hbar}
\end{equation*}
by the ideal generated by $\hbar - 1$.

\begin{proof}[{Proof of \cref*{thm:CatOEquiv}}]

To prove the \namecref{thm:CatOEquiv}, we examine the category equivalences of
\cref{cor:los5.2}. Note that $\wWhit*{\fm}(\cA) = \wWhit~*{e,\fp}$
for $\fp$ the parabolic subalgebra determined by~$\theta$, and similarly the
subcategories of $\ft$\nobreakdash-semisimple modules are equal,
$\wWhit{\fm}(\cA) = \wWhit~{e,\fp}$.

It remains to construct an equivalence between $\wWhit*{\fm}(\cA')$ and
$\wO~*{e,\fp}$; let this be given by the invariant functor with respect to
$\fm \cap V$, which can be seen to be a Lagrangian subspace of~$V$ by applying
\cref{GGprop5} of good gradings in the Levi subalgebra.
\begin{align*}
\fK' \colon \wWhit*{\fm}(\cA') & \to \wO~*{e,\fp} \\
  M & \mapsto M^{\fm \cap V}
\end{align*}
The desired functor $\fK \colon \wWhit~*{e,\fp} \to \wO~*{e,p}$ is the
composition of functors $\fK' \circ \Phi_*$. It preserves the necessary
subcategories by \cref{cor:los5.2}.
\end{proof}

\section[Equivalences between categories \texorpdfstring{$\cO$}{O}]
{Equivalences between categories \texorpdfstring{$\mathbfcal{O}$}{O}}

We shall use a similar argument to that of Losev to prove an equivalence between
the category $\wO~*{e_2,\fp'}$ and an appropriate subcategory of Whittaker
vectors in $\wO~*{e_1,\fp}$ for a pair of nilpotent elements $e_1 \le e_2$ in a Lie
algebra of \typeA\, where the partial ordering considered is the
\emph{refinement ordering}. We expand upon the reductions produced in
\cref*{chap:TypeA} to produce these.

\begin{defn}
Let $\lambda = (\lambda_1,\dotsc,\lambda_k)$ and $\mu = (\mu_1,\dotsc,\mu_\ell)$ be two partitions of
$n$ satisfying $\lambda_1 \ge \dotsb \ge \lambda_k > 0$ and $\mu_1 \ge \dotsb \ge \mu_\ell > 0$. The
\emph{refinement ordering} is the partial ordering on the set of partitions
where $\lambda \ge \mu$ if and only if $\mu$ is a \emph{refinement} of $\lambda$, i.e.~there
exists a partition $\nu_1,\nu_2,\dotsc,\nu_k$ of the set $\set{1,\dotsc,\ell}$ such
that $\lambda_i = \sum_{j \in \nu_i} \mu_j$ for every $i$.

We also call the partial ordering on nilpotent orbits in $\fsl_n$ induced by
this ordering \emph{the refinement ordering}. Concretely, let $\orbit_1$ and
$\orbit_2$ be two nilpotent orbits in $\fsl_n$ corresponding to partitions $\lambda$
and $\mu$, respectively. We say that $\orbit_1 \ge \orbit_2$ in the refinement
ordering if $\lambda \ge \mu$ in the refinement ordering.
\end{defn}

\begin{rem}
Note that the dominance ordering is a \emph{refinement} of the refinement
ordering; this means that $\lambda \ge \mu$ in the dominance ordering whenever this holds
in the refinement ordering. In general, it can be the case that $\lambda \ge \mu$ in the
dominance ordering, but $\lambda \ngeq \mu$ in the refinement ordering. For example,
$(3,1) > (2,2)$ in the dominance ordering, but not in the refinement ordering.
\end{rem}

Choose a pair of \sltriples\ $\set{e_1,h_1,f_1}$ and $\set{e_2,h_2,f_2}$ for each of the
nilpotents elements $e_1$ and $e_2$. As in \cref{sec:LosevPaper}, we seek to
make a set of choices which satisfy the hypotheses of \cref{prop:los5.1}.

\begin{enumerate}
\item Define $\fv \coloneqq \set{\xi - \chi_2(\xi) \st \xi \in \liez{e_1}}$.

\item Note that since $e_1 < e_2$ in the refinement ordering, the pair of Levi
      subalgebras~$\fl_1$ and~$\fl_2$, chosen such that $e_1$ and~$e_2$ are
      regular their respective Levi subalgebras, can further be chosen so
      that~$\fl_1$ is a subalgebra of~$\fl_2$. As a result, a maximal connected
      torus $T$ in the centraliser $Z_\fg \pp[\big]{\set{e_2,h_2,f_2}}$ can be
      chosen which preserves $\fv$. Choose an arbitrary generic co-character
      $\theta_2 \in \Hom(\CC*,T)$.

\item Choosing $m > 2 + 2d$, where $d$ is the maximum eigenvalue of $\ad h_2$ on
      $\fg$, we define the grading on $\fv$ to be give by
      $\fv_i = \set{\xi \in \fv \st (\ad h_2-m\theta)\xi = (i-2)\xi}$.

\item Define $\cA \coloneqq U(\fg,e_1)$ and $\cA' = U(\fg,e_2) \otimes \Weyl{V}$, where
      $V = [\fg,f_2] \cap \liez{e_1}$ is the symplectic leaf in $\slodowy_{e_1}$
      passing through $e_2$.
\end{enumerate}

\begin{lem}
\label{lem:MyLos5.1}
These choices satisfy the hypotheses of \cref{prop:los5.1}.
\end{lem}

\begin{proof}
It is clear from the construction that $\fv$ is a finite-dimensional graded
vector space, and that the torus $T$ preserves the grading. Since
$\cA = U(\fg,e_1)$ is a deformation quantisation of the Slodowy slice
$\slodowy_{\chi_1}$, it follows that it comes from a $T$\invariant\ homogeneous
star product on $\Cpoly[\big]{\slodowy_{\chi_1}} = \Sym({\liez{e}})$. $\omega_1$ is
again a symplectic form on $\fv_1$, and $\fm$ is defined similarly. It remains
no demonstrate that $\cA$ and $\cA'$ are both algebras coming from
$A = \Sym(\liez{e})$ with different star products.

Consider the subalgebras $\fm_1$ and $\fm_2$ constructed in \cref{chap:TypeA}.
As the moment map pre-image $\mu_i^{-1}(\chi_i) = \chi_i + \fm_i^{*,\bot}$, this gives
the commutative diagram.
\begin{equation}
\begin{tikzcd}[]
\fg^*                    \rar[-, double equal sign distance]
  & \fg^*                                                                \\
\chi_2 + \fm_2^{*,\bot}     \uar[hookrightarrow] \rar[hookrightarrow]{\iota}
  & \chi_1 + \fm_1^{*,\bot} \uar[hookrightarrow]                            \\
\slodowy_{\chi_2}           \uar[twoheadleftarrow, xshift=.60ex]
                            \uar[hookrightarrow, xshift=-.60ex]
                            \rar[hookrightarrow]{\varphi}
  & \slodowy_{\chi_1}       \uar[twoheadleftarrow, xshift=.60ex]
                            \uar[hookrightarrow, xshift=-.60ex]
\end{tikzcd}
\label{eq:SlodowyInclusions}
\end{equation}
Here, the vertical maps $\chi_i + \fm_i^{*,\bot} \hookrightarrow \fg^*$ and
$\chi_i + \fm_i^{*,\bot} \twoheadrightarrow \slodowy_{\chi_i}$ are the natural maps
coming from the expression of $\slodowy_{\chi_i}$ as a Hamiltonian reduction of
$\fg^*$, and the maps $\slodowy_{\chi_i} \hookrightarrow \chi_i + \fm_i^{*,\bot}$ are
the maps which come from the natural expression
$\slodowy_{\chi_i} = \kappa \pp[\big]{e_i + \liez{f_i}} \subseteq \fg^*$. The map $\iota$ comes
from the embedding of $\fm_1$ as a subspace of $\fm_2$ by the decomposition
$\fm_2 = \fm_1 \rtimes \fk$, and the map $\varphi$ is defined as the composition of
the necessary morphisms. Note that the embedding
$\varphi \colon \slodowy_{\chi_2} \hookrightarrow \slodowy_{\chi_1}$ is a transverse slice
to the symplectic symplectic leaf
$V = \liebr{\fg, e_2} \cap \liez{e_1} \subseteq \slodowy_{\chi_1}$.
\end{proof}

\begin{lem}
\label{lem:MyHeartMap}
With these choices, there is a map
$\Phi \colon U(\fg,e_1)^\heartsuit \to \pp[\big]{U(\fg,e_2) \otimes \Weyl{V}}^\heartsuit$.
\end{lem}

\begin{proof} \cite[Theorem~3.3.1]{Los:QuantSymplWalg}
Consider Losev's \emph{equivariant Slodowy slices}:
$X_i \coloneqq G \times \slodowy_{\chi_i}$, where $G = SL_n$. The embedding
$\slodowy_{\chi_2} \hookrightarrow \slodowy_{\chi_1}$ extends to an embedding
$X_2 \hookrightarrow X_1$. Note that both $X_i$ have an action of
$\widetilde{G} \coloneqq G \times \CC* \times G_0$, where $G_0$ is the subgroup of
$Z_G(e_2,h_2,f_2)$ which preserves the grading, defined as follows:
\begin{align*}
g \cdot (g_1,\alpha)   & = (g g_1, \alpha) &
t \cdot (g_1,\alpha)   & = (g_1 \gamma(t)^{-1}, t \cdot \alpha) &
g_0 \cdot (g_1,\alpha) & = (g_1 g_0^{-1}, g_0 \alpha)
\end{align*}
Both $X_i$ are stable under the action of $\widetilde{G}$. Let $x = (1,\chi_2)$ and
note that since $\widetilde{G} \cdot x = G \cdot x$, which implies that
$\widetilde{G} \cdot x$ is closed. The stabiliser of $x$ is
$\set{(g_0 \gamma(t), t, g_0) \st t \in \CC*, g_0 \in G_0}$, which can be identified with
$G_0 \times \CC*$.

The symplectic subspace $T_x X_2 \subseteq T_x X_1$ has an orthogonal complement
identified with $V$, resulting in a $(G_0 \times \CC*)$\equivariant\
symplectomorphism $\psi \colon T_x X_1 \to T_x X_2 \oplus V^*$. The Fedosov star products
on $\power{\Cpoly{X_1}}{\hbar}$ and $\power{\Cpoly{X_2 \times V^*}}{\hbar}$ are
both differential, so they extend to the completions
$\power{\Cpoly{X_1}^\wedge_{G x}}{\hbar}$ and
$\power{\Cpoly{X_2 \times V^*}^\wedge_{G x}}{\hbar}$.

Now, proceeding as in \cite[Theorem~3.3.1]{Los:QuantSymplWalg}, the map $\psi$ and
the Luna slice theorem can be used to produce a $\widetilde{G}$\equivariant\
isomorphism
$\Phi_\hbar \colon \power{\Cpoly{X_1}^\wedge_{G x}}{\hbar}
            \to \power{\Cpoly{X_2 \times V^*}^\wedge_{G x}}{\hbar}$. Taking
$G$\invariants\ induces a map 
$\Phi_\hbar \colon \power{\Cpoly{\slodowy_{\chi_1}}^\wedge_{G x}}{\hbar}
            \to \power{\Cpoly{\slodowy_{\chi_2} \times V^*}^\wedge_{G x}}{\hbar}$,
and restricting to the $\CC*$-finite parts produces the desired map
$\Phi \colon U(\fg,e_1)^\heartsuit \to \pp[\big]{U(\fg,e_2) \otimes \Weyl{V}}^\heartsuit$.
\end{proof}

Combining \cref{lem:MyLos5.1,lem:MyHeartMap} with \cref{prop:los5.1,cor:los5.2}
completes the proof of the following \namecref{thm:RefinedCatOEquiv}, forming an
analogue of Losev's results for our case.
\begin{thm}
\label{thm:RefinedCatOEquiv}
There is an equivalence of categories
\begin{align*}
K \colon \wWhit{\fm}(U(\fg,e_1)) \to \wO{e_2,\fp_2},
\end{align*}
where $\wWhit{\fm}(U(\fg,e_1))$ is a full subcategory of $\wO{e_1,\fp_1}$.
\end{thm}

\addcontentsline{toc}{chapter}{Bibliography}
\newcommand{\etalchar}[1]{$^{#1}$}
\providecommand{\bysame}{\leavevmode\hbox to3em{\hrulefill}\thinspace}
\providecommand{\MR}{\relax\ifhmode\unskip\space\fi MR }
\providecommand{\MRhref}[2]{%
  \href{http://www.ams.org/mathscinet-getitem?mr=#1}{#2}
}
\providecommand{\href}[2]{#2}


\begin{thebibliography}{MMO{\etalchar{+}}}

\bibitem[BFK]{BFK:CatTLAlg}
Joseph Bernstein, Igor Frenkel, and Mikhail Khovanov, \emph{A categorification
  of the {Temperley--Lieb} algebra and schur quotients of {$U(\fsl_2)$} via
  projective and {Zuckerman} functors}, Sel. Math. New Ser. \textbf{5} (1999),
  199--241, \href{http://front.math.ucdavis.edu/math/0002087}{{\tt
  arXiv:math/0002087 [math.QA]}}.

\bibitem[BG]{BG:GGPoly}
Jonathan Brundan and Simon Goodwin, \emph{Good grading polytopes}, Proc. London
  Math. Soc. \textbf{94} (2007), no.~1, 155--180,
  \href{http://front.math.ucdavis.edu/0510205}{{\tt arXiv:0510205 [math.QA]}}.

\bibitem[BGG]{BGG:CatO}
Joseph Bernstein, Israïl Gel'fand, and Sergei Gel'fand, \emph{A certain
  category of $\fg$-modules}, Funkcional. Anal. i Prilo\v{z}en \textbf{10}
  (1976), no.~2, 1--8.

\bibitem[BGK]{BGK:HWTWAlg}
Jonathan Brundan, Simon Goodwin, and Alexander Kleshchev, \emph{Highest weight
  theory for finite {W}-algebras}, Int. Math. Res. Notices \textbf{11} (2008),
  Art.~ID rnn051, 53pp, \href{http://front.math.ucdavis.edu/0801.1337}{{\tt
  arXiv:0801.1337 [math.RT]}}.

\bibitem[CG]{CG:RepTheoryGeom}
Neil Chriss and Victor Ginzburg, \emph{Representation theory and complex
  geometry}, Birkhäuser, 1997.

\bibitem[CM]{CM:Nilorbits}
David~H. Collingwood and William~M. McGovern, \emph{Nilpotent orbits in
  semisimple {Lie} algebras}, Mathematics series, Van Nostrand Reinhold, 1993.

\bibitem[dBT]{dBT:QuantWAlg}
Jan de~Boer and Tjark Tjin, \emph{Quantization and representation theory of
  finite {W}-algebras}, Commun. Math. Phys. \textbf{158} (1993), 485--516,
  \href{http://front.math.ucdavis.edu/hep-th/9211109}{{\tt arXiv:hep-th/9211109
  [physics.hep-th]}}.

\bibitem[EK]{EK:ClassGG}
Alexander Elashvili and Victor Kac, \emph{Classification of good gradings of
  simple {Lie} algebras}, Lie groups and invariant theory (E.B. Vinberg, ed.),
  Amer. Math. Soc. Transl., vol. 213, AMS, 2005, pp.~85--104,
  \href{http://front.math.ucdavis.edu/math-ph/0312030}{{\tt
  arXiv:math-ph/0312030 [physics.math-ph]}}.

\bibitem[GG]{GG:QuantSlod}
Wee~Liang Gan and Victor Ginzburg, \emph{Quantization of {Slodowy} slices},
  Int. Math. Res. Notices \textbf{5} (2002), 243--255,
  \href{http://front.math.ucdavis.edu/math/0105225}{{\tt arXiv:math/0105225
  [math.RT]}}.

\bibitem[KMS1]{KMS:CatSpecht}
Mikhail Khovanov, Volodymyr Mazorchuk, and Catharina Stroppel, \emph{A
  categorification of integral {Specht} modules}, Proc. Amer. Math. Soc.
  \textbf{136} (2008), 1163--1169,
  \href{http://front.math.ucdavis.edu/math/0607630}{{\tt arXiv:math/0607630
  [math.RT]}}.

\bibitem[KMS2]{KMS:AbCat}
\bysame, \emph{A brief review of abelian categorifications}, Theory and
  Applications of Categories \textbf{22} (2009), no.~19, 479--508,
  \href{http://front.math.ucdavis.edu/math/0702746}{{\tt arXiv:math/0702746
  [math.RT]}}.

\bibitem[Kos]{Kos:WhitRep}
Bertram Kostant, \emph{On {Whittaker} vectors and representation theory},
  Invent. Math. \textbf{48} (1978), no.~2, 101--184.

\bibitem[Los1]{Los:CatOWAlg}
Ivan Losev, \emph{On the structure of category {$\cO$} for {W}-algebras},
  Sémin. Congr. \textbf{25} (2010), 351--368,
  \href{http://front.math.ucdavis.edu/0812.1584}{{\tt arXiv:0812.1584
  [math.RT]}}.

\bibitem[Los2]{Los:QuantSymplWalg}
\bysame, \emph{Quantized symplectic actions and {W}-algebras}, J. Amer. Math.
  Soc. \textbf{23} (2010), 35--59,
  \href{http://front.math.ucdavis.edu/0707.3108}{{\tt arXiv:0707.3108
  [math.RT]}}.

\bibitem[Los3]{Los:FinRepWAlg}
\bysame, \emph{Finite dimensional representations of {W}-algebras}, Duke Math.
  J. \textbf{159} (2011), no.~1, 99--143,
  \href{http://front.math.ucdavis.edu/0807.1023}{{\tt arXiv:0807.1023
  [math.RT]}}.

\bibitem[McC]{McC:UserGuideSS}
John McCleary, \emph{A user's guide to spectral sequences}, Cambridge studies
  in advanced mathematics, vol.~58, Cambridge University Press, 2000.

\bibitem[MMO{\etalchar{+}}]{MMO:HamRed}
Jerrold~E. Marsden, Gerard Misiołek, Juan-Pablo Ortega, Matthew Perlmutter,
  and Tudor~S. Ratiu, \emph{Hamiltonian reduction by stages}, Lecture notes in
  mathematics, vol. 1913, Springer, 2007.

\bibitem[MS]{MS:CompWhit}
Dragan Mili\v{c}i\'{c} and Wolfgang Soergel, \emph{The composition series of
  modules induced from {Whittaker} modules}, Comment. Math. Helv. \textbf{72}
  (1997), 503--520.

\bibitem[Pre1]{Pre:TransSlice}
Alexander Premet, \emph{Special transverse slices and their enveloping
  algebras}, Adv. in Math. \textbf{170} (2002), 1--55, with an appendix by
  Serge Skyrabin.

\bibitem[Pre2]{Pre:PrimWAlg}
\bysame, \emph{Primitive ideals, non-restricted representations and finite
  {W}-algebras}, Mosc. Math. J. \textbf{7} (2007), 743--762,
  \href{http://front.math.ucdavis.edu/math/0612465}{{\tt arXiv:math/0612465
  [math.RT]}}.

\bibitem[Sad]{Sad:PairesAdmissibles}
Guilnard Sadaka, \emph{Paires admissibles d'une algèbre de {Lie} simple
  complexe et {W\nobreakdash-algèbres} finies},
  \href{http://front.math.ucdavis.edu/1405.6390v1}{{\tt arXiv:1405.6390v1
  [math.RT]}}.

\bibitem[Vai]{Vai:LecGeoPoisson}
Izu Vaisman, \emph{Lectures on the geometry of {Poisson} manifolds}, Progress
  in Mathematics, vol. 118, Birkhäuser Verlag, 1994.

\bibitem[Wan]{Wan:NilOrbWAlg}
Weiqiang Wang, \emph{Nilpotent orbits and finite {W}-algebras}, Fields Inst.
  Commun. \textbf{59} (2011), 71--105,
  \href{http://front.math.ucdavis.edu/0912.0689}{{\tt arXiv:0912.0689
  [math.RT]}}.

\bibitem[Web]{Web:CatOWAlg}
Ben Webster, \emph{Singular blocks of parabolic category {$\cO$} and finite
  {W}-algebras}, J. Pure Appl. Algebra \textbf{215} (2011), no.~12, 2797--2804,
  \href{http://front.math.ucdavis.edu/0909.1860}{{\tt arXiv:0909.1860
  [math.RT]}}.

\end{thebibliography}
\end{document}